\documentclass{amsart}

\usepackage{amssymb}
\usepackage{amsthm}
\usepackage{amsmath}
\usepackage{enumitem}
\usepackage[mathscr]{euscript}
\usepackage{mathtools}
\usepackage{stmaryrd}
\usepackage{cancel}
\usepackage{soul}
\usepackage{tikz}
\usepackage{tikz-cd}
\usepackage[normalem]{ulem}
\usepackage[unicode]{hyperref}
\hypersetup{colorlinks, linkcolor=blue}
\usepackage{cleveref}
\usepackage{tikz}
\usetikzlibrary{fadings}
\usepackage{float}
\usetikzlibrary{patterns}
\usetikzlibrary{shadows.blur}
\usetikzlibrary{shapes}
\usepackage{xcolor}
\usepackage{todonotes}
\usepackage{comment}

\theoremstyle{plain}
\newtheorem{theoremalpha}{Theorem}

\newtheorem{proposition}{Proposition}[section]
\newtheorem{theorem}[proposition]{Theorem}
\newtheorem{lemma}[proposition]{Lemma}
\newtheorem{corollary}[proposition]{Corollary}
\theoremstyle{definition}

\newtheorem{definition}[proposition]{Definition}

\newtheorem{remark}[proposition]{Remark}

\DeclareMathOperator{\D}{D}

\DeclareMathOperator{\tr}{tr}

\DeclareMathOperator{\Id}{Id}

\DeclareMathOperator{\SL}{\mathsf{SL}}
\DeclareMathOperator{\GL}{\mathsf{GL}}

\DeclareMathOperator{\PO}{\mathsf{PO}}

\DeclareMathOperator{\PGL}{\mathsf{PGL}}

\DeclareMathOperator{\interior}{int}

\DeclareMathOperator{\Gr}{Gr} 
 
\DeclareMathOperator{\diag}{diag}

\DeclareMathOperator{\Span}{Span}
\DeclareMathOperator{\Ad}{Ad}

\DeclareMathOperator{\WH}{WH}
\DeclareMathOperator{\NW}{NW}
\DeclareMathOperator{\Stab}{Stab}
\DeclareMathOperator{\T}{T}

\DeclareMathOperator{\Bc}{\mathcal{B}}
\DeclareMathOperator{\Cc}{\mathcal{C}}
\DeclareMathOperator{\Dc}{\mathcal{D}}

\DeclareMathOperator{\Fc}{\mathcal{F}}
\DeclareMathOperator{\Gc}{\mathcal{G}}

\DeclareMathOperator{\Pc}{\mathcal{P}}

\DeclareMathOperator{\Sc}{\mathcal{S}}

\DeclareMathOperator{\Uc}{\mathcal{U}}

\DeclareMathOperator{\vb}{\mathbf{v}}
\DeclareMathOperator{\wb}{\mathbf{w}}

\DeclareMathOperator{\Hb}{\mathbb{H}}

\DeclareMathOperator{\Nb}{\mathbb{N}}
\DeclareMathOperator{\Pb}{\mathbb{P}}
\DeclareMathOperator{\Rb}{\mathbb{R}}
\DeclareMathOperator{\Sb}{\mathbb{S}}

\DeclareMathOperator{\Zb}{\mathbb{Z}}

\newcommand{\abs}[1]{\left|#1\right|}

\newcommand{\norm}[1]{\left\|#1\right\|}

\newcommand{\wh}[1]{\widehat{#1}}

\newcommand{\opp}{\textnormal{opp\,}}

\begin{document}

\title[Geometric finiteness in paracomplex hyperbolic spaces]{Geometric finiteness in paracomplex hyperbolic spaces}

\author{Tianqi Wang}
\address{Department of Mathematics, Yale University, New Haven, CT, USA}
\email{tq.wang@yale.edu}
\thanks{Wang was partially supported by the NSF grant No. DMS-2424139}

\author{Zhufeng Yao}
\address{Department of Mathematics, National University of Singapore, Singapore}
\email{yaozhufeng@u.nus.edu}

\author{Tengren Zhang}
\address{Department of Mathematics, National University of Singapore, Singapore}
\email{matzt@nus.edu.sg}
\thanks{Zhang and Yao were partially supported by NUS-MOE grant A-8004148-00-00.}

\date{\today}

\begin{abstract}
    We develop a framework for studying discrete subgroups of $\PGL(d+1,\Rb)$ via the paracomplex hyperbolic space $\Hb_\tau^d$, a rank-$1$ pseudo-Riemannian symmetric space. We characterize projective transverse, relatively Anosov, and Anosov subgroups in terms of properly discontinuous, geometrically finite, and convex-cocompact actions respectively on their weak hulls, which are canonical flow spaces in the spacelike unit tangent bundle of $\Hb_\tau^d$. A key ingredient is the construction of a Busemann-type horofunction on the spacelike unit tangent bundle with the properties needed to describe cuspidal geometry. We further prove for relatively Anosov subgroups that the geodesic flows on their weak hulls are uniformly hyperbolic, giving a relative analogue of the Axiom A property.
\end{abstract}

\maketitle

\setcounter{tocdepth}{1}
\tableofcontents

\section{Introduction}

Let $\mathsf G$ be a connected, semisimple, real Lie group with finite center and let $\mathsf K$ be a maximal compact subgroup of $\mathsf G$. When the symmetric space $\mathsf G/\mathsf K$ has rank $1$, convex-cocompact and geometrically finite subgroups of $\mathsf G$ admit well-known geometric descriptions in terms of their actions on the convex hull of their limit sets. Convex-cocompactness requires this action to be cocompact, while geometric finiteness allows in addition finitely many orbits of cusp regions. A direct extension of this picture to higher-rank symmetric spaces by imposing the same conditions is obstructed by the presence of higher-dimensional flats in $\mathsf G/\mathsf K$; the rigidity theorems of Kleiner and Leeb \cite{kleiner2004rigidity} and Quint \cite{quint2005groupesconvexes} imply that, in an irreducible higher-rank symmetric space, a Zariski-dense discrete subgroup acting cocompactly on a nonempty closed convex subset must be a uniform lattice.

The notion of an Anosov subgroup, introduced by Labourie \cite{labourie2004anosov}, and later developed in \cite{guichard2012anosov,guichard2017anosov,kapovich2014morse,kapovich2017anosov,kapovich2018relativizing,bochi2019anosov}, provides a higher-rank generalization of convex-cocompactness by the uniform contraction properties associated to some choice of parabolic subgroups of $\mathsf G$. Several relative versions of Anosovness developed in Kapovich--Leeb \cite{kapovich2018relativizing}, Zhu \cite{zhu2021relatively}, and Zhu--Zimmer \cite{zhu2022relatively}, generalize geometrically finite subgroups in rank $1$ to the higher-rank setting in a similar way. A broader class of subgroups of $\mathsf G$, called transverse subgroups (or regular, antipodal subgroups) were introduced by \cite{kapovich2017anosov} as a natural extension of rank $1$ discrete subgroups. 

Instead of working in the higher-rank Riemannian symmetric space $\mathsf G/\mathsf K$, in some cases, one can replace $\mathsf K$ by a suitable noncompact subgroup $\mathsf H$ such that $\mathsf G/\mathsf H$ is a pseudo-Riemannian symmetric space  of rank $1$. Although the $\mathsf G$-invariant symmetric two-tensor on $\mathsf G/\mathsf H$ is indefinite, the spacelike geometry of $\mathsf G/\mathsf H$ exhibits several features that mirror negatively curved Riemannian geometry. At the same time, the ideal boundary of $\mathsf G/\mathsf H$ retains the flag manifold structure, on which the convergence dynamics of certain kinds of Anosov, relatively Anosov, and transverse subgroups of $\mathsf G$ takes place. 

This point of view has appeared in several related settings. For instance, when $\mathsf G=\PO(p,q+1)$, Danciger--Guéritaud--Kassel \cite{dgkpq} developed the theory of $\mathbb H^{p,q}$-convex cocompact groups and studied their relations to projective Anosov subgroups. Related uses of pseudo-Riemannian geometry also appear in the recent work of Rungi--Tamburelli \cite{rungi2025complex} in the study of cyclic Higgs bundles for $\mathsf G =\SL(2n+1,\Rb)$.

In this paper, we focus on the special case when \[\mathsf G=\PGL(d+1,\Rb) \quad \text{and} \quad \mathsf H=\mathsf P(\Rb^*\times\GL(d,\Rb))~.\] Our results propose an alternative framework in which projective Anosov, relatively projective Anosov, and projective transverse subgroups (see \Cref{section: Anosov} for precise definitions) of $\PGL(d+1,\Rb)$ can be studied, namely via the rank $1$ pseudo-Riemannian metric space $\mathsf G/\mathsf H$. This special case is particularly important, since any Anosov (respectively, relatively Anosov, transverse) subgroup in a connected, semisimple, real Lie group can be realized as a projective Anosov (respectively, relatively projective Anosov, projective transverse) subgroup of $\PGL(d+1,\Rb)$ via a suitable linear representation, see for example \cite[Section~3]{guichard2017anosov}.

\subsection{Paracomplex hyperbolic spaces}

The corresponding space \[\Hb_\tau^d = \PGL(d+1,\Rb)/\mathsf P(\Rb^*\times\GL(d,\Rb))\] is usually called the \emph{paracomplex hyperbolic space}, as it admits a natural construction using paracomplex numbers. Recall that the algebra of paracomplex numbers is $\Rb_\tau=\Rb\oplus\Rb\tau$ where $\tau\not \in \Rb$ is such that $\tau^2 =1$. We denote by $q$ the standard parahermitian form on $\Rb_\tau^{d+1}$ given by the matrix \[ Q = \left(\begin{matrix}  &  &1 \\  & ... &  \\ 1 &  &  \end{matrix}\right) \in \mathrm{Mat}_{(d+1)\times (d+1)}(\Rb_\tau)~.\] Similar to the definition of hyperbolic spaces, the \emph{hyperboloid model} of the paracomplex hyperbolic space is \[\Hb_\tau^d \cong\{\vb\in\Rb_\tau^{d+1}:q(\vb,\vb)=-1\}/\Uc\] where \[\Uc = \{a+b \tau: a,b\in\Rb, \ a^2 -b^2 =1\}\subset \Rb_\tau~.\] The real part of $q$ induces a $\PGL(d+1,\Rb)$-invariant pseudo-Riemannian metric of signature $(d,d)$ on $\Hb_\tau^d$. In the null locus of $q$ sits a natural ideal boundary of $\Hb_\tau^d$. This boundary is a double cover of the partial flag manifold \[\Fc_{1,d} = \{([v],[\phi])\in\mathbb{RP}^d\times\mathbb{RP}^{d*}:\phi(v)=0\}~.\] We refer to this ideal boundary as the \emph{spherical boundary}, denoted by $\partial_\infty^{\mathbb S}\Hb_\tau^d$, and to $\Fc_{1,d}$ as the \emph{projective boundary}, denoted by $\partial_\infty^{\Pb}\Hb_\tau^d$.

Using the hyperboloid model, we classify all spacelike geodesics in $\Hb_\tau^d$ and describe their endpoints in both the projective and the spherical boundaries of $\Hb_\tau^d$, see Sections \ref{Section: Spacelike geodesics} and \ref{Spacelike geodesics 2}. In particular, any spacelike geodesic in $\Hb_\tau^d$ has two endpoints in $\partial_\infty^{\Pb}\Hb_\tau^d$ that are transverse, and any two transverse points in $\partial_\infty^{\Pb}\Hb_\tau^d\cong\Fc_{1,d}$ are the endpoints of exactly two spacelike geodesics in $\Hb_\tau^d$. These two geodesics correspond to the two compatible choices of lifts of transverse pairs in the projective boundary to the spherical boundary.

We denote by $\T^1\Hb_\tau^d$ the bundle of spacelike unit tangent vectors of $\Hb_\tau^d$ and by \[\varphi^t:\T^1\Hb_\tau^d\to \T^1\Hb_\tau^d\] the spacelike geodesic flow. For $z\in\T^1\Hb_\tau^d$, we write $z^+$ and $z^-$ for the forward and backward endpoints in $\partial_\infty^{\Pb}\Hb_\tau^d$ of the spacelike geodesic tangent to $z$. There is a natural involution $\opp$ on $\T^1\Hb_\tau^d$ that exchanges the tangent vectors of two spacelike geodesics having the same pair of projective endpoints, see \Cref{hyperboloid model} and \Cref{comparison of boundaries}.

\subsection{The characterization theorem}

When $\mathsf G$ is a connected, semisimple real Lie group of rank $1$, convex-cocompactness and geometric finiteness of a discrete, infinite subgroup $\Gamma\subset\mathsf G$ can be characterized by properties of the convex hull $\Cc$ in the Riemannian symmetric space $\mathsf G/\mathsf K$ of the limit set of $\Gamma$. Indeed, $\Gamma$ is convex-cocompact if and only if the $\Gamma$-action on $\Cc$ is cocompact, and $\Gamma$ is geometrically finite if and only if $\Cc$ admits thick-thin decompositions (with cocompact thick part). Briefly, this means that when $\Gamma$ is geometrically finite, there is a $\Gamma$-invariant collection of pairwise disjoint horoballs in $\mathsf G/\mathsf K$ centered at the parabolic points in the limit set of $\Gamma$, such that the $\Gamma$-action on the complement in $\Cc$ of the union of these horoballs is cocompact. The intersection of $\Cc$ with the union of said horoballs is the thin part of $\Cc$, while its complement in $\Cc$ is the thick part of $\Cc$. Via the projection $\T^1(\mathsf G/\mathsf K)\to\mathsf G/\mathsf K$, one can pullback the horoballs and thick-thin decompositions of $\Cc$ to get horoballs and thick-thin decompositions of the restriction $\T^1(\mathsf G/\mathsf K)\vert_{\Cc}$.

When $\Gamma$ is a discrete, infinite subgroup of $\PGL(d+1,\Rb)$ however, we formulate the relevant notions using subsets, called \emph{good flow spaces}, of $\T^1\Hb_\tau^d$ instead of subsets of $\Hb_\tau^d$. The main reason for doing so is our discovery of a well-behaved notion of horofunctions on $\T^1\Hb_\tau^d$. These horofunctions allow us to obtain a notion of horoballs in the good flow spaces, which behave very similarly to the horoballs in $\T^1(\mathsf G/\mathsf K)|_{\Cc}$ when $\mathsf G$ is rank $1$. Using these, one can make sense of thick-thin decompositions of good flow spaces. Unlike the Riemannian case, these objects do not have an analog in $\Hb_\tau^d$.

We will now be more precise. For any $x \in\partial_\infty^{\Pb}\Hb_\tau^d$, the \emph{horofunction centered at $x$} is a certain continuous function
\[\Bc_x:\T^1\Hb_\tau^d \to \Rb~\]
whose precise definition is given in \Cref{section: the horofunctions}. Very roughly, if we denote by $\pi:\T^1\Hb_\tau^d\to\Hb_\tau^d$ the projection, then for any $z\in\T^1\Hb_\tau^d$, $\Bc_x(z)$ is an ${\rm opp}$-invariant and negation-invariant aggregate of ``how far" $\pi(z)$ and $z^+$ are from $x$. See \Cref{definition: horofunctions} for the definition and see \Cref{remark: horofunctions for general Lie groups} for an interpretation in general Lie group setting. These horofunctions satisfy the following properties,
\begin{enumerate}
    \item For any $g\in\PGL(d+1,\Rb)$ and $x\in \partial_\infty^{\Pb}\Hb_\tau^d$, the value $\Bc_{gx}(gz)-\Bc_x(z)$ is independent of the choice of $z\in \T^1\Hb_\tau^d$ and is, in fact, an additive Busemann cocycle evaluated at $g$, see \Cref{properties of the horofunction}.
    \item For any $z\in\T^1\Hb_\tau^d$, \[\Bc_{z^+}(\varphi^tz) = \Bc_{z^+}(z)+t~ \quad \text{and} \quad \Bc_{z^-}(\varphi^t z) = \Bc_{z^-}(z)-t~.\] Furthermore, if $x\notin\{z^+,z^-\}$, then the map \[t \mapsto \Bc_x(\varphi^tz)\] is unimodal, tends to $-\infty$ in both directions and has a unique maximum, see \Cref{prop: unimodal}.
\end{enumerate}

We say that a subset $U\subset\T^1\Hb_\tau^d$ is a \emph{good flow space} for $\Gamma$ if it is closed, non-empty, $\Gamma$-invariant, flow-invariant, negation-invariant, $\opp$-invariant, convex, and the $\Gamma$-action on $U$ is properly discontinuous. Here convexity means that any two distinct points in the \emph{endpoint set of $U$},
\[\Lambda_U:=\{z^+:z\in U\}\cup\{z^-:z\in U\}\subset\partial_\infty^{\Pb}\Hb_\tau^d\cong\Fc_{1,d}~,\]
are joined by a flow line in $U$. It turns out that the $\Gamma$-action on $U$ extends to a convergence group action on $\Lambda_U$, see \Cref{proposition: convergence action}. When $\Gamma$ is projective transverse, a good flow space for $\Gamma$ always exists. Indeed, by viewing the limit set $\Lambda_\Gamma$ of $\Gamma$ as a subset of $\partial_\infty^{\Pb}\Hb_\tau^d$ via the identification $\Fc_{1,d}\cong \partial_\infty^{\Pb}\Hb_\tau^d$, we may define the \emph{weak hull} $\WH(\Lambda_\Gamma)$ of $\Lambda_\Gamma$ to be the set of all spacelike unit vectors in $\T^1 \Hb_\tau^d$ that are tangent to geodesics with both endpoints in $\Lambda_\Gamma$. One can verify, see \Cref{proposition: domain of discontinuity}, that this weak hull is a good flow space for $\Gamma$.

Let $U$ be a good flow space for $\Gamma$. For any $x\in\Lambda_U$ and $C\in\Rb$, the \emph{horoball in $U$ centered at $x$ of radius $C$} is the superlevel set 
\[H_x(C) = \{z\in U:\Bc_x(z)>C\}~.\]
Another important feature of the horofunctions  defined above is that for any distinct pair of points $x,y\in \Lambda_U$, there exists $C>0$ such that $H_x(C)$ and $H_y(C)$ are disjoint, see \Cref{lemma: disjointness of two horoballs} (this is false for superlevel sets of $\Bc_x$ and $\Bc_y$ in $\T^1\Hb_\tau^d$). We say that $U$ admits a \emph{thick--thin decomposition} if there exists a mutually disjoint collection of horoballs in $U$ that are centered in the limit set of the $\Gamma$-action on $\Lambda_U$, which consists of finitely many $\Gamma$-orbits, such that the $\Gamma$-action on the complement of these horoballs in $U$ is cocompact.

Using the above properties of horofunctions, we prove the following characterization of projective transverse, relatively projective Anosov, and projective Anosov subgroups of $\PGL(d+1,\Rb)$. We denote by $\Zb_2 \cong \Zb / 2\Zb$ the two-element group.

\begin{theoremalpha}\label{theoremalpha: characterizations}
    Let $\Gamma\subset\PGL(d+1,\Rb)$ be a discrete, infinite irreducible subgroup. 
    \begin{enumerate}
        \item $\Gamma$ is projective transverse if and only if there is a good flow space $U\subset \T^1\Hb_\tau^d$ for $\Gamma$ on which the $\Gamma\times\Rb\times \Zb_2$-action is topologically transitive, where the $\Rb$-action is by the geodesic flow and the $\Zb_2$-action is induced by the involution $\opp$ on $\T^1\Hb_\tau^d$.
        \item $\Gamma$ is relatively projective Anosov if and only if there is good flow space $U\subset \T^1\Hb_\tau^d$ for $\Gamma$ that admits a thick-thin decomposition.
        \item $\Gamma$ is projective Anosov if and only if there is a good flow space $U\subset \T^1\Hb_\tau^d$ for $\Gamma$ on which the $\Gamma$-action is cocompact.
    \end{enumerate} 
    Furthermore, in all cases, $U=\WH(\Lambda_\Gamma)$.
\end{theoremalpha}

The irreducibility assumption in \Cref{theoremalpha: characterizations} is needed mainly for the backward directions of all three parts to hold. For the forward directions to hold, one can weaken the irreducibility assumption to simply requiring that $\WH(\Lambda_\Gamma)$ is non-empty, see \Cref{section: geometric finiteness via relative Anosovness}.

\subsection{Hyperbolicity of weak hulls of limit sets.}
Delarue, Monclair, and Sanders \cite[Theorem A]{delarue2025locally} associate to every torsion-free projective Anosov subgroup of $\PGL(d+1,\Rb)$ a real analytic Axiom A flow whose non-wandering set is conjugate to Sambarino's refraction flow \cite{sambarino2024report}. 
Here the notion of an Axiom A flow was first introduced by \cite{smale1967differentiable}, and is defined to be a flow whose  non-wandering set is a compact hyperbolic set, and whose periodic orbits are dense in its non-wandering set.

Motivated by the ideas of Delarue, Monclair, and Sanders \cite{delarue2025locally}, we consider for any projective transverse subgroup $\Gamma\subset\PGL(d+1,\Rb)$, the space
\[\Dc_\Gamma := \{z \in \T^1\Hb_\tau^d : \forall \eta \in \Lambda_{\Gamma}, \text{ at least one of } z^+, z^- \text{ is transverse to } \eta\}~.\] 
Notice that $\Dc_\Gamma$ is a $\Gamma$-invariant, flow-invariant, negation-invariant, ${\opp}$-invariant, open subset of $\T^1\Hb_\tau^d$ and contains $\WH(\Lambda_\Gamma)$. We prove that the $\Gamma$-action on $\Dc_\Gamma$ is properly discontinuous, see \Cref{proposition: domain of discontinuity}. Hence, if $\Gamma$ is torsion free, then $\Dc_\Gamma/\Gamma$ is a smooth manifold, and the geodesic flow $\varphi^t$ on $\Dc_\Gamma$ descends to a smooth flow on $\Dc_\Gamma/\Gamma$, which we again denote by $\varphi^t$. We then prove the following theorem, which characterizes the hyperbolic set of $\Dc_\Gamma/\Gamma$.

\begin{theoremalpha}\label{theoremalpha: nonwandering}
    If $\Gamma\subset\PGL(d+1,\Rb)$ is projective transverse, non-elementary, and torsion free, then the non-wandering set of $(\Dc_\Gamma/\Gamma,\varphi^t)$ is $\WH(\Lambda_\Gamma)/\Gamma$. In particular, the periodic orbits in the non-wandering set are dense.
\end{theoremalpha}

The tangent bundle of $\T^1\Hb_\tau^d$ admits a natural flow-invariant splitting \[\T(\T^1\Hb_\tau^d)=\widetilde E_u\oplus \widetilde E_0\oplus \widetilde E_s~,\] 
where for every $z\in\WH(\Lambda_\Gamma)$, $\widetilde E_0|_z\subset \T_z(\T^1\Hb_\tau^d)$ is the line given by the flow direction, and  $\widetilde E_u|_z$ and $\widetilde E_s|_z$ are tangent to the submanifolds 
\[M_u(z):=\{w \in \T^1\Hb_\tau^d: w^- = z^-\text{ and } \Bc_{z^-} (w) = \Bc_{z^-} (z)\}\]
and 
\[M_s(z):=\{w \in \T^1\Hb_\tau^d: w^+ = z^+\text{ and } \Bc_{z^+} (w) = \Bc_{z^+} (z)\}\]
respectively, see \Cref{proposition: differential of endpoint map}.

Since the stabilizer of a spacelike unit vector in $\T^1\Hb_\tau^d$ is non-compact, there is no $\PGL(d+1,\Rb)$-invariant Riemannian metric on $\T(\T^1\Hb_\tau^d)$. Nevertheless, we show that when one restricts to $\WH(\Lambda_\Gamma)$, one can construct a $\Gamma$-invariant Riemannian metric in $\T(\T^1\Hb_\tau^d)|_{\WH(\Lambda_\Gamma)}$ for which $\widetilde E_s|_{\WH(\Lambda_\Gamma)}$ and $\widetilde E_u|_{\WH(\Lambda_\Gamma)}$ are respectively uniformly contracting and uniformly expanding under the geodesic flow.

\begin{theoremalpha}\label{theoremalpha: contraction}
 Suppose that $\Gamma\subset\PGL(d+1,\Rb)$ is relatively projective Anosov, non-elementary, and torsion free. There is a $\Gamma$-invariant Riemannian metric $\norm{\cdot}$ on $\T(\T^1\Hb_\tau^d)$ and constants $B,b>0$ such that for any $z\in \WH(\Lambda_\Gamma)$ and $t\geqslant 0$, \[\norm{\D\varphi^t(v)}_{\varphi^t(z)}\leqslant Be^{-bt}\norm{v}_z \quad \text{for all}\quad v\in \widetilde E_s\vert_z~,\] \[\text{and}\quad \norm{\D\varphi^{-t}(v)}_z\leqslant Be^{-bt}\norm{v}_{\varphi^t(z)} \quad \text{for all}\quad v\in \widetilde E_u\vert_{\varphi^t(z)}~.\]
\end{theoremalpha}

The Riemannian metric $\norm{\cdot}$ uses the thick-thin decomposition of $\WH(\Lambda_\Gamma)$ given by \Cref{theoremalpha: characterizations}, together with an idea of Zhu and Zimmer \cite{zhu2022relatively}, who constructed analogous metrics for bundles over abstract flow spaces of relatively hyperbolic groups.

In particular, when $\Gamma$ is projective Anosov, \Cref{theoremalpha: characterizations} implies that $\WH(\Lambda_\Gamma)/\Gamma$ is compact, so by combining Theorems \ref{theoremalpha: nonwandering} and \ref{theoremalpha: contraction}, we obtain that the geodesic flow on $\Dc_\Gamma/\Gamma$ is an Axiom~A flow. In this case, $\Dc_\Gamma$ is naturally a double cover of the flow constructed by Delarue, Monclair, and Sanders \cite{delarue2025locally}, while $\WH(\Lambda_\Gamma)$ is a double cover of the refraction flow, so one recovers the result of \cite{delarue2025locally} mentioned above. 
\subsection{Structure of the paper}

In \Cref{section: preliminaries on transverse groups}, we recall the notions of projective transverse, projective Anosov, and relatively projective Anosov subgroups, and discuss some of their basic properties. In \Cref{section: paracomplex hyperbolic geometry}, we introduce different models of $\Hb_\tau^d$ and describe its spherical and projective boundaries, before classifying the spacelike geodesics in $\Hb_\tau^d$. Then, we prove \Cref{theoremalpha: characterizations} in \Cref{section: the characterization theorem}, and Theorems \ref{theoremalpha: nonwandering} and \ref{theoremalpha: contraction} in \Cref{section: axiom A flows}.

\subsection{Acknowledgments}

This paper is based upon work partially supported by the National Science Foundation under Grant No. DMS-2424139, while the first and third authors were in residence at the Simons Laufer Mathematical Sciences Institute in Berkeley, California, during the Spring 2026 semester.

The authors also acknowledge the support and hospitality of the Institut Henri Poincaré (UAR 839 CNRS–Sorbonne Université), where part of this work was carried out during the conference Low-dimensional phenomena: geometry and dynamics in June 2025.

The first author also thanks the National University of Singapore for its hospitality during two visits where the work on this project was initiated and further developed.

The second and third authors were partially supported by NUS-MOE grant A-8004148-00-00.

\section{Transverse, relatively Anosov, and Anosov subgroups}\label{section: preliminaries on transverse groups}

In this section, we recall the notion of transverse, relatively Anosov, and Anosov subgroups of $\PGL(d+1,\Rb)$, and some of their basic properties. 

\subsection{Convergence groups}\label{sec: convergence groups}

We recall some basic terminology and facts from the theory of convergence group actions.

Let $X$ be a compact metrizable space and let $G\subset{\rm Homeo}(X)$ be a discrete infinite subgroup. We say that an escaping sequence $(\gamma_n)$ in $G$ is a \emph{collapsing sequence} if there are points $a,b\in X$  such that $\gamma_n\vert_{X\setminus \{b\}}$ converges uniformly on compact sets to (the constant map whose image is) $a$. If this happens, we refer to $a$ and $b$ as the \emph{attracting point} and \emph{repelling point} of $(\gamma_n)$. In this case $(\gamma_n^{-1})$ is also a collapsing sequence, and its attracting point and repelling point are $b$ and $a$ respectively. We say that $G$ acts on $X$ as a \emph{convergence group} if every escaping sequence in $G$ has a collapsing subsequence. The following theorem of Bowditch \cite[Proposition 1.1]{bowditch1999convergence} gives a characterization of convergence group actions.

\begin{theorem}\label{Thm: bowditch}
    The group $G$ acts as a convergence group on $X$ if and only if the diagonal $G$-action on the set $X^{(3)}$ of pairwise distinct triples in $X$ is properly discontinuous. 
\end{theorem}

Suppose that $G$ acts on $X$ as a convergence group. The \emph{limit set} $L\subset X$ of $G$ is the set of all attracting points (equivalently repelling points) of collapsing sequences in $G$. Notice that $L\subset X$ is closed and $G$-invariant, $G$ acts as a convergence group on $L$, and the $G$-action on $X\setminus L$ is properly discontinuous. We say that the $G$-action on $X$ is \emph{non-elementary} if $L$ is infinite, or equivalently, if every $G$-orbit in $X$ is infinite. When the $G$-action on $X$ is non-elementary, it is straightforward to verify that $L\subset X$ is the unique minimal subset of $X$ for the $G$-action, which is perfect.

An element $g\in G$ is \emph{loxodromic} if it has infinite order and has two distinct fixed points in $X$. In this case, $(g^n)$ is a collapsing sequence and the attracting point $g^+$ and the repelling point $g^-$ are precisely the two fixed point of $g$. We call them the \emph{attracting fixed point} and \emph{repelling fixed point} of $g$ respectively, and they necessarily lie in the limit set of $G$. The following proposition is well-known. See Gelander~\cite[Lemma~2.1]{gelander2015convergence} for a proof. 

\begin{proposition}\label{dense pairs}
If $G$ acts on $X$ as a minimal convergence group, then the set 
\[\{(g^+,g^-): g\in G \text{ is loxodromic}\}\] 
is dense in the set $X^{(2)}$ of distinct pairs of points in $X$.
\end{proposition}

Recall that for a topological space $Y$, the action of a subgroup $H \subset{\rm Homeo}(Y)$ on $Y$ is \emph{topologically transitive} if for every pair of open sets $U,V\subset Y$, there is some $h\in H$ such that $h(U)\cap V$ is non-empty. The following corollary is a consequence of \Cref{dense pairs}.

\begin{corollary}\label{prop: topological transitive}
    If $G$ acts on $X$ as a minimal convergence group, then the diagonal $G$-action on $X^{(2)}$ is topologically transitive.
\end{corollary}

\begin{proof}
Let $U,V\subset X^{(2)}$ be any two nonempty open subsets. We show that there is some $g\in G$ such that $g(U)\cap V$ is non-empty. 

By shrinking $U$ and $V$, we may assume that there exist open sets $U_1,U_2,V_1,V_2\subset X$ such that $U_1\cap U_2 = V_1\cap V_2 = \emptyset$ and $U = U_1 \times U_2$, $V = V_1 \times V_2$. By \Cref{dense pairs}, there exists $h \in G$ such that $h^+ \in V_1$ and $h^- \in V_2$. Since the $G$-action on $X$ is minimal, there exists $r\in G$ such that $r^{-1} h^-\in U_2$. We set $y := r^{-1} h^-$ and take $x\in U_1$. Then $(x,y)\in U_1\times U_2 = U$, so $x\ne y$ and hence $rx \ne h^-$. Then for $n\in \Nb$ large enough, \[h^nr(x,y) = (h^n rx, h^-)\in V_1 \times V_2 = V~.\] Hence $h^nr U \cap V \ne \emptyset$.
\end{proof}

Recall that a point $x\in X$ is \emph{conical} if there is a collapsing sequence $(\gamma_n)$ in $G$ with repelling point $x$, such that the sequence $(\gamma_n x)$ converges to a point in $X$ that is distinct from the attracting point of $(\gamma_n)$. A point $x\in X$ is \emph{parabolic} if $\Stab_G(x)$ is infinite and does not contain loxodromic elements. A parabolic point $x\in X$ is \emph{bounded} if $(X\setminus x)/\Stab_G(x)$ is compact. Notice that all conical points and parabolic points in $X$ lie in the limit set.

A convergence group action of $G$ on $X$ is \emph{uniform} if every point in $X$ is a conical point, and is \emph{geometrically finite} if every point in $X$ is either conical or bounded parabolic.

\subsection{Dynamics in projective spaces}

Next we recall some basic facts about the linear group actions on projective spaces.

Let $e_1,\dots,e_{d+1}$ be the standard basis of $\Rb^{d+1}$. Each $g\in \PGL(d+1,\Rb)$ admits a singular value decomposition \[g \;=\; k\,a(g)\,l\] such that $k,l\in \mathsf{PO}(d+1,\Rb)$ and \[a(g) \;=\; \operatorname{diag}\big(\sigma_1(g),\dots,\sigma_{d+1}(g)\big)~,\] where \[\sigma_1(g)\geqslant \cdots \geqslant \sigma_{d+1}(g) > 0\] denote the singular values of some (any) linear representative of $g$ with unit determinant. If $\sigma_p(g)> \sigma_{p+1}(g)$ for some $1\leqslant p \leqslant d$, we set \[U_p(g) \;:=\; k\,\Span(e_1,\dots,e_p)~.\] Even though $k$ and $l$ are not unique to $g$, it is straightforward to verify that $U_p(g)$ does not depend on the choice of singular value decomposition of $g$.

For any $1\leqslant p\leqslant d$, let $d_\angle$ denote the angle metric on $\Gr_p(\Rb^{d+1})$ with respect to the standard inner product on $\Rb^{d+1}$. The next lemma is a well-known result that estimates the distances between $gU_p(h)$, $U_p(gh)$, and $U_p(g)$ for $h,g\in\PGL(d+1,\Rb)$. See \cite[Lemma A.4 and A.5]{bochi2019anosov} for a proof.

\begin{lemma}\label{slmsfsvjsgfhb}
Let $p\in\{1,\dots,d\}$, and let $h,g\in\PGL(d+1,\Rb)$.
\begin{enumerate}
    \item If $\sigma_p(gh)>\sigma_{p+1}(gh)$ and $\sigma_p(h)>\sigma_{p+1}(h)$, then
\[d_\angle(gU_p(h),U_p(gh))\leqslant \sigma_1(g)\sigma_1(g^{-1})\frac{\sigma_{p+1}}{\sigma_{p}}(h)~.\]
    \item If $\sigma_p(gh)>\sigma_{p+1}(gh)$ and $\sigma_p(g)>\sigma_{p+1}(g)$, then
\[d_\angle(U_p(g),U_p(gh))\leqslant \sigma_1(h)\sigma_1(h^{-1})\frac{\sigma_{p+1}}{\sigma_{p}}(g)~.\]
\end{enumerate}
\end{lemma}

For any $x\in\mathbb{RP}^d\times\mathbb{RP}^{d*}$, let $x^1$ and $x^d$ denote its projection to $\mathbb{RP}^d$ and $\mathbb{RP}^{d*}$ respectively. The following lemma is a well-known dynamical description of the singular value decomposition. See \cite[Proposition~2.3]{canary2023patterson} for a proof.

\begin{lemma}\label{lemma: projective north-south dynamics}
    Let $a^+,a^-\in \Fc_{1,d}$ and let $(g_n)$ be an escaping sequence in $\PGL(d+1,\Rb)$. Then the following are equivalent.
    \begin{enumerate}
    \item[(1)] As $n\to \infty$, we have $\frac{\sigma_1(g_n)}{\sigma_{2}(g_n)}, \frac{\sigma_{d}(g_n)}{\sigma_{d+1}(g_n)} \to +\infty$, $\big(U_1(g_n),U_d(g_n)\big)\to a^+$, and $\big(U_1(g_n^{-1}),U_d(g_n^{-1})\big)\to a^-$. 
    \item[(2)] For any $x\in \mathbb{RP}^{d}\times \mathbb{RP}^{d*}$ such that $x^1$ is transverse to $(a^-)^{d}$ and $x^{d}$ is transverse to $(a^-)^1$, we have $g_n x \to a^+$ as $n\to +\infty$. Moreover, the convergence is uniform on any compact subset of $\mathbb{RP}^{d}\times \mathbb{RP}^{d*}$ whose elements are transverse to $a^-$.
    \item[(3)] There are open sets $O, O'\subset \mathbb{RP}^{d}\times \mathbb{RP}^{d*}$ such that $g_n x \to a^+$ as $n\to\infty$ for any $x\in O$ and $g_n^{-1}x' \to a^-$ as $n\to\infty$ for any $x'\in O'$. 
    \end{enumerate}
\end{lemma}

Notice that if any of the statements in the lemma above holds, then $a^+$ and $a^-$ are unique to the sequence $(g_n)$. We call them the \emph{attracting flag} and \emph{repelling flag} of $(g_n)$ respectively. 

The singular value decomposition also allows us to deduce the following more precise description of the attracting and repelling dynamics on the flag manifold.

\begin{lemma}\label{lemma: attracting contraction}
Suppose that $(g_n)$ is a sequence in $\PGL(d+1,\Rb)$ such that \[\frac{\sigma_1(g_n)}{\sigma_{2}(g_n)}\to +\infty~, \frac{\sigma_{d}(g_n)}{\sigma_{d+1}(g_n)} \to +\infty \quad \text{as} \quad n\to \infty~,\] \[\big(U_1(g_n),U_d(g_n)\big)\to a^+~, \big(U_1(g_n^{-1}),U_d(g_n^{-1})\big)\to a^-\quad \text{as} \quad n\to \infty~.\] Then the following hold.
\begin{enumerate}
    \item For any compact subsets $K^1\subset\mathbb{RP}^d$ whose points are all transverse to $(a^-)^d$, there are constants $A > 0$ and $N\in \Nb$ such that for all $n\geqslant N$ and $x^1,y^1\in K^1$, we have \[d_\angle(g_nx^1,g_ny^1) \leqslant A\frac{\sigma_2(g_n)}{\sigma_1(g_n)} d_\angle(x^1,y^1)~.\]
    \item For any compact subsets $K^d\subset\mathbb{RP}^{d*}$ whose points are all transverse to $(a^-)^1$, there are constants $A > 0$ and $N\in \Nb$ such that for all $n\geqslant N$ and $x^d,y^d\in K^d$, we have \[d_\angle(g_nx^d,g_ny^d) \leqslant A\frac{\sigma_{d+1}(g_n)}{\sigma_d(g_n)} d_\angle(x^d,y^d)~.\]
    \item Let $\norm{\cdot}'$ denote the norm on $\T\Fc_{1,d}$ induced by the product angular metric. For any compact subset $K\subset\Fc_{1,d}$ whose points are all transverse to $a^-$, there are constants $A>0$ and $N\in\Nb$ such that for all $n\geqslant N$, $x\in K$, and $v\in\T_x\Fc_{1,d}$, we have \[\norm{\D g_n(v)}'_{g_nx} \leqslant A\left(\frac{\sigma_2(g_n)}{\sigma_1(g_n)} + \frac{\sigma_{d+1}(g_n)}{\sigma_d(g_n)} \right) \norm{v}'_x.~\]
\end{enumerate}
\end{lemma}

\begin{proof}
    We only prove (1) as (2) follows by applying (1) to $((g_n^{-1})^T)$, and (3) follows by differentiating the inequalities in (1) and (2) and applying the Cauchy-Schwarz inequality. 
    
    For each $n\in\Nb$, let $g_n=k_n a(g_n) l_n$ be a singular value decomposition of $g_n$, where $k_n,l_n\in\PO(d+1,\Rb)$ and $a(g_n)={\rm diag}(\sigma_1(g_n),\dots,\sigma_{d+1}(g_n))$. Set 
    \[E:=\Span(e_2,\ldots,e_{d+1}).\] 
    Since $l_n^{-1}E=U_d(g_n^{-1})\to(a^-)^d$ as $n\to \infty$ and every point in $K^1$ is transverse to $(a^-)^d$, it follows that the sequence of compact sets $(l_nK^1)$ remain uniformly transverse to the projective hyperplane $\mathbb{P}(E)$. As such, there exist constants $C\geqslant 0$ and $N_0\in\Nb$ such that, for all $n\geqslant N_0$, \[l_n K^1 \subset U_C := \{[e_1+u]\in\mathbb{RP}^d: u \in E~,\ \norm{u}\leqslant C\}~,\] where $\norm{\cdot}$ denotes the standard Euclidean metric on $E$. We denote by $\chi: U_C \to E~$ the affine coordinate map defined by $\chi([e_1+u]) = u$ for any $u\in E$. Then there is some $A_0>0$ such that $\chi$ is $A_0$-biLipschitz with respect to the angle metric on $U_C$ and the standard Euclidean metric on $E$. For every $n\in\Nb$, let $D_n:E\to E$ be the linear map defined by \[D_n = \diag\left( \frac{\sigma_2(g_n)}{\sigma_1(g_n)},\dots,\frac{\sigma_{d+1}(g_n)}{\sigma_1(g_n)}\right)~.\] Then $a(g_n)[e_1+u]=[e_1+D_nu]$ and $\norm{D_n}= \frac{\sigma_2(g_n)}{\sigma_1(g_n)}$. In particular, $D_n$ preserves $\chi(U_C)$.

    For $n\geqslant N_0$ and $x^1,y^1\in K^1$, let $u,v\in E$ be such that $l_n x^1=[e_1 + u]$ and $l_n y^1=[e_1 + v]$. Then 
    \begin{align*}
        d_\angle(g_n x^1, g_n y^1) & = d_\angle\bigl(a(g_n)l_nx^1,a(g_n)l_ny^1\bigr) = d_\angle([e_1 + D_n u], [e_1 + D_n v]) \\
        & \leqslant A_0 \norm{D_n u - D_n v} \leqslant A_0 \frac{\sigma_2(g_n)}{\sigma_1(g_n)} \norm{u - v}\\
        & \leqslant A_0^2\frac{\sigma_2(g_n)}{\sigma_1(g_n)} d_\angle(l_n x^1, l_n y^1) =  A_0^2\frac{\sigma_2(g_n)}{\sigma_1(g_n)} d_\angle(x^1, y^1)~.
    \end{align*} This complete the proof of (1).
\end{proof}

\subsection{Divergent, transverse, and Anosov subgroups}\label{section: Anosov}
Recall that a subgroup $\Gamma \subset \PGL(d+1,\Rb)$ is called  \emph{projective divergent} if every escaping sequence $(\gamma_n)$ in $\Gamma$ is \emph{divergent}, that is,
 \[\log \frac{\sigma_1(\gamma_n)}{\sigma_2(\gamma_n)} \to +\infty \quad\text{as}\quad n\to \infty~.\]
 Equivalently, every escaping sequence $(\gamma_n)$ in $\Gamma$ satisfies
 \[\log \frac{\sigma_d(\gamma_n)}{\sigma_{d+1}(\gamma_n)} \to +\infty \quad\text{as}\quad n\to \infty~.\]
For any projective divergent subgroup $\Gamma\subset\PGL(d+1,\Rb)$, the \emph{limit set of $\Gamma$} is \[\Lambda_\Gamma = \{F\in \Fc_{1,d} : F \text{ is the attracting flag of an escaping sequence } (\gamma_n)\subset \Gamma\}~.\] 
The limit set of $\Gamma$ is a $\Gamma$-invariant and compact subset of $\Fc_{1,d}$. We say that $\Gamma$ is \emph{non-elementary} if $\Lambda_\Gamma$ is infinite. In this case, the $\Gamma$-action on $\Lambda_\Gamma$ is minimal.

A subgroup $\Gamma\subset \PGL(d+1,\Rb)$ is \emph{projective transverse} if it is projective divergent and any two distinct flags in $\Lambda_\Gamma$ are transverse. It follows from \Cref{lemma: projective north-south dynamics} that if $\Gamma\subset\PGL(d+1,\Rb)$ is projective transverse, then $\Gamma$ acts on $\Lambda_\Gamma$ as a convergence group, and the biproximal elements in $\Gamma$ act as loxodromic elements on $\Lambda_\Gamma$. The limit set (in the sense of convergence groups) of the $\Gamma$-action on $\Lambda_\Gamma$ is all of $\Lambda_\Gamma$. As a consequence, one can deduce that if $\Gamma$ is projective transverse, then $\Gamma$ is elementary if and only if $\Lambda_\Gamma$ has either one or two points, and $\Gamma$ is non-elementary if and only if $\Lambda_\Gamma$ is a perfect set. In particular, every irreducible projective transverse subgroup is non-elementary.

The following lemma is an application of \Cref{lemma: attracting contraction} in the setting of projective transverse groups.

\begin{lemma}\label{lemma: projective expanding}
Let $\Gamma\subset \PGL(d+1,\Rb)$ be a projective transverse subgroup.
\begin{enumerate}
    \item If $z\in \Lambda_\Gamma$ is a conical point, then for any $C>1$, there exist open neighborhoods $O_1\subset \mathbb{RP}^{d}$ of $z^1$ and $O_d\subset \mathbb{RP}^{d*}$ of $z^d$ with the following property. There exists $\gamma\in \Gamma$ such that for any $x^1,y^1\in O_1$ and any $x^d,y^d\in O_d$, we have
    \[d_{\angle}(\gamma x^1,\gamma y^1)\geqslant C d_{\angle}(x^1,y^1)\quad\text{and}\quad d_{\angle}(\gamma x^d,\gamma y^d)\geqslant C d_{\angle}(x^d,y^d)~.\]
    \item If $p \in \Lambda_\Gamma$ is a parabolic point, $(\gamma_n)$ is a divergent sequence in $\Stab_\Gamma(p)$, and $K\subset\mathbb{RP}^d\times\mathbb{RP}^{d*}$ is a compact subset whose points are all transverse to $p$, then there is some $N\in\Nb$ such that for all $x,y\in K$ and all $n\geqslant N$,
    \[d_{\angle}(\gamma_n x^1,\gamma_ny^1) <d_{\angle}(x^1,y^1)\quad\text{and}\quad d_{\angle}(\gamma_n x^d,\gamma_ny^d) <d_{\angle}(x^d,y^d)\quad~.\]
\end{enumerate} 
\end{lemma}

\begin{proof}
(1) Since $z$ is a conical point, there is a collapsing sequence $(\gamma_n)$ in $\Gamma$ with repelling point $z$ and attracting point $b\in\Lambda_\Gamma$ such that $\gamma_n z\to a$ for some $a\in\Lambda_\Gamma\setminus\{b\}$. As $\Gamma$ is transverse, $a,b\in \Lambda_\Gamma$ are transverse. We choose open neighborhoods $U_1\subset\mathbb{RP}^d$ of $a^1$ and $U_d\subset\mathbb{RP}^{d*}$ of $a^d$ such that every point in $\overline{U_1}$ is transverse to $b^d$ and every point in $\overline{U_d}$ is transverse to $b^1$.

The sequence $(\gamma_n^{-1})$ is collapsing with attracting point $z$ and repelling point $b$. By \Cref{lemma: projective north-south dynamics,lemma: attracting contraction}, for all large enough $n$, both $\gamma_n^{-1}\vert_{\overline{U_1}}$ and $\gamma_n^{-1}\vert_{\overline{U_d}}$ are $\frac{1}{C}$-Lipschitz. Since $\gamma_n z\to a$, after enlarging $n$ if necessary, we may also assume that $(\gamma_nz)^1\in U_1$ and $(\gamma_nz)^d\in U_d$.

Fix such an $n$, and set $\gamma=\gamma_n$, $O_1=\gamma^{-1}U_1$ and $O_d=\gamma^{-1}U_d$. Then $O_1$ and $O_d$ are open neighborhoods of $z^1$ and $z^d$, respectively. For any $x^1,y^1\in O_1$, we have $\gamma x^1,\gamma y^1\in U_1$, and hence \[d_\angle(x^1,y^1) = d_\angle\bigl(\gamma^{-1}\gamma x^1, \gamma^{-1}\gamma y^1\bigr) \leqslant \frac{1}{C} d_\angle(\gamma x^1,\gamma y^1)~.\] The proof for the $\mathbb{RP}^{d*}$ part is similar.

(2) Fix $0<\epsilon<1$. We show that every subsequence of $(\gamma_n)$ admits a further subsequence along which $\gamma_n\vert_K$ is $\epsilon$-Lipschitz for large enough $n$.

For an arbitrary subsequence of $(\gamma_n)$, by passing to a further subsequence, still denoted by $(\gamma_n)$, we may assume that $(\gamma_n)$ is collapsing. Let $a$ and $b$ be its repelling point and attracting point respectively. It suffices to show that $a=b$. Once we have done so, we simply apply \Cref{lemma: attracting contraction} to complete the proof.

Since $(\gamma_n)$ is a sequence in $\Stab_\Gamma(p)$, we have $p\in\{a,b\}$. If $p = a \ne b$, then by definition $p$ is conical, which  contradicts the fact that a parabolic point cannot be conical \cite[Proposition~3.2]{bowditch1999convergence}. By the same argument to $(\gamma_n^{-1})$, we also derive a contradiction from $p = b \ne a$. Therefore $a = b = p$.
\end{proof}

Among the various equivalent characterizations of Anosov representations, we refer to the formulation below as the definition (see for example Kapovich--Leeb--Porti \cite{kapovich2014morse} and Bochi--Potrie--Sambarino \cite{bochi2019anosov}).

\begin{definition}\label{theorem: Anosov via homeomorphic limit maps}
    We say that a subgroup $\Gamma\subset\PGL(d+1,\Rb)$ is \emph{projective Anosov} if and only if it is projective transverse, and it acts on $\Lambda_\Gamma$ as a uniform convergence group. We say that $\Gamma$ is \emph{relatively projective Anosov} if it is projective transverse, and it acts on $\Lambda_\Gamma$ as a geometrically finite convergence group.
\end{definition}

We say that an element $g\in \PGL(d+1,\Rb)$ is \emph{weakly unipotent} if all eigenvalues of a lift of $g$ in $\GL(d+1,\Rb)$ have the same modulus, and we say that a subgroup $P \subset \PGL(d+1,\Rb)$ is \emph{weakly unipotent} if it consists only of weakly unipotent elements. The following is another useful structural feature of relatively projective Anosov subgroups. 

\begin{proposition}[{\cite[Proposition~1.13]{zhu2022relatively}}]\label{proposition: parabolic subgroups are weakly unipotent}
     If $\Gamma\subset\PGL(d+1,\Rb)$ is relatively projective Anosov, then the stabilizers in $\Gamma$ of parabolic points in $\Lambda_\Gamma$ are all weakly unipotent subgroups.
\end{proposition}

\section{Paracomplex hyperbolic geometry}\label{section: paracomplex hyperbolic geometry}

In this section, we recall basic facts about paracomplex numbers (also known as split-complex numbers) and paracomplex hyperbolic space. The paracomplex numbers were first introduced by Cockle \cite{cockle1849}, while the paracomplex hyperbolic space already appears implicitly in Cartan’s classification of pseudo-Riemannian symmetric spaces, where it arises as a projective quadric associated to a split quadratic form. Its interpretation in terms of paracomplex or parahermitian geometry developed later, notably through the work of Libermann \cite{libermann1952structures} (see \cite[Section~8]{trettel2019families} for a modern description). In a recent work, Rungi--Tamburelli \cite{rungi2025complex} and Rungi \cite{rungi2026equiaffineimmersionspseudoriemannianspace} studied minimal and harmonic maps into such spaces.

\subsection{Paracomplex numbers}

Recall that the algebra of paracomplex numbers $\Rb_\tau$ is defined to be the algebra generated by $1$ and $\tau$ over $\Rb$ where $\tau$ is non-real and $\tau^2 = 1$. For $z = a + b \tau \in \Rb_\tau$ where $a,b \in \Rb$, the \emph{conjugate} of $z$ is defined as
\[\overline{z} = a - b\tau~.\] 
Notice that $\Rb_\tau$ is a commutative algebra, but is not an integral domain. Indeed, $\Rb_\tau$ contains exactly two idempotent elements other than $0$ and $1$, given by
\[e_+ = \dfrac{1+\tau}{2} \quad \quad \text{and} \quad \quad e_- = \dfrac{1-\tau}{2}~.\]
Observe that $e_+e_-=0$, and that conjugation is an involutive automorphism of $\Rb_\tau$ (as a commutative algebra) which permutes $e_+$ and $e_-$.

Instead of writing elements in $\Rb_\tau$ as $\Rb$-linear combinations of $1$ and $\tau$, it is often convenient to write them as $\Rb$-linear combinations of $e_+$ and $e_-$. Indeed, for any $a,b,c,d\in\Rb$, we have 
\[(a e_++be_-)(ce_++de_-)=ace_++bde_-~,\]
and in particular, $\overline{(a e_++be_-)}(a e_++be_-)=ab$. From this, we see that the group (under multiplication) of invertible elements in $\Rb_\tau$ is
\[\Rb_\tau^*=\{z\in \Rb_{\tau}:\bar{z} z\neq 0\}=\{ae_++be_-\in\Rb_\tau:ab\neq 0\}~,\]
and if $ae_++be_-$ is invertible, then $(ae_++be_-)^{-1}=a^{-1}e_++b^{-1}e_-$. The following subgroups of $\Rb_\tau^*$ will also be of interest:
\[\Rb_\tau^+ = \{z\in \Rb_{\tau}:\bar{z} z>0\}=\{ae_++be_-\in\Rb_\tau:ab>0\}\]
and
\[\mathcal U = \{z\in \Rb_{\tau}:\bar{z} z=1\}=\{ae_++be_-\in\Rb_\tau:ab=1\}\cong\Rb^*~.\]
Notice that $\Rb_\tau^+\subset\Rb_\tau^*$ and $\mathcal U\subset\Rb_\tau^+$ are normal subgroups. Furthermore, the maps
\[(\Zb_2)\times\Rb_\tau^+\to\Rb_\tau^*\quad \text{and}\quad \Rb^+\times \mathcal U\to\Rb_\tau^+\]
given by $(t,z)\mapsto \tau^tz$ and $(s,z)\mapsto sz$ respectively are group isomorphisms.

\subsection{The paracomplex hyperbolic space}\label{hyperboloid model}
Define the pairing
\[q:\Rb_{\tau}^{d+1}\times\Rb_\tau^{d+1}\to\Rb_\tau\]
by $q(\vb,\wb) = \vb^T Q \overline{\wb}$ for any $\vb, \wb \in \Rb_\tau^{d+1}$, where \[ Q = \left(\begin{matrix}  &  &1 \\  & ... &  \\ 1 &  &  \end{matrix}\right) \in \mathrm{Mat}_{(d+1)\times (d+1)}(\Rb_\tau)~.\] 
Observe that this pairing is \emph{parahermitian}, i.e., 
\begin{itemize}
    \item $q(a\vb+b\vb',\wb)=aq(\vb,\wb)+bq(\vb',\wb)$ for all $a,b\in\Rb_\tau$ and $\vb,\vb',\wb\in\Rb_\tau^{d+1}$, 
    \item $q(\wb,\vb)=\overline{q(\vb,\wb)}$ for all $\vb,\wb\in\Rb_\tau^{d+1}$. 
\end{itemize}
In particular, $q(\vb,\vb)\in\Rb$.

If we write $\vb = e_+ v_+ + e_- v_-$ and $\wb = e_+ w_+ + e_- w_-$ where $v_+,v_-,w_+,w_- \in \Rb^{d+1}$, then 
\[q(\vb,\wb) = v_+^T Q w_-e_++v_-^T Q w_+e_-~.\] 
In particular, $q(\vb,\vb) = v_+^T Q v_-$. Furthermore, if we view $\Rb_\tau^{d+1}$ as a $(2d+2)$-dimensional real vector space and $\Rb_\tau$ as a $2$-dimensional real vector space, then $q$ is an $\Rb$-bilinear map, and the real part of $q$, denoted
\[(\cdot,\cdot)_q:\Rb_{\tau}^{d+1}\times\Rb_\tau^{d+1}\to\Rb~,\]
is an $\Rb$-bilinear form given explicitly by \[(\vb,\wb)_{q}=\frac{1}{2}(q(\vb,\wb)+\overline{q(\vb,\wb)})=\frac{1}{2}(v^{T}_+Q w_-+v_-^{T}Q w_+)~.\] 
One can verify that this is symmetric, non-degenerate, with signature $(d+1,d+1)$.

Denote $\widehat{\Rb_\tau^{d+1}}=\{e_+v_++e_-v_-\in\Rb_\tau^{d+1}:v_+\neq0\neq v_-\}$, and set 
\[\Sb(\Rb_\tau^{d+1}) = \widehat{\Rb_\tau^{d+1}} / \Rb_\tau^+\quad \text{and}\quad \Pb(\Rb_\tau^{d+1}) = \widehat{\Rb_\tau^{d+1}} / \Rb_\tau^*~.\] 
Observe that both $\Sb(\Rb_\tau^{d+1})$ and $\Pb(\Rb_\tau^{d+1})$ are compact Hausdorff spaces, and the obvious quotient map $\Sb(\Rb_\tau^{d+1})\to\Pb(\Rb_\tau^{d+1})$ is a double cover. Also, observe that if $z\in\Rb_\tau^*$ (respectively, $z\in\Rb_\tau^+$, $z\in\mathcal U$) and $\vb\in\Rb_\tau^{d+1}$, then $q(z\vb,z\vb)=z\bar{z}q(\vb,\vb)$ is non-zero if and only if the same is true for (respectively, has the same sign as, is equal to) $q(\vb,\vb)$. As such, we may define the \emph{paracomplex hyperbolic $d$-space} by 
\begin{align*}
    \Hb_\tau^d &= \{[[\vb]] \in \Pb(\Rb_\tau^{d+1}) : q(\vb,\vb) \neq  0\}\\
    &\cong \{[\vb] \in \Sb(\Rb_\tau^{d+1}) : q(\vb,\vb) < 0\}\\
    &\cong\{\vb\in\Rb_\tau^{d+1}:q(\vb,\vb)=-1\}/\mathcal U.
\end{align*}
In the remainder of the paper, we use $[[\cdot]]$ and $[\cdot]$ to denote the equivalence classes in $\Pb(\Rb_\tau^{d+1})$ and $\Sb(\Rb_\tau^{d+1})$ respectively. We refer to the first description as the \emph{projective model} of $\Hb_\tau^d$, the second description as the \emph{spherical model} of $\Hb_\tau^d$, and the third description as the \emph{hyperboloid model} of $\Hb_\tau^d$. 

For each $\vb\in\Rb_\tau^{d+1}$, let $\T_{\vb}\Rb_\tau^{d+1}$ denote the tangent space to $\Rb_\tau^{d+1}$ at $\vb$. Translation by $\vb$ gives a natural identification \[\T_{\vb}\Rb_\tau^{d+1}\cong\T_{0}\Rb_\tau^{d+1}\cong\Rb_\tau^{d+1}~,\] so $q$ and $(\cdot,\cdot)_q$ respectively define an $\Rb$-bilinear pairing and an $\Rb$-bilinear form, also denoted $q$ and $(\cdot,\cdot)_q$, on $\T_{\vb}\Rb_\tau^{d+1}$. Via this identification, multiplication by $\tau$ also defines an involution \[J:\T_{\bf v}\Rb_\tau^{d+1}\to \T_{\bf v}\Rb_\tau^{d+1}\] such that $(J \cdot, J\cdot)_q = - (\cdot,\cdot)_q$.

If we further assume that $q(\vb,\vb)\neq 0$, then the projective model of $\Hb_\tau^d$ allows us to identify the tangent space $\T_{[[\vb]]}\Hb_\tau^d$ to $\Hb_\tau^d$ at $[[\vb]]$ with the $q$-orthogonal complement in $\T_{\vb}\Rb_\tau^{d+1}$ of $\Rb_\tau \vb$, i.e. 
\[\T_{[[{\bf v}]]}\Hb_\tau^d\cong\{{\bf w}\in\Rb_\tau^{d+1}:q({\bf w},{\bf v})=0\}~.\]
(Notice that the right hand side does not depend on the choice of representative ${\bf v}$ of $[[{\bf v}]]$.) One can then verify that the restriction of $(\cdot,\cdot)_q$ to $\T_{[[\vb]]}\Hb_\tau^d$ is non-degenerate and of signature $(d,d)$, and that $J$ restricts to a well-defined involution on $\T_{[[\bf v]]}\Hb_\tau^d$. Hence, $\Hb_\tau^d$ is endowed with a canonical pseudo-Riemannian metric of signature $(d,d)$, as well as an involutive bundle automorphism $J:\T\Hb_\tau^d\to \T\Hb_\tau^d$, where $\T\Hb_\tau^d$ denotes the tangent bundle to $\Hb_\tau^d$. (In fact, these give a para-K\"ahler structure on $\Hb_\tau^d$, but we will not use this fact).

It is straightforward to verify that the group of automorphisms of $\Rb_\tau^{d+1}$ (as a $\Rb_\tau$-module) that leaves the parahermitian pairing $q$ invariant is identified with $\GL(d+1,\Rb)$ via the action \[g\cdot (e_+v_+ + e_-v_-) = e_+(gv_+) + e_-(Q (g^{-1})^T Q v_-)~,\] 
and that $\GL(d+1,\Rb)$ acts transitively on $\{\vb\in\Rb_\tau^{d+1}:q(\vb,\vb)=-1\}$. Under this identification, the center of $\GL(d+1,\Rb)$ (which is $\Rb^*$) acts as multiplication by $\mathcal U$, and so by the hyperboloid model of $\Hb_\tau^d$, we see that the $\GL(d+1,\Rb)$-action on $\Hb_\tau^d$ factors through \[\PGL(d+1,\Rb)=\GL(d+1,\Rb)/\Rb^*~.\] 
In other words, $\PGL(d+1,\Rb)$ is the automorphism group of the para-K\"ahler structure on $\Hb_\tau^d$, i.e. it is the group of isometries of $\Hb_\tau^d$ whose derivative commutes with the involution $J:\T\Hb_\tau^d\to \T\Hb_\tau^d$.

We may therefore view $\Hb_\tau^d$ as a homogeneous space for $\PGL(d+1,\Rb)$. To make this explicit, let $(v_1,\dots,v_{d+1})$ be the standard basis of $\Rb^{d+1}$ and consider the basepoint $[v_1e_+-v_{d+1}e_-]\in \Hb_\tau^d$. Its stabilizer in $\GL(d+1,\Rb)$ is the subgroup $\Rb^*\times \GL(d,\Rb)$ consisting of block diagonal matrices \[\begin{pmatrix} t & 0\\ 0 & A \end{pmatrix}, \quad\text{where } A\in \GL(d,\Rb)\text{ and } t\in \Rb^*~.\] 
Therefore,
\[\Hb_\tau^d = \GL(d+1,\Rb) / (\Rb^*\times \GL(d,\Rb))= \PGL(d+1,\Rb)/\mathsf P(\Rb^*\times \GL(d,\Rb))~.\]

There are two different, but closely related, notions of boundary one can take for $\Hb_\tau^d$. The first, denoted $\partial_\infty^{\mathbb S}\Hb_\tau^d$ is the boundary of $\Hb_\tau^d$ when viewed as a subset of $\Sb(\Rb_\tau^{d+1})$, i.e.
\[\partial_\infty^{\mathbb S}\Hb_\tau^d = \{[\vb]=[e_+v_++e_-v_-]\in \Sb(\Rb_\tau^{d+1}) : (\vb,\vb)_q = 0\}~.\]
The second is the boundary of $\Hb_\tau^d$ when viewed as a subset of $\Pb(\Rb_\tau^{d+1})$, i.e.,
\[\partial_\infty^{\Pb}\Hb_\tau^d = \{[[\vb]]=[e_+v_++e_-v_-]\in \Pb(\Rb_\tau^{d+1}) : (\vb,\vb)_q = 0\}~.\]
We refer to the former as the \emph{spherical boundary} and the latter as the \emph{projective boundary}. The map 
\begin{equation}\label{double cover} 
    P:\partial_\infty^{\mathbb S}\Hb_\tau^d\to\partial_\infty^{\Pb}\Hb_\tau^d\quad\text{given by}\quad P:[{\bf v}]\mapsto[[{\bf v}]]
\end{equation}
is a double cover, and the covering involution of $P$ is the map 
\begin{equation}\label{eqn: covering involution}
    {\opp}:\partial_\infty^{\mathbb S}\Hb_\tau^d\to\partial_\infty^{\mathbb S}\Hb_\tau^d\quad\text{given by}\quad {\opp}:[e_+v+e_-\phi^*]\mapsto[e_+v-e_-\phi^*]~.
\end{equation}
Since both $\Sb(\Rb_\tau^{d+1})$ and $\Pb(\Rb_\tau^{d+1})$ are compact, it follows that $\partial_\infty^{\mathbb S}\Hb_\tau^d\cup\Hb_\tau^d$ and $\partial_\infty^{\Pb}\Hb_\tau^d\cup\Hb_\tau^d=\Pb(\Rb_\tau^{d+1})$ are both compactifications of $\Hb_\tau^d$, to which the $\PGL(d+1,\Rb)$-action on $\Hb_\tau^d$ naturally extends.

\subsection{Paracomplex hyperbolic space and real projective geometry}\label{real projective}
    Notice that the bilinear pairing $q$ induces a non-degenerate, symmetric bilinear form $\langle\cdot,\cdot\rangle$ on $\Rb^{d+1}$ defined by \[\langle v,w\rangle:=q(e_+v+e_-w,e_+v+e_-w)=w^TQv~.\] This induces a linear identification $\Rb^{d+1}\to\Rb^{d+1*}$ which sends every $v\in\Rb^{d+1}$ to $v^*:=\langle v,\cdot\rangle\in\Rb^{d+1*}$. Explicitly, if we write $v$ and $v^*$ in the standard basis on $\Rb^{d+1}$ and $(\Rb^{d+1})^*$, then $v^*=v^TQ$. For every $\phi\in\Rb^{d+1*}$, we also let $\phi^*\in\Rb^{d+1}$ denote the vector such that $\phi=(\phi^*)^*$. 

    Via this linear identification, we may now define the $\GL(d+1,\Rb)$-equivariant homeomorphism
    \[\begin{matrix}
        I: & \Rb_\tau^{d+1}& \to & \mathbb{R}^{d+1} \times \mathbb{R}^{d+1*} \\
        & e_+ v + e_- \phi^*  & \mapsto & (v, \phi)~,
    \end{matrix}\]
    where the $\GL(d+1,\Rb)$ action on $\Rb^{d+1}\times\Rb^{d+1*}$ is given by $g(v,\phi)=(g(v),\phi\circ g^{-1})$.

    This descends to a $\PGL(d+1,\Rb)$-equivariant homeomorphism \[\begin{matrix}
        I^{\Pb}: & \Pb(\Rb_\tau^{d+1})& \to & \mathbb{RP}^d \times \mathbb{RP}^{d*} \\
        & [[ e_+ v + e_- \phi^* ]] & \mapsto & ([v] , [\phi]),
    \end{matrix}\]
    which further restricts to the  $\PGL(d+1,\Rb)$-equivariant homeomorphisms \[I^{\Pb}\vert_{\Hb_\tau^d}:  \Hb_\tau^d  \to  \{ ([v],[\phi])\in \mathbb{RP}^d \times \mathbb{RP}^{d*} : \phi(v)\neq 0\}~.\] and
    \begin{align}\label{equation: extended model homeomorphism}
        F^{\Pb}:=I^{\Pb}\vert_{\partial_\infty^{\Pb}\Hb_\tau^d}:\partial_\infty^{\Pb}\Hb_\tau^d\to\Fc_{1,d}:=\{ ([v],[\phi])\in \mathbb{RP}^d \times \mathbb{RP}^{d*} : \phi(v)= 0\}.
    \end{align}
The points in the space $\Fc_{1,d}$ are called  \emph{$(1,d)$-flags}. 

Throughout this paper, we adopt the usual abuse of notation by identifying the points in $\mathbb{RP}^{d*}$ with projective hyperplanes in $\mathbb{RP}^d$. Then the above identifications allow us to view points in $\Hb_\tau^d$ and $\partial_\infty^{\Pb}\Hb_\tau^d$ using real projective geometry. Specifically, we may view every point in $\Hb_\tau^d$ (respectively, $\partial_\infty^{\Pb}\Hb_\tau^d$) as a transverse (respectively, non-transverse) pair consisting of a point in $\mathbb{RP}^d$ and a projective hyperplane in $\mathbb{RP}^d$. Since $I^{\Pb}$ is a homeomorphism, a sequence in $\Hb_\tau^d$ converges to a point in $\partial_\infty^{\Pb}\Hb_\tau^d$ if and only if we have the corresponding convergence in $\mathbb{RP}^d\times\mathbb{RP}^{d*}$. 

\subsection{Spacelike tangent vectors to paracomplex hyperbolic spaces}

For a point $[\vb] \in \Hb_\tau^d$, we denote the stabilizer of $[\vb]$ in $\PGL(d+1,\Rb)$ by $\Stab([\vb])$. We previously observed that for any $[\vb]\in\Hb_\tau^d$, the symmetric bilinear pairing $(\cdot,\cdot)_q$ on $\T_{[\vb]}\Hb_\tau^d$ has signature $(d,d)$. We say that a tangent vector $\wb\in \T_{[\vb]}\Hb_\tau^d$ is \emph{spacelike} if $(\wb,\wb)_q>0$.

\begin{lemma}\label{vectors}
    If $[\vb] \in \Hb_\tau^d$, then the $\Stab([\vb])$-action on $T_{[\vb]}\Hb_\tau^d$ is transitive on the set of spacelike unit vectors.
\end{lemma}

\begin{proof}
    In this proof, all vectors in $\Rb^{d+1}$ and $\Rb^{d+1*}$ are written in the standard bases. Since $\PGL(d+1,\Rb)$ acts transitively on $ \Hb_\tau^d$, we fix a base point \[\vb_0 = e_+ (1,0,...,0)^T + e_- (0,...,0,-1)^T \in \Rb_\tau^{d+1}~.\] Note that $q(\vb_0,\vb_0) = -1$. Then 
    \begin{align*}
        \Stab([\vb_0]) &= \Big\{\begin{bmatrix} a &  0 \\ 0 &  A  \end{bmatrix}\in\PGL(d+1,\Rb) : \ a\in \Rb^*\quad \text{and} \quad A\in \GL(d,\Rb) \Big\}\\
        &\cong \Big\{\begin{pmatrix} 1 &  0 \\ 0 &  A  \end{pmatrix}\in\GL(d+1,\Rb) : A\in \GL(d,\Rb) \Big\}~.
    \end{align*} 

    Recall that the tangent space $T_{[\vb_0]}\Hb_\tau^d$ is identified with \[ \{ \wb \in \Rb_\tau^{d+1} : q(\wb , \vb_0 ) = 0 \} = \{ e_+ ( 0 , X^T   )^T + e_- ( Y^T , 0 )^T : X,Y\in \Rb^d \}~.\] 
    With this identification, $\Stab([\vb_0])$ acts on $T_{[\vb_0]}\Hb_\tau^d$ by
    \[\begin{pmatrix} 1 &  0 \\ 0 &  A  \end{pmatrix}\cdot (e_+ ( 0 , X^T   )^T + e_- ( Y^T , 0 )^T)=e_+(0,(AX)^T)^T+ e_- ( (Q(A^{-1})^TQ(Y))^T , 0 )^T~,\]
    where $Q$ is the $d\times d$ matrix whose entries down the anti-diagonal are $1$'s and all other entries are $0$. 
    For any $Y = (y_1, y_2,..., y_d)^T\in\Rb^d$, we denote $\hat{Y} = (y_d, y_{d-1},...,y_1)^T\in\Rb^d$. Then notice that for $\wb = e_+ ( 0 , X^T   )^T + e_- ( Y^T , 0 )^T$, $(\wb,\wb)_q = X^T \hat{Y}$.

    If $\wb$ is a spacelike unit vector, i.e. $X^T \hat{Y} = 1$, let $A$ be a $d\times d$, real valued, invertible matrix with first column $X$, and such that every other column $Z$ of $A$ satisfies $Z^T\hat Y=0$. Then we have \[ \left(\begin{matrix} 1 &  0 \\ 0 &  A  \end{matrix}\right)\cdot (e_+ (0,1,0,...,0)^T + e_- (0,...,0,1,0)^T) = e_+ (0,X^T)^T + e_- (Y^T,0)^T~.\] Thus, $\wb$ is in the $\Stab([\vb_0])$-orbit of $e_+ (0,1,0,...,0)^T + e_- (0,...,0,1,0)^T$.
\end{proof}

\subsection{Spacelike geodesics}\label{Section: Spacelike geodesics}
A geodesic (under the Levi-Civita connection with respect to the $(d,d)$-pseudo-Riemannian structure) in $\Hb_\tau^d$ is \emph{spacelike} if all of its tangent vectors are spacelike. Since $\PGL(d+1,\Rb)$ acts transitively on $\Hb_\tau^d$, we see that up to the $\PGL(d+1,\Rb)$-action and reparameterization, there is a unique unit speed spacelike geodesic. We will now describe one such geodesic.

Let $v_1,v_2\in\Rb^{d+1}$ and $\phi_1,\phi_2\in\Rb^{d+1*}$ such that $\phi_1(v_1)=0=\phi_2(v_2)$ and $\phi_1(v_2)=1=\phi_2(v_1)$. Then let
\[\hat\ell_{v_1,\phi_1,v_2,\phi_2}:\Rb\to\Rb_\tau^{d+1}\]
be the map given by 
\[\hat\ell_{v_1,\phi_1,v_2,\phi_2}(t)=e_+\left(\frac{e^t}{\sqrt{2}}v_1+\frac{e^{-t}}{\sqrt{2}}v_2\right)+e_-\left(-\frac{e^t}{\sqrt{2}}\phi_1^*-\frac{e^{-t}}{\sqrt{2}}\phi_2^*\right)~.\]
Notice that for all $t\in\Rb$, $\hat\ell_{v_1,\phi_1,v_2,\phi_2}(t)$ lies in $\widehat{\Rb_\tau^{d+1}}$, so it descends to a map
\[\ell_{v_1,\phi_1,v_2,\phi_2}:\Rb\to\Hb_\tau^d~.\]

\begin{proposition}\label{prop: spacelike geodesics}
    The map $\ell_{v_1,\phi_1,v_2,\phi_2}$ is a unit speed spacelike geodesic in $\Hb_\tau^d$ with $[e_+v_1-e_-\phi_1^*]$ and $[e_+v_2-e_-\phi_2^*]$ in $\partial_\infty^{\mathbb S}\Hb_\tau^d$ as its forward and backward endpoints. Up to changing the origin, this is the unique unit speed spacelike geodesic in $\Hb_\tau^d$ with those endpoints in $\partial_\infty^{\mathbb S}\Hb_\tau^d$. Furthermore, the subgroup in $\PGL(d+1,\Rb)$ that fixes the two endpoints in $\partial_\infty^{\mathbb S}\Hb_\tau^d$ of $\ell$ acts transitively on $\ell_{v_1,\phi_1,v_2,\phi_2}(\Rb)$.
\end{proposition}

\begin{proof}
    First, we verify that $\ell=\ell_{v_1,\phi_1,v_2,\phi_2}$ is a geodesic in $\Hb_\tau^d$. The tangent vector to $\hat{\ell}=\hat{\ell}_{v_1,\phi_1,v_2,\phi_2}$ at $\hat{\ell}(t)$ is given by \[\hat{\ell}'(t) = e_+\left(\frac{e^t}{\sqrt{2}}v_1-\frac{e^{-t}}{\sqrt{2}}v_2\right)+ e_-\left(-\frac{e^t}{\sqrt 2}\phi_1^*+\frac{e^{-t}}{\sqrt 2}\phi_2^*\right)~.\] Then by direct computation, one verifies that \[q(\hat\ell(t),\hat\ell(t))=-1,\quad q(\hat{\ell}'(t),\hat{\ell}(t))=0,\quad\text{and}\quad q(\hat{\ell}'(t),\hat{\ell}'(t))=1\] for all $t\in\Rb$. Via the identification \[\T_{\ell(t)}\Hb_\tau^d \cong \{\wb\in \Rb_\tau^{d+1} : q(\wb,\hat{\ell}(t))=0\}~,\] the first two equalities imply that $\ell'(t)=\hat{\ell}'(t)$ for all $t\in\Rb$, and the third implies that $\ell$ is a unit speed spacelike curve.

    Differentiating once more, we find $\hat{\ell}''(t)=\hat{\ell}(t)$, which spans the $q$-normal direction to $T_{\ell(t)}\Hb_\tau^d$. This implies that the covariant derivative of $\hat{\ell}'$ with respect to the Levi--Civita connection of the pseudo-Riemannian metric on $\Hb_\tau^d$ is $0$. Therefore, $\ell(t)$ is a spacelike geodesic in $\Hb_\tau^d$. The description of the endpoints of $\hat\ell$ is immediate.

    Next, we prove the uniqueness claim of the proposition. By \Cref{vectors}, $\PGL(d+1,\Rb)$ acts transitively on the set of spacelike unit vectors, and hence on the set of spacelike unit speed geodesics. It thus suffices to verify that every element in the stabilizer in $\PGL(d+1,\Rb)$ of $[e_+v_1-e_-\phi_1^*]$ and $[e_+v_2-e_-\phi_2^*]$ necessarily sends $\ell$ to itself, possibly with the origin moved.

    Choose a basis $(f_0,\dots,f_d)$ of $\Rb^{d+1}$ such that $[f_0]=[v_1]$, $[f_d]=[v_2]$, and $\Span(f_1,\dots,f_{d-1})=[\phi_1]\cap[\phi_2]$. Then every $g\in\PGL(d+1,\Rb)$ that fixes both $[e_+v_1-e_-\phi_1^*]$ and $[e_+v_2-e_-\phi_2^*]$ has a linear representative $\hat g\in\GL(d+1,\Rb)$ that, when written in this basis, is of the form \[\hat g:=\begin{pmatrix}
        a&0&0\\
        0&M&0\\
        0&0&b \end{pmatrix}\] 
    for some $a,b\in\Rb$ such that $ab>0$ and some $M\in\GL(d-1,\Rb)$. Since $\hat g(v_1)=av_1$, $\hat g(v_2)=bv_2$, $\hat g(\phi_1)=\frac{1}{b}\phi_1$, and $\hat g(\phi_2)=\frac{1}{a}\phi_2$, it follows that
    \begin{align}\label{glt formula}
        \begin{split}
            g(\ell(t))&=\left[e_+\left(\frac{e^t}{\sqrt{2}}\sqrt{\frac{a}{b}}v_1+\frac{e^{-t}}{\sqrt{2}}\sqrt{\frac{b}{a}}v_2\right)+ e_-\left(-\frac{e^t}{\sqrt 2}\sqrt{\frac{a}{b}}\phi_1^*-\frac{e^{-t}}{\sqrt 2}\sqrt{\frac{b}{a}}\phi_2^*\right)\right]\\
            &=\ell\left(t+\frac{1}{2}\log\frac{a}{b}\right).
        \end{split}
    \end{align}

    Since $a$ and $b$ can be any pair of non-zero real numbers with the same sign as we vary $g$ in the subgroup of $\PGL(d+1,\Rb)$ that fixes $[e_+v_1-e_-\phi_1^*]$ and $[e_+v_2-e_-\phi_2^*]$, it follows from the above computation that the subgroup of $\PGL(d+1,\Rb)$ that fixes the endpoints in $\partial_\infty^{\mathbb{S}}\Hb_\tau^d$ of a spacelike geodesic acts transitively on its image.
\end{proof}

\begin{remark}\label{comparison of boundaries}
Observe that in the same basis $(f_0,\dots,f_d)$ that we used in the proof of \Cref{prop: spacelike geodesics}, the stabilizer in $\PGL(d+1,\Rb)$ of $[[e_+v_1-e_-\phi_1^*]]$ and $[[e_+v_2-e_-\phi_2^*]]$ in $\partial_\infty^{\Pb}\Hb_\tau^d\cong\Fc_{1,d}$ is
\[\left\{\begin{bmatrix}
a&0&0\\
0&M&0\\
0&0&b
\end{bmatrix}:ab\neq 0\text{ and }M\in\GL(d-1,\Rb)\right\}~.\]
The last computation used in the proof of \Cref{prop: spacelike geodesics} implies that some elements in this stabilizer (namely the elements where $ab<0$) send $\ell=\ell_{v_1,\phi_1,v_2,\phi_2}$ to (up to moving the origin) the geodesic $\ell^{\opp}:=\ell_{v_1,-\phi_1,-v_2,\phi_2}$.

As such, up to moving the origin, the geodesics $\ell$ and $\ell^{\opp}$ are the only ones that have $[[e_+v_1-e_-\phi_1^*]]$ and $[[e_+v_2-e_-\phi_2^*]]$ as their forward and backward endpoints in $\partial_\infty^{\mathbb{P}}\Hb_\tau^d$. However, neither the forward nor backward endpoints of $\ell$ and $\ell^{\opp}$ in $\partial_\infty^{\mathbb S}\Hb_\tau^d$ coincide. In particular, if $[[e_+v_1-e_-\phi_1^*]]$ and $[[e_+v_2-e_-\phi_2^*]]$ are transverse flags in $\partial_\infty^{\Pb}\Hb_\tau^d\cong\Fc_{1,d}$, then $[e_+v_1-e_-\phi_1^*]\in\partial_\infty^{\mathbb S}\Hb_\tau^d$ is joined by a spacelike geodesic in $\Hb_\tau^d$ to exactly one of $[e_+v_2-e_-\phi_2^*]$ or $[e_+v_2+e_-\phi_2^*]={\opp}([e_+v_2-e_-\phi_2^*])$.
\end{remark}

For any $z\in \T^1\Hb_\tau^d$, let $z^+$ and $z^-$ respectively denote the forward and backward endpoints in $\partial_\infty^{\Pb}\Hb_\tau^d$ of the spacelike geodesic in $\Hb_\tau^d$ tangent to $z$. Recall that we have a $\PGL(d+1,\Rb)$-equivariant identification 
\[\Hb_\tau^d\cup \partial_\infty^{\Pb}\Hb_\tau^d\cong\mathbb{RP}^d\times\mathbb{RP}^{d*}~.\]
The following is an immediate corollary of \Cref{prop: spacelike geodesics}.

\begin{corollary}\label{observation: a spacelike geodesic is totally transverse}
Let $z\in \T^1\Hb_\tau^d$, and let $\ell:\Rb\to\Hb_\tau^d$ be a spacelike geodesic tangent to $z$.
\begin{enumerate}
    \item Any two distinct points in $\ell(\Rb)\cup\{z^+,z^-\}$ are transverse.
    \item If $z^+=([v_1],[\phi_1])$, $z^-=([v_2],[\phi_2])$, $t\in\Rb$, and $\ell(t)=([u],[\psi])$, then $[\psi]$ does not intersect the connected component of $\mathbb{RP}^d\setminus([\phi_1]\cup[\phi_2])$ that contains $[u]$.
\end{enumerate}
\end{corollary}

Let $(\cdot,\cdot;\cdot,\cdot)$ denote the cross ratio in $\mathbb{RP}^1$, with the convention that $(0,\infty;1,t)=t$. The following is a consequence of a standard and straightforward computation using \Cref{prop: spacelike geodesics}.

\begin{corollary}
    Let $\ell:\Rb\to\Hb_\tau^d$ be a unit speed spacelike geodesic, and let $t_1,t_2\in\Rb$. For both $i=1,2$, let $u_i\in\Rb^{d+1}$ and $\psi_i\in\Rb^{d+1*}$ such that $\ell(t_i)=[e_+u_i+e_-\psi_i^*]$. Then \[\cosh^2(t_2-t_1) = \frac{\psi_1(u_2)\psi_2(u_1)}{\psi_1(u_1)\psi_2(u_2)}~.\]
\end{corollary}

As another consequence of \Cref{prop: spacelike geodesics}, we deduce the following description of the images of spacelike geodesics in $\Hb_\tau^d$ using real projective geometry. 

\begin{corollary}\label{projective geometry description}
    Let $([v_1],[\phi_1])$ and $([v_2],[\phi_2])$ be a transverse pair of $(1,d)$-flags. For all $[{\bf u}]=[e_+u+e_-\psi^*]\in\Hb_\tau^d$, let $[w]:=[\psi]\cap ([v_1]+[v_2])$. Then $[{\bf u}]$ lies along a spacelike geodesic with $([v_1],[\phi_1])$ and $([v_2],[\phi_2])$ as its endpoints in $\partial_\infty^{\Pb}\Hb_\tau^d$ if and only if $[v_1],[v_2],[u]\in\mathbb{RP}^d$ lie in a projective line, $[\phi_1],[\phi_2],[\psi]\in\mathbb{RP}^{d*}$ lie in a projective line, and $([v_1],[v_2];[w],[u])=-1$, see \Cref{figure: spacelike geodesics}.
\end{corollary}

\begin{proof}
    First, suppose that $([v_1],[v_2];[w],[u])=-1$. This implies that we may choose representatives $v_1,v_2\in\Rb^{d+1}$ of $[v_1],[v_2]\in\mathbb{RP}^d$ such that $u=v_1+v_2$ and $w=v_1-v_2$. Then choose representatives $\phi_1,\phi_2\in\Rb^{d+1*}$ of $[\phi_1],[\phi_2]\in\mathbb{RP}^{d*}$ such that $\phi_1(v_2)=1=\phi_2(v_1)$. It is straightforward to verify that $[{\bf u}]=\ell_{v_1,\phi_1,v_2,\phi_2}(0)$.

    Conversely, suppose that $[{\bf u}]$ lies along a spacelike geodesic with $([v_1],[\phi_1])$ and $([v_2],[\phi_2])$ as its endpoints in $\partial_\infty^{\Pb}\Hb_\tau^d$. \Cref{prop: spacelike geodesics} implies that by translating by an element in $\PGL(d+1,\Rb)$, we may assume that $\ell_{v_1,\phi_1,v_2,\phi_2}(0)=[{\bf u}]$ for some representatives $v_1,v_2\in\Rb^{d+1}$ of $[v_1],[v_2]\in\mathbb{RP}^d$ and $\phi_1,\phi_2\in\Rb^{d+1*}$ of $[\phi_1],[\phi_2]\in\mathbb{RP}^{d*}$ such that $\phi_1(v_2)=1=\phi_2(v_1)$. It is straightforward to verify that $v_1+v_2$ and $v_1-v_2$ are representatives $[u]$ and $[w]$ respectively, so $([v_1],[v_2];[w],[u])=-1$.
\end{proof}

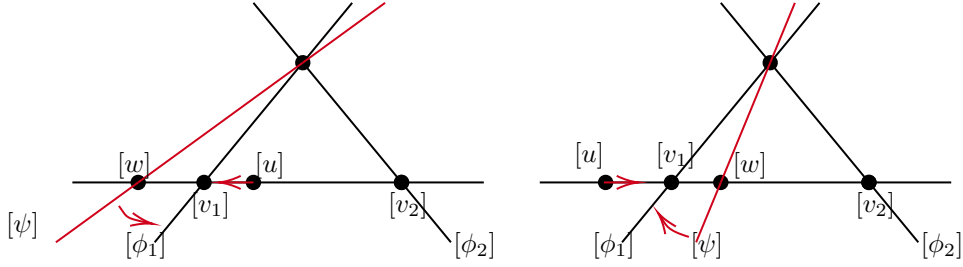
\begin{figure}[ht]
    \tikzset{every picture/.style={line width=0.75pt}} %set default line width to 0.75pt        

\begin{tikzpicture}[x=0.75pt,y=0.75pt,yscale=-1,xscale=1]
%uncomment if require: \path (0,152); %set diagram left start at 0, and has height of 152

%Straight Lines [id:da7187571866398156] 
\draw    (151.17,38) -- (200.67,98) ;
%Straight Lines [id:da04230696890606889] 
\draw    (200.67,98) -- (225.41,128) ;
%Straight Lines [id:da06831630076220763] 
\draw    (101.67,98) -- (76.92,128) ;
%Straight Lines [id:da7752270099883998] 
\draw    (151.17,38) -- (126.42,8) ;
%Straight Lines [id:da32031874217599277] 
\draw    (151.17,38) -- (175.92,8) ;
%Straight Lines [id:da17351014058306846] 
\draw    (101.67,98) -- (126.42,98) ;
\draw [shift={(126.42,98)}, rotate = 0] [color={rgb, 255:red, 0; green, 0; blue, 0 }  ][fill={rgb, 255:red, 0; green, 0; blue, 0 }  ][line width=0.75]      (0, 0) circle [x radius= 3.35, y radius= 3.35]   ;
%Straight Lines [id:da4351068967343491] 
\draw    (101.67,98) -- (151.17,38) ;
\draw [shift={(151.17,38)}, rotate = 309.52] [color={rgb, 255:red, 0; green, 0; blue, 0 }  ][fill={rgb, 255:red, 0; green, 0; blue, 0 }  ][line width=0.75]      (0, 0) circle [x radius= 3.35, y radius= 3.35]   ;
%Straight Lines [id:da6836927246835016] 
\draw    (126.42,98) -- (200.67,98) ;
\draw [shift={(200.67,98)}, rotate = 0] [color={rgb, 255:red, 0; green, 0; blue, 0 }  ][fill={rgb, 255:red, 0; green, 0; blue, 0 }  ][line width=0.75]      (0, 0) circle [x radius= 3.35, y radius= 3.35]   ;
%Straight Lines [id:da9626441546824632] 
\draw    (68.67,98) -- (101.67,98) ;
\draw [shift={(101.67,98)}, rotate = 0] [color={rgb, 255:red, 0; green, 0; blue, 0 }  ][fill={rgb, 255:red, 0; green, 0; blue, 0 }  ][line width=0.75]      (0, 0) circle [x radius= 3.35, y radius= 3.35]   ;
%Straight Lines [id:da6049692555750698] 
\draw    (35.67,98) -- (68.67,98) ;
\draw [shift={(68.67,98)}, rotate = 0] [color={rgb, 255:red, 0; green, 0; blue, 0 }  ][fill={rgb, 255:red, 0; green, 0; blue, 0 }  ][line width=0.75]      (0, 0) circle [x radius= 3.35, y radius= 3.35]   ;
%Straight Lines [id:da23984838304719125] 
\draw    (241.91,98) -- (200.67,98) ;
\draw [shift={(200.67,98)}, rotate = 180] [color={rgb, 255:red, 0; green, 0; blue, 0 }  ][fill={rgb, 255:red, 0; green, 0; blue, 0 }  ][line width=0.75]      (0, 0) circle [x radius= 3.35, y radius= 3.35]   ;
%Straight Lines [id:da12404669478202279] 
\draw [color={rgb, 255:red, 208; green, 2; blue, 27 }  ,draw opacity=1 ]   (151.17,38) -- (68.67,98) ;
%Straight Lines [id:da07487127649646597] 
\draw [color={rgb, 255:red, 208; green, 2; blue, 27 }  ,draw opacity=1 ]   (151.17,38) -- (192.42,8) ;
%Straight Lines [id:da7044572251159658] 
\draw [color={rgb, 255:red, 208; green, 2; blue, 27 }  ,draw opacity=1 ]   (68.67,98) -- (27.42,128) ;
%Straight Lines [id:da9225036544055151] 
\draw    (385.46,38) -- (434.95,98) ;
%Straight Lines [id:da592281561462114] 
\draw    (434.95,98) -- (459.7,128) ;
%Straight Lines [id:da20044718486757773] 
\draw    (335.96,98) -- (311.21,128) ;
%Straight Lines [id:da34693038762157646] 
\draw    (385.46,38) -- (360.71,8) ;
%Straight Lines [id:da8896562302420641] 
\draw    (385.46,38) -- (410.2,8) ;
%Straight Lines [id:da2417756679409493] 
\draw    (335.96,98) -- (360.71,98) ;
\draw [shift={(360.71,98)}, rotate = 0] [color={rgb, 255:red, 0; green, 0; blue, 0 }  ][fill={rgb, 255:red, 0; green, 0; blue, 0 }  ][line width=0.75]      (0, 0) circle [x radius= 3.35, y radius= 3.35]   ;
%Straight Lines [id:da4307276571166758] 
\draw    (335.96,98) -- (385.46,38) ;
\draw [shift={(385.46,38)}, rotate = 309.52] [color={rgb, 255:red, 0; green, 0; blue, 0 }  ][fill={rgb, 255:red, 0; green, 0; blue, 0 }  ][line width=0.75]      (0, 0) circle [x radius= 3.35, y radius= 3.35]   ;
%Straight Lines [id:da18308103866139747] 
\draw    (360.71,98) -- (434.95,98) ;
\draw [shift={(434.95,98)}, rotate = 0] [color={rgb, 255:red, 0; green, 0; blue, 0 }  ][fill={rgb, 255:red, 0; green, 0; blue, 0 }  ][line width=0.75]      (0, 0) circle [x radius= 3.35, y radius= 3.35]   ;
%Straight Lines [id:da9436615248924517] 
\draw    (302.96,98) -- (335.96,98) ;
\draw [shift={(335.96,98)}, rotate = 0] [color={rgb, 255:red, 0; green, 0; blue, 0 }  ][fill={rgb, 255:red, 0; green, 0; blue, 0 }  ][line width=0.75]      (0, 0) circle [x radius= 3.35, y radius= 3.35]   ;
%Straight Lines [id:da3935003141551363] 
\draw    (269.96,98) -- (302.96,98) ;
\draw [shift={(302.96,98)}, rotate = 0] [color={rgb, 255:red, 0; green, 0; blue, 0 }  ][fill={rgb, 255:red, 0; green, 0; blue, 0 }  ][line width=0.75]      (0, 0) circle [x radius= 3.35, y radius= 3.35]   ;
%Straight Lines [id:da4960970996455806] 
\draw    (476.2,98) -- (434.95,98) ;
\draw [shift={(434.95,98)}, rotate = 180] [color={rgb, 255:red, 0; green, 0; blue, 0 }  ][fill={rgb, 255:red, 0; green, 0; blue, 0 }  ][line width=0.75]      (0, 0) circle [x radius= 3.35, y radius= 3.35]   ;
%Straight Lines [id:da9421834890105522] 
\draw [color={rgb, 255:red, 208; green, 2; blue, 27 }  ,draw opacity=1 ]   (385.46,38) -- (360.71,98) ;
%Straight Lines [id:da573685923628149] 
\draw [color={rgb, 255:red, 208; green, 2; blue, 27 }  ,draw opacity=1 ]   (385.46,38) -- (397.83,8) ;
%Straight Lines [id:da8354077911230225] 
\draw [color={rgb, 255:red, 208; green, 2; blue, 27 }  ,draw opacity=1 ]   (360.71,98) -- (348.33,128) ;
%Straight Lines [id:da5039643801868267] 
\draw [color={rgb, 255:red, 208; green, 2; blue, 27 }  ,draw opacity=1 ]   (126.42,98) -- (111.12,98) ;
\draw [shift={(109.12,98)}, rotate = 360] [color={rgb, 255:red, 208; green, 2; blue, 27 }  ,draw opacity=1 ][line width=0.75]    (10.93,-3.29) .. controls (6.95,-1.4) and (3.31,-0.3) .. (0,0) .. controls (3.31,0.3) and (6.95,1.4) .. (10.93,3.29)   ;
%Curve Lines [id:da37535084359652815] 
\draw [color={rgb, 255:red, 208; green, 2; blue, 27 }  ,draw opacity=1 ]   (59,109.75) .. controls (63.3,115.72) and (61.67,115.76) .. (74.14,118.36) ;
\draw [shift={(76,118.75)}, rotate = 191.69] [color={rgb, 255:red, 208; green, 2; blue, 27 }  ,draw opacity=1 ][line width=0.75]    (10.93,-3.29) .. controls (6.95,-1.4) and (3.31,-0.3) .. (0,0) .. controls (3.31,0.3) and (6.95,1.4) .. (10.93,3.29)   ;
%Straight Lines [id:da16509243762474568] 
\draw [color={rgb, 255:red, 208; green, 2; blue, 27 }  ,draw opacity=1 ]   (317.46,98) -- (302.16,98) ;
\draw [shift={(319.46,98)}, rotate = 180] [color={rgb, 255:red, 208; green, 2; blue, 27 }  ,draw opacity=1 ][line width=0.75]    (10.93,-3.29) .. controls (6.95,-1.4) and (3.31,-0.3) .. (0,0) .. controls (3.31,0.3) and (6.95,1.4) .. (10.93,3.29)   ;
%Curve Lines [id:da8683041699627875] 
\draw [color={rgb, 255:red, 208; green, 2; blue, 27 }  ,draw opacity=1 ]   (330.29,114.78) .. controls (333.53,119.79) and (333.6,121.92) .. (344.71,125.29) ;
\draw [shift={(329.14,113.11)}, rotate = 54.25] [color={rgb, 255:red, 208; green, 2; blue, 27 }  ,draw opacity=1 ][line width=0.75]    (10.93,-3.29) .. controls (6.95,-1.4) and (3.31,-0.3) .. (0,0) .. controls (3.31,0.3) and (6.95,1.4) .. (10.93,3.29)   ;

% Text Node
\draw (60.23,121.4) node [anchor=north west][inner sep=0.75pt]    {$[\phi _{1}]$};
% Text Node
\draw (225.23,121.4) node [anchor=north west][inner sep=0.75pt]    {$[\phi _{2}]$};
% Text Node
\draw (93.5,101.4) node [anchor=north west][inner sep=0.75pt]    {$[v_{1}]$};
% Text Node
\draw (192.49,101.4) node [anchor=north west][inner sep=0.75pt]    {$[v_{2}]$};
% Text Node
\draw (125.71,81.4) node [anchor=north west][inner sep=0.75pt]    {$[u]$};
% Text Node
\draw (56.64,81.4) node [anchor=north west][inner sep=0.75pt]    {$[w]$};
% Text Node
\draw (294.81,121.4) node [anchor=north west][inner sep=0.75pt]    {$[\phi _{1}]$};
% Text Node
\draw (459.51,121.4) node [anchor=north west][inner sep=0.75pt]    {$[\phi _{2}]$};
% Text Node
\draw (326.64,76.4) node [anchor=north west][inner sep=0.75pt]    {$[v_{1}]$};
% Text Node
\draw (426.78,101.4) node [anchor=north west][inner sep=0.75pt]    {$[v_{2}]$};
% Text Node
\draw (365.17,81.4) node [anchor=north west][inner sep=0.75pt]    {$[w]$};
% Text Node
\draw (285.75,76.4) node [anchor=north west][inner sep=0.75pt]    {$[u]$};
% Text Node
\draw (1,111.4) node [anchor=north west][inner sep=0.75pt]    {$[\psi]$};
% Text Node
\draw (343.94,121.4) node [anchor=north west][inner sep=0.75pt]    {$[\psi]$};

\end{tikzpicture}
    \caption{Real projective geometry description of spacelike geodesics.}
    \label{figure: spacelike geodesics}
\end{figure}

Let $\gamma \in \PGL(d+1,\Rb)$ be a biproximal element, let $\gamma^+,\gamma^- \in \Fc_{1,d}$ denote its attracting and repelling flags respectively, and let $\lambda_1(\gamma)$ and $\lambda_{d+1}(\gamma)$ respectively denote the moduli of the largest and smallest eigenvalues of some (any) linear representative in $\SL^{\pm 1}(d+1,\Rb)$ of $\gamma$. Since $\gamma^+$ and $\gamma^-$ are transverse, \Cref{comparison of boundaries} implies that there are (up to reparameterization) two geodesics in $\Hb_\tau^d$ that have $\gamma^+$ and $\gamma^-$ as endpoints, both of which are preserved by $\gamma$. Then $\gamma$ either preserves each geodesic between $\gamma^+$ and $\gamma^-$ or permutes them. In the former case, by \Cref{prop: spacelike geodesics}, we may calculate the translation distance of $\gamma$ along these geodesics in terms of $\lambda_1(\gamma)$ and $\lambda_{d+1}(\gamma)$.

\begin{corollary}\label{observation: period length}
    Suppose a biproximal element $\gamma \in \PGL(d+1,\Rb)$ preserves a spacelike geodesic $\ell$ in $\Hb_\tau^d$ joining $\gamma^-,\gamma^+\in\partial_\infty^{\mathbb{P}}\Hb_\tau^d$. Then the translation length of $\gamma$ 
    along $\ell$ is $\frac{1}{2} \log \frac{\lambda_1(\gamma)}{\lambda_{d+1}(\gamma)}$.
\end{corollary}

\begin{proof}
By \Cref{prop: spacelike geodesics}, we may choose $v_1,v_2\in\Rb^{d+1}$ and $\phi_1,\phi_2\in\Rb^{d+1*}$ such that 
$\gamma^+=([v_1],[\phi_1])$, $\gamma^-=([v_2],[\phi_2])$, $\phi_1(v_1)=0=\phi_2(v_2)$, $\phi_1(v_2)=1=\phi_2(v_1)$, and $\ell=\ell_{v_1,\phi_1,v_2,\phi_2}$. If we set $s(\gamma):=\frac{1}{2}\log\frac{\lambda_1(\gamma)}{\lambda_{d+1}(\gamma)}$, then 
\begin{align*}
    \gamma(\ell(t))&=\left[e_+\left(\frac{\lambda_1(\gamma)e^t}{\sqrt{2}}v_1+\frac{\lambda_{d+1}(\gamma)e^{-t}}{\sqrt{2}}v_2\right)+e_-\left(-\frac{\lambda_{d+1}(\gamma)^{-1}e^t}{\sqrt{2}}\phi_1^*-\frac{\lambda_1(\gamma)^{-1}e^{-t}}{\sqrt{2}}\phi_2^*\right)\right]\\
    &=\left[e_+\left(\frac{e^{t+s(\gamma)}}{\sqrt{2}}v_1+\frac{e^{-t-s(\gamma)}}{\sqrt{2}}v_2\right)+e_-\left(-\frac{e^{t+s(\gamma)}}{\sqrt{2}}\phi_1^*-\frac{e^{-t-s(\gamma)}}{\sqrt{2}}\phi_2^*\right)\right]\\
    &=\ell(t+s(\gamma)).
\end{align*}
Since $\ell$ is a unit speed geodesic, $\gamma$ translates along $\ell$ by $s(\gamma)$.
\end{proof}

\subsection{Points joined by a spacelike geodesic}\label{Spacelike geodesics 2}

We now discuss when a distinct pair of points in $\Hb_\tau^d\cup\partial_\infty^{\mathbb S}\Hb_\tau^d$ are joined by a spacelike geodesic in $\Hb_\tau^d$. Observe that if $[{\bf u}_1]=[e_+u_1+e_-\psi_1^*]$ and $[{\bf u}_2]=[e_+u_2+e_-\psi_2^*]$ are distinct points in $\mathbb S(\Rb_\tau^{d+1})$, then whether or not the inequality $\psi_1(u_2)\psi_2(u_1)>\psi_1(u_1)\psi_2(u_2)$ holds does not depend on the choice of representatives of $[{\bf u}_1]$ and $[{\bf u}_2]$.

\begin{proposition}\label{points along spacelike geodesics v2}
Let $[{\bf u}_1]=[e_+u_1+e_-\psi_1^*]$ and $[{\bf u}_2]=[e_+u_2+e_-\psi_2^*]$ be distinct points in $\Hb_\tau^d\cup\partial_\infty^{\mathbb S}\Hb_\tau^d\subset\mathbb S(\Rb_\tau^{d+1})$. There is a spacelike geodesic in $\Hb_\tau^d$ joining $[{\bf u}_1]$ and $[{\bf u}_2]$ if and only if $[u_1]\neq[u_2]$, $[\psi_1]\neq[\psi_2]$, and $\psi_1(u_2)\psi_2(u_1)>\psi_1(u_1)\psi_2(u_2)$. Moreover, such a spacelike geodesic is unique up to reparameterization. 
\end{proposition}

\begin{proof}
Suppose first that there is a spacelike geodesic joining $[{\bf u}_1]$ and $[{\bf u}_2]$. Then by \Cref{vectors} and \Cref{prop: spacelike geodesics}, they lie in the closure of $\ell_{v_1,\phi_1,v_2,\phi_2}(\Rb)$ in $\Hb_\tau^d\cup\partial_\infty^{\mathbb S}\Hb_\tau^d$ for some $v_1,v_2\in\Rb^{d+1}$ and $\phi_1,\phi_2\in\Rb^{d+1*}$ such that $\phi_1(v_1)=0=\phi_2(v_2)$ and $\phi_1(v_2)=1=\phi_2(v_1)$. It follows from the explicit description of $\ell_{v_1,\phi_1,v_2,\phi_2}$ that the forward direction holds.

For the proof of the converse, first notice that $\psi_1(u_2)\psi_2(u_1)>\psi_1(u_1)\psi_2(u_2)\ge0$, so we may choose representatives ${\bf u}_1=e_+u_1+e_-\psi_1^*$ of $[{\bf u}_1]$ and ${\bf u}_2=e_+u_2+e_-\psi_2^*$ of $[{\bf u}_2]$ to ensure that \[\psi_1(u_2)=-1=\psi_2(u_1)~.\] The remainder of the proof proceeds in three cases.

{\bf Case 1: $\psi_1(u_1)=0=\psi_2(u_2)$.}  By \Cref{prop: spacelike geodesics}, $\ell_{u_1,-\psi_1,u_2,-\psi_2}$ is the unique (up to reparameterization) spacelike geodesic in $\Hb_\tau^d$ that joins $[{\bf u}_1]$ and $[{\bf u}_2]$.

{\bf Case 2: $\psi_1(u_1)\neq 0=\psi_2(u_2)$ or $\psi_1(u_1)= 0\neq\psi_2(u_2)$.} We only give the proof in the case when $\psi_1(u_1)\neq 0=\psi_2(u_2)$; the other case is similar. 

Let $L:=[u_1]+[u_2]$ and let $[w_1]:=[\psi_1]\cap L$. Choose a projective identification $L\cong\mathbb{RP}^1=\partial_\infty\Hb^2$. Then we may view $[w_1]$, $[u_1]$, and $[u_2]$ as points along $\partial_\infty\Hb^2$. Let $g_1$ be the geodesic in $\Hb^2$ with $[w_1]$ and $[u_1]$ as its endpoints, and let $g$ be the geodesic in $\Hb^2$ that has $[u_2]$ as an endpoint, and is perpendicular to $g_1$. Then let $[v]$ be the endpoint of $g$ that is not $[u_2]$, and let $[\phi]:=[\psi_1]\cap[\psi_2]+[v]$. Notice that $[v]$ is the unique point along $L$ such that $([v],[u_2];[w_1],[u_1])=-1$.

Choose a representative $v\in\Rb^{d+1}$ of $[v]$ such that $\psi_1(v)<0$. By construction, the lines \[[w_1]<[v]<[u_1]<[u_2]<[w_1]\] lie along $L$ in this order, $\psi_1(w_1)=0$, $\psi_1(u_1)<0$, and $\psi_1(u_2)<0$. It follows that by scaling $u_1$, $u_2$, $\psi_1$, and $\psi_2$ by positive numbers, we may ensure that (in addition to $\psi_1(u_2)=-1=\psi_2(u_1)$) we have $u_1=u_2+v$. Since $([v],[u_2];[w_1],[u_1])=-1$, it follows that $w_1:=u_2-v\in\Rb^{d+1}$ is a representative of $[w_1]$. Then choose the representative $\phi\in\Rb^{d+1*}$ of $[\phi]$ such that $\phi(u_2)=-1$, and let ${\bf v}:=e_+v+e_-\phi^*$.

Note that \[\psi_2(v)=\psi_2(u_1-u_2)=-1~,\] so we may consider the curve $\ell_{v,-\phi,u_2,-\psi_2}$, which by \Cref{prop: spacelike geodesics} is a spacelike geodesic in $\Hb_\tau^d$ with $[{\bf v}]$ and $[{\bf u}_2]$ as its forward and backward endpoints in $\partial_\infty^{\mathbb S}\Hb_\tau^d$. We now show that $\ell_{v,-\phi,u_2,-\psi_2}$ joins $[{\bf u}_1]$ and $[{\bf u}_2]$ by verifying that $\ell_{v,-\phi,u_2,-\psi_2}(0)=[{\bf u}_1]$. Since we already know that $u_1=u_2+v$, by the definition of $\ell_{v,-\phi,u_2,-\psi_2}$, we need to show that $\psi_1=\psi_2+\phi$. For that, notice that $L+[\psi_1]\cap[\psi_2]=\Rb^{d+1}$, so it suffices to verify that $\psi_2+\phi$ and $\psi_1$ agree at $w_1$ and $v$. We compute \[(\psi_2+\phi)(w_1)=\phi(u_2)-\psi_2(v)=0\] and \[(\psi_2+\phi)(v)=\psi_2(v)=\psi_2(u_1-u_2) = -1~.\] Since $\psi_1(w_1)=0$ by the definition of $w_1$ and $\psi_1(v)=\psi_1(u_2-w_1) = -1$ by our normalization, it follows that $\psi_1=\psi_2+\phi$. 

To prove uniqueness, let $\widetilde{\ell}$ be any spacelike geodesic joining $[{\bf u}_1]$ and $[{\bf u}_2]$, and let $[[e_+\widetilde v+e_-\widetilde\phi^*]]$ be its endpoint in $\partial_\infty^{\mathbb P}\Hb_\tau^d$ different from $[[{\bf u}_2]]$. By \Cref{projective geometry description}, the point $[\widetilde v]\in\mathbb{RP}^d$ lies in $L$ and \[([\widetilde v],[u_2];[w_1],[u_1])=-1~.\] Hence $[\widetilde v]=[v]$. Moreover, the hyperplane $[\psi_1]\in\mathbb{RP}^{d*}$ passes through the point $[\widetilde\phi]\cap[\psi_2]\in\mathbb{RP}^d$, so \[[\widetilde\phi] =([\psi_1]\cap[\psi_2])+[v] = [\phi]~.\] Thus $\widetilde{\ell}$ and $\ell_{v,-\phi,u_2,-\psi_2}$ have the same endpoints in $\partial_\infty^{\mathbb P}\Hb_\tau^d$. Since they also have $[{\bf u}_2]$ as a common endpoint in $\partial_\infty^{\mathbb S}\Hb_\tau^d$, \Cref{prop: spacelike geodesics} and \Cref{comparison of boundaries} implies that they agree up to reparameterization.

{\bf Case 3: $\psi_1(u_1)\neq 0\neq\psi_2(u_2)$.}
As before, let $L:=[u_1]+[u_2]$. For each $i=1,2$, $[w_i]:=[\psi_i]\cap L$, and let $g_i$ be the geodesic in $\Hb^2$ with $[w_i]$ and $[u_i]$ as endpoints in $L\cong \partial \Hb^2$. Notice that
\[([w_1],[w_2];[u_1],[u_2])=\frac{\psi_1(u_2)\psi_2(u_1)}{\psi_1(u_1)\psi_2(u_2)}>1\]
by assumption, so the points 
\[[w_1]<[u_1]<[u_2]<[w_2]<[w_1]\] 
lie along $L$ in this cyclic order. Hence, $g_1$ and $g_2$ do not intersect in $\Hb^2$, so there is a unique geodesic $g$ in $\Hb^2$ that is perpendicular to both $g_1$ and $g_2$. Let $[v_1]$ and $[v_2]$ be the endpoints in $L\cong\partial_\infty\Hb^2$ of $g$ such that 
\[[w_1]<[v_1]<[u_1]<[u_2]<[v_2]<[w_2]<[w_1]\] 
lie along $L$ in this cyclic order. Then for each $i=1,2$, let $[\phi_i]:=([\psi_1]\cap[\psi_2])+[v_i]\in\mathbb{RP}^{d*}$. 

Choose representatives $v_1,v_2\in\Rb^{d+1}$ of $[v_1],[v_2]$ respectively, so that $\psi_1(v_1)<0$ and $\psi_1(v_2)<0$. Since $\psi_1(w_1)=0$, $\psi_1(u_1)<0$ and $\psi_1(u_2)<0$, by scaling $u_1$, $u_2$, $\psi_1$, and $\psi_2$ by positive numbers, we may ensure that (in addition to $\psi_1(u_2)=-1=\psi_2(u_1)$) we have $u_1=v_1+v_2$ and $u_2=e^tv_1+e^{-t}v_2$ for some $t\in\Rb$. Observe that $([v_1],[v_2];[w_i],[u_i])=-1$ for both $i=1,2$, so $w_1:=v_2-v_1\in\Rb^{d+1}$ is a representative of $[w_1]$ and $w_2:=e^{-t}v_2-e^tv_1\in\Rb^{d+1}$ is a representative of $[w_2]$. Then for both $i=1,2$, choose the representative $\phi_i\in\Rb^{d+1*}$ of $[\phi_i]$ such that $\phi_i(u_1)=-1$, and let ${\bf v}_i:=e_+v_i+e_-\phi_i^*$.

For both $i=1,2$,
\[\phi_i(v_{3-i})=\phi_i(u_1-v_i)=-1~,\]
so we may consider the curve $\ell_{v_1,-\phi_1,v_2,-\phi_2}$, which by \Cref{prop: spacelike geodesics} is a spacelike geodesic in $\Hb_\tau^d$ with $[{\bf v}_1]$ and $[{\bf v}_2]$ as its forward and backward endpoints in $\partial_\infty^{\mathbb S}\Hb_\tau^d$. We now show that $\ell_{v_1,-\phi_1,v_2,-\phi_2}$ joins $[{\bf u}_1]$ and $[{\bf u}_2]$ by verifying that 
\[\ell_{v_1,-\phi_1,v_2,-\phi_2}(0)=[{\bf u}_1]\quad\text{and}\quad\ell_{v_1,-\phi_1,v_2,-\phi_2}(t)=[{\bf u}_2]~.\]
Since we already know that $u_1=v_1+v_2$ and $u_2=e^tv_1+e^{-t}v_2$, by the definition of $\ell_{v_1,-\phi_1,v_2,-\phi_2}$, we need only to show that $\psi_1=C_1(\phi_1+\phi_2)$ and $\psi_2=C_2(e^t\phi_1+e^{-t}\phi_2)$ for some $C_1,C_2>0$.  For the former, we compute
\[(\phi_1+\phi_2)(w_1)=\phi_1(v_2)-\phi_2(v_1)=0\]
and
\[(\phi_1+\phi_2)(u_1)=\phi_1(v_2)+\phi_2(v_1)=-2~.\]
Since $\psi_1(w_1)=0$ and $\psi_1(u_1)<0$ by definition, the former holds for $C_1 =- \frac{\psi_1(u_1)}{2}$. Similarly, for the latter, we compute
\[(e^t\phi_1+e^{-t}\phi_2)(w_2)=\phi_1(v_2)-\phi_2(v_1)=0\]
and
\[(e^t\phi_1+e^{-t}\phi_2)(u_1)=-(e^t+e^{-t})~.\]
Since $\psi_2(w_2)=0$ and $\psi_2(u_1)=-1$, the latter holds for $C_2 = \frac{1}{e^t+e^{-t}}$.

To prove uniqueness, let $\widetilde{\ell}$ be any spacelike geodesic passing through $[{\bf u}_1]$ and $[{\bf u}_2]$, and let $[[e_+\widetilde v_1+e^-\widetilde\phi_1^*]]$ and $[[e_+\widetilde v_2+e_-\widetilde\phi_2]]$ denote its endpoints in $\partial_\infty^{\mathbb P}\Hb_\tau^d$. By \Cref{projective geometry description}, the points $[\widetilde v_1],[\widetilde v_2]\in\mathbb{RP}^d$ both lie in $L$, and \[([\widetilde v_1],[\widetilde v_2];[w_i],[u_i])=-1\] for $i=1,2$. Hence the geodesic in $\Hb^2$ with endpoints $[\widetilde v_1]$ and $[\widetilde v_2]$ is perpendicular to both $g_1$ and $g_2$. By the uniqueness of the common perpendicular, up to interchanging the indices, we have \[[\widetilde v_1]=[v_1] \quad \text{and} \quad [\widetilde v_2]=[v_2]~.\] Moreover, the hyperplanes $[\psi_1],[\psi_2]\in\mathbb{RP}^{d*}$ both contain the point $[\widetilde\phi_1]\cap[\widetilde\phi_2]\in\mathbb{RP}^d$, so \[[\widetilde\phi_i] = ([\psi_1]\cap[\psi_2])+[v_i] = [\phi_i]\] for $i=1,2$. Thus $\widetilde{\ell}$ and $\ell_{v_1,-\phi_1,v_2,-\phi_2}$ have the same endpoints in $\partial_\infty^{\mathbb P}\Hb_\tau^d$. By \Cref{prop: spacelike geodesics} and \Cref{comparison of boundaries}, either $\widetilde\ell$ and $\ell_{v_1,-\phi_1,v_2,-\phi_2}$ have disjoint images, or they agree up to reparameterization. Since they both pass through $[{\bf u}_1]$, the latter holds.
\end{proof}

As a consequence of \Cref{points along spacelike geodesics v2}, we have the following openness result for points joined by spacelike geodesics in $\Hb_\tau^d$.

\begin{corollary}\label{observation: joining by spacelike is an open condition}
    The subset of pairs of distinct points in $(\Hb_\tau^d\cup\partial_\infty^{\mathbb S}\Hb_\tau^d)^2$ that are joined by a spacelike geodesic in $\Hb_\tau^d$ is open.
\end{corollary}

\subsection{The tangent bundle to paracomplex hyperbolic spaces}

Rec all that $\T\Hb_\tau^d$ denotes the tangent bundle to $\Hb_\tau^d$. Let $\T^1\Hb_\tau^d\subset \T\Hb_\tau^d$ denote the \emph{spacelike unit tangent bundle to $\Hb_\tau^d$}, i.e. it is the set of spacelike unit tangent vectors in $\T\Hb_\tau^d$. Notice that $\T^1\Hb_\tau^d\subset \T\Hb_\tau^d$ is a closed subset. Indeed, if a sequence of spacelike vectors in $\T\Hb_\tau^d$ converges, then its limit is necessarily spacelike or lightlike, but the sequence cannot consist only of unit vectors in the latter case. 

There are two natural involutions on $\T^1\Hb_\tau^d$. The first is \emph{negation}, which is given by
\[\mathsf{ne}:\T^1\Hb_\tau^d\to\T^1\Hb_\tau^d~, v\mapsto -v~.\] 
For the second involution, recall from \Cref{comparison of boundaries} that for any unit speed spacelike geodesic $\ell:\Rb\to\Hb_\tau^d$, the map $\ell^{\opp}:\Rb\to\Hb_\tau^d$ is a unit speed spacelike geodesic in $\Hb_\tau^d$ that has the same forward and backward endpoints as $\ell$ in $\partial_\infty^{\Pb}\Hb_\tau^d$ and with opposite sides in $\partial_\infty^{\mathbb S}\Hb_\tau^d$. By \Cref{prop: spacelike geodesics}, $\ell^{\opp}$ is uniquely determined by that if we write $\ell(0) = (u_1,\psi_1)$ and $\ell^{\opp}(0)=(u_2,\psi_2)$, then $\psi_1(u_2)=\psi_2(u_1)=0$. We define
\begin{align}\label{opp definition}
    \opp: \T^1\Hb_{\tau}^d \to\T^1\Hb_{\tau}^d \quad\text{ given by } \quad \opp: \ell'(0)\mapsto (\ell^{\opp})'(0)~.
\end{align}
Note that these two involutions commute, and they both commute with the $\PGL(d+1,\Rb)$-action and the geodesic flow on $\T^1\Hb_\tau^d$ (which themselves also commute). 

Recall that we previously defined a map ${\opp}:\partial_\infty^{\mathbb S}\Hb_\tau^d\to\partial_\infty^{\mathbb S}\Hb_\tau^d$ in \eqref{eqn: covering involution}. Observe that the map ${\opp}:\T^1\Hb_{\tau}^d \to\T^1\Hb_{\tau}^d$ defined by \eqref{opp definition} extends to this map in the sense that if $x\in \partial_\infty^{\mathbb S}\Hb_\tau^d$ is the forward endpoint (equivalently, backward endpoint) of a unit speed spacelike geodesic $\ell$ in $\Hb_\tau^d$, then ${\opp}(x)\in \partial_\infty^{\mathbb S}\Hb_\tau^d$  is the forward endpoint (equivalently, backward endpoint) of $\ell^{\opp}$.

\section{The characterization theorem}\label{section: the characterization theorem}

The main goal of this section is to prove \Cref{theoremalpha: characterizations}. In \Cref{section: the horofunctions}, we define the horofunctions on $\T^1\Hb_\tau^d$ and explain their basic properties. Then, in \Cref{rfejkrflek}, we define the notion of a good flow space for a discrete subgroup of $\PGL(d+1,\Rb)$, showed that the weak hulls of tranverse groups are examples, and use the horofunctions to define the notion of a thick-thin decomposition on a good flow space. Then in \Cref{section: geometric finiteness via relative Anosovness}, we prove the forward directions of \Cref{theoremalpha: characterizations}. Finally, we prove the backward directions of all three parts of \Cref{theoremalpha: characterizations} in \Cref{section: existence of good flow spaces implies group is transverse}. 

\subsection{The horofunctions}\label{section: the horofunctions}
Given any point $x=(x^1,x^d)\in\partial_\infty^{\Pb}\Hb_\tau^d$, our goal is to define a horofunction on $\T^1\Hb_\tau^d$ centered at $x$ and explain some of its basic properties. This is a crucial ingredient needed for the characterization theorems, which are the main results in this section.

To do so, we use the following notation. Let $\norm{\cdot}$ denote the norm and $\angle(\cdot,\cdot)$ denote the angle on both $\Rb^{d+1}$ and $\Rb^{d+1*}$ induced by the standard inner products. For all non-zero $v\in\Rb^{d+1}$ and $\phi\in\Rb^{d+1*}$, the \emph{angle} between $[v]$ and $[\phi]$ is defined by
\[\angle([v],[\phi]):=\frac{\abs{\phi(v)}}{\norm{v}\norm{\phi}}~.\]
Then for any pair of points $a=(a^1,a^d)$ and $b=(b^1,b^d)$ in $\mathbb{RP}^d\times\mathbb{RP}^{d*}$, denote
\[\Delta(a,b):=\angle(a^1,b^d)\angle(a^d,b^1)~.\]

Recall that $\pi:\T^1 \Hb_\tau^d\to \Hb_\tau^d$ denotes the usual projection map, and that for any $z\in \T^1\Hb_\tau^d$, $z^+$ and $z^-$ denote the forward and backward endpoint of $z$ in $\partial_\infty^{\Pb}\Hb_\tau^d$. Recall also that we may identify $\Hb_\tau^d$ with the set of transverse pairs in $\mathbb{RP}^d\times\mathbb{RP}^{d*}$.

First, we define the map \[\Bc^0_x:\T^1\Hb_\tau^d\to(-\infty,\infty]\] by \[\mathcal{B}^0_x(z) = -\frac{1}{2}\log \frac{\Delta(x,\pi(z))}{\sqrt{\Delta(\pi(z),\pi(z))}}=\frac{1}{2}\log \abs{\frac{\psi(u)\norm{\phi}{\norm{v}}}{\psi(v)\phi(u)}}~.\] where $u,v\in\mathbb{R}^{d+1}$ and $\psi,\phi\in\mathbb{R}^{d+1*}$ are the vectors and covectors such that $x=([v],[\phi])$ and $\pi(z)=([u],[\psi])$. Using this, we define \[\Bc^1_x:\T^1\Hb_\tau^d\to(-\infty,\infty]\] by \[\mathcal{B}^1_x(z)= -\frac{1}{4}\log \big(\exp(-4 \mathcal{B}^0_x(z))+\exp(-4 \mathcal{B}^0_x(\opp(z)))\big)~.\]

Next, we define the map \[\Bc^2_x:\T^1\Hb_\tau^d\to(-\infty,\infty]\] by 
\begin{align*}
    \mathcal{B}^2_x (z)
    & = \frac{1}{4}\log \frac{\Delta(z^+,z^-)}{\Delta(x,z^+)\Delta(x,z^-)}=\frac{1}{4}\log \abs{\frac{\phi^+(v^-) \phi^-(v^+)\norm{v}^2 \norm{\phi}^2}{\phi(v^+)\phi(v^-) \phi^+(v) \phi^-(v)}}
\end{align*}
where $v,v^+,v^-\in\Rb^{d+1}$ and $\phi,\phi^+,\phi^-\in\Rb^{d+1*}$ are the vectors and covectors such that $x=([v],[\phi])$, $z^+=([v^+],[\phi^+])$, and $z^-=([v^-],[\phi^-])$. Note that $\Bc^2_x(z)$ depends only on $z^+$ and $z^-$, so $\Bc^2_x$ is invariant under the geodesic flow. 

\begin{definition}\label{definition: horofunctions}
    The \emph{horofunction centered at $x\in\partial_\infty^{\Pb}\Hb_\tau^d$} is the map \[\Bc_x:\T^1\Hb_\tau^d\to\Rb~,\] defined by \[\mathcal{B}_x(z) = -\log \big( \exp(-\mathcal{B}_x^1(z))+ \exp(-\mathcal{B}_x^2(z))\big)~.\]
\end{definition}

Observe that $\Bc^1_x$ and $\Bc^2_x$, and hence $\Bc_x$, are invariant under ${\opp}$ and negation. In addition, for all $x\in\partial_\infty^{\Pb}\Hb_\tau^d$ and $z\in\T^1\Hb_\tau^d$, since 
\[\exp(-\Bc_x(z))=\exp(-\mathcal{B}_x^1(z))+ \exp(-\mathcal{B}_x^2(z))~,\]
we have
\[\max\{\exp(-\mathcal{B}_x^1(z)), \exp(-\mathcal{B}_x^2(z))\}\leqslant \exp(-\Bc_x(z))\] \[\leq 2\max\{\exp(-\mathcal{B}_x^1(z)), \exp(-\mathcal{B}_x^2(z))\}~,\]
which implies that 
\begin{align}\label{horofunction component bound}
    \min \{\mathcal{B}^1_x(z),\mathcal{B}^2_x (z)\} - \log 2\leqslant \mathcal{B}_x(z) \leqslant \min \{\mathcal{B}^1_x(z),\mathcal{B}^2_x (z)\}~.
    \end{align}
For similar reasons, 
\begin{align}\label{horofunction component bound2}
\min \{\mathcal{B}^0_x(z),\mathcal{B}^0_x ({\opp}(z))\} - \frac{\log 2}{4}\leqslant \mathcal{B}^1_x(z) \leqslant \min \{\mathcal{B}^0_x(z),\mathcal{B}^0_x ({\opp}(z))\}.
\end{align}
Furthermore, notice that $\Bc_x$ is continuous because both $\Bc_x^0$ and $\Bc_x^2$ are continuous (as maps from $\T^1\Hb_\tau^d$ to $(-\infty,+\infty]$). 

The following proposition records another observation that follows from a straightforward computation.

\begin{proposition}\label{properties of the horofunction}
For any $g\in \PGL(d+1,\Rb)$, any $x\in\partial_\infty^{\Pb}\Hb_\tau^d$, any $z\in\T^1\Hb_\tau^d$, and any $i=0,1,2$, we have
\[\mathcal{B}^i_{gx}(gz)-\mathcal{B}^i_x (z) = \frac{1}{2} \left(\log\frac{\norm{g v}}{\norm{v}}+ \log\frac{\norm{g\phi}}{\norm{\phi}} \right)~,\]
where $v\in\mathbb{R}^{d+1}$ and $\phi\in\mathbb{R}^{d+1*}$ is the vector and covector such that $x=([v],[\phi])$. In particular, 
\[\mathcal{B}_{gx}(gz)-\mathcal{B}_x (z) = \frac{1}{2} \left(\log\frac{\norm{g v}}{\norm{v}}+ \log\frac{\norm{g\phi}}{\norm{\phi}} \right)~.\]
\end{proposition}

As an immediate consequence of the formula on the right hand side of the equalities in \Cref{properties of the horofunction}, we have the following corollary.

\begin{corollary}\label{unipotents preserve horofunctions}
    If $g\in\PGL(d+1,\Rb)$ is an weakly unipotent element that fixes $x\in\partial_\infty^{\Pb}\Hb_\tau^d$, then for any $z\in\T^1\Hb_\tau^d$, \[\mathcal{B}^i_{x}(gz) =\mathcal{B}^i_x (z) \quad \text{for}\ i = 0,1,2 \quad \text{and} \quad \mathcal{B}_{x}(gz) =\mathcal{B}_x (z)~.\]
\end{corollary}

The fact that the right hand side of the equalities in \Cref{properties of the horofunction} do not depend on $z$ allows to deduce the following corollary.

\begin{corollary}\label{constant difference}
    Let $g\in\PGL(d+1,\Rb)$, and let $x,y\in\partial_\infty^{\Pb}\Hb_\tau^d$ such that $gx=y$. Then there is some $B\in\Rb$ such that for all $C\in\Rb$, we have \[g\{z\in\T^1\Hb_\tau^d:\Bc_x(z)>C\}=\{z\in\T^1\Hb_\tau^d:\Bc_y(z)>C+B\}~.\]
\end{corollary}

In general, the superlevel sets of the horofunctions on $\T^1\Hb_\tau^d$ do not behave like horoballs in hyperbolic space. For instance, one can verify that for any $x,x'\in\partial_\infty^{\Pb}\Hb_\tau^d$ and any $C,C'>0$, the intersection
\[\{z\in\T^1\Hb_\tau^d:\Bc_x(z)>C\}\cap\{z\in\T^1\Hb_\tau^d:\Bc_{x'}(z)>C'\}\]
is non-empty. As such, we almost always restrict our horofunctions to weak hulls of transverse sets, which we now define.

\begin{definition}
    Let $\Lambda\subset\partial_\infty^{\Pb}\Hb_\tau^d$ be a transverse set, that is, any pair of distinct points in $\Lambda$ are transverse (as flags in $\Fc_{1,d}$).
    \begin{enumerate}
        \item The \emph{weak hull of $\Lambda$} is the set \[\WH(\Lambda) \subset \T^1\Hb_\tau^d\] of all spacelike unit vectors that are tangent to geodesics with both endpoints in $\Lambda$.
        \item For every $x\in\Lambda$ and $C>0$, the \emph{horoball of $\WH(\Lambda)$ centered at $x$ of radius $C$} is the set \[\{z\in\WH(\Lambda):\Bc_x(z) > C\}~.\]
    \end{enumerate}
\end{definition}

Fix a closed, transverse set $\Lambda\subset \partial_\infty^{\Pb}\Hb_\tau^d$ that has at least two points (this implies that its weak hull is non-empty). For every $x\in\Lambda$ and $C>0$, let $H_x(C)$ denote the horoball of $\WH(\Lambda)$ centered at $x$ of radius $C$. We now prove several results about such horoballs, which are reminiscent of the behavior of horoballs in hyperbolic space. First, we prove the following disjointness property of horoballs.

\begin{proposition}\label{lemma: disjointness of two horoballs}
    Let $x_0,y_0\in\Lambda$ be distinct. There exists $C>0$ and open neighborhoods $U$ and $V$ in $\Lambda$ of $x_0$ and $y_0$ respectively, such that for all $x\in U$ and $y\in V$, $H_x(C)\cap H_y(C)$ is empty.
\end{proposition}

To do so, we use the following lemma.

\begin{lemma}\label{proposition: two of three}
    Let $(z_n)$ be a sequence in $\WH(\Lambda)$ for which there is some $z^+,z^-\in\Lambda$ and some $a,a'\in\Hb_\tau^d\cup\partial_\infty^{\Pb}\Hb_\tau^d$ such that as $n\to\infty$,
    \[z_n^+\to z^+, \quad z_n^-\to z^-, \quad\pi(z_n)\to a,\quad\text{and}\quad\pi({\opp}(z_n))\to a'~.\] 
    Let $(x_n)$ be a sequence in $\Lambda$ that converges to some $x\in\Lambda$. If $\mathcal{B}_{x_n}(z_n)\to +\infty$ as $n\to\infty$, then either
    \begin{itemize}
        \item $z^+=x=z^-$,
        \item $z^-\neq z^+=x=a=a'$, or 
        \item $z^+\neq z^-=x=a=a'$. 
    \end{itemize}
\end{lemma}

\begin{proof}
Since $\mathcal{B}_{x_n}(z_n)\to +\infty$ as $n\to\infty$, \eqref{horofunction component bound} and \eqref{horofunction component bound2} imply that the same is true for $\Bc^0_{x_n}(z_n)$, $\Bc^0_{x_n}({\opp}(z_n))$, and $\Bc^2_{x_n}(z_n)$. 

The fact that $\Bc^2_{x_n}(z_n)\to+\infty$ as $n\to\infty$ implies
\[\Delta(x,z^+)\Delta(x,z^-)=\lim_{n\to\infty}\Delta(x_n,z_n^+)\Delta(x_n,z_n^-)=0~.\]
Hence, $x$ is not transverse, and therefore equal, to $z^+$ or $z^-$. If $z^+=x=z^-$, then we are done. Otherwise, by negating $z_n$, we may assume that $z^+=x\neq z^-$. Since $\pi(z_n)$ and $\pi({\opp}(z_n))$  both lie along a spacelike geodesics joining $z_n^+$ and $z_n^-$, the cross ratio condition in \Cref{projective geometry description} implies that $a$ and $a'$ are either both equal to $z^+$, both equal to $z^-$, or they both lie along a spacelike geodesic joining $z^+$ and $z^-$. 

By \Cref{observation: a spacelike geodesic is totally transverse}~(1), to prove that $a=a'=z^+$, it now suffices to verify that $a$ is not transverse to $z^+$. Since $\Bc^0_{x_n}(z_n)\to+\infty$ as $n\to\infty$, it follows that 
\[\Delta(x,a)=\lim_{n\to\infty}\Delta(x_n,\pi(z_n))=0~,\]
so $a$ is not transverse to $x=z^+$. 
\end{proof}

\begin{proof}[Proof of \Cref{lemma: disjointness of two horoballs}]
Suppose not. Then there is a sequence $(z_n)$ in $\WH(\Lambda)$ and sequences $(x_n)$ and $(y_n)$ in $\Lambda$ such that $x_n\to x_0$, $y_n\to y_0$, and 
\[z_n\in H_{x_n}(n)\cap H_{y_n}(n)\]
for all $n\in\Nb$. Taking a subsequence, we may assume that there exist $a,a'\in \Hb_\tau^d\cup\partial_\infty\Hb_\tau^d$ and $z^+,z^-\in\partial_\infty\Hb_\tau^d$ such that 
\[\pi(z_n)\to a,\quad\pi({\opp}(z_n))\to a',\quad z_n^+\to z^+,\quad\text{and}\quad z_n^-\to z^-~.\] 
Observe that since $\Bc_{x_n}(z_n)$ and $\Bc_{y_n}(z_n)$ both grow to $+\infty$ as $n\to\infty$, \Cref{proposition: two of three} implies that $x_0$ and $y_0$ are both equal to two of $\{z^-,z^+,a\}$. However, this is impossible since $x_0\neq y_0$.
\end{proof}

The next proposition describes how the flow lines in $\WH(\Lambda)$ intersect the horoballs.

\begin{proposition}\label{prop: unimodal}
    Let $z\in\WH(\Lambda)$ be any point.
    \begin{enumerate}
        \item For all $t\in\Rb$, $\Bc_{z^\pm}(\varphi^t(z))-\Bc_{z^\pm}(z)=\pm t$.
        \item For any $x\in\Lambda\setminus\{z^+,z^-\}$, the function $t\mapsto \Bc_x(\varphi^t(z))$ has a unique critical point at \[t_0:= \frac{1}{4}\log\frac{\Delta(x,z^-)}{\Delta(x,z^+)}- \frac{1}{4}\log\frac{\Delta(\pi(z),z^-)}{\Delta(\pi(z),z^+)}~.\] Furthermore, $t_0$ is a local maximum, and \[\lim_{t\to \pm \infty}\Bc_x(\varphi^t(z)) = -\infty~.\] In particular, this function is unimodal.
    \end{enumerate}
\end{proposition}

\begin{proof}
    Note that the formulas we are going to prove are equivariant under $\mathsf{PGL}(d+1,\mathbb{R})$ action. As a result, let $z\in\WH(\Lambda)$, after $\mathsf{PGL}(d+1,\mathbb{R})$-translation we can assume \[z^+=([e_1],[e_{d+1}^*]),\quad z^-=([e_{d+1}],[e_1^*])~,\quad\text{and}\quad\pi(z)=([e_1+e_{d+1}],[e_{d+1}^*+e_1^*])~,\] where $(e_1,\dots,e_{d+1})$ and $(e_1^*,\dots,e_{d+1}^*)$ are the standard basis of $\Rb^{d+1}$ and $\Rb^{d+1*}$.  Recall that $\varphi^t$ denotes the geodesic flow in $\T^1\Hb_\tau^d$. By \Cref{prop: spacelike geodesics} and \Cref{comparison of boundaries}, for all $t\in\Rb$, we have \[\pi(\varphi^t(z))=([e^te_1+e^{-t}e_{d+1}],[e^te_{d+1}^*+e^{-t}e_1^*])\] and \[\pi(\varphi^t({\opp}(z)))=([e^te_1-e^{-t}e_{d+1}],[-e^te_{d+1}^*+e^{-t}e_1^*])~.\]

    (1) Using the formulas above, it is a straightforward computation to verify that \[\mathcal{B}^1_{z^\pm}(\varphi^t(z))=\pm t~.\] Observe that $\exp(-\Bc^2_{z^+}(\varphi^t(z)))=0=\exp(-\Bc^2_{z^-}(\varphi^t(z)))$, so $\Bc_{z^{\pm}}(\varphi^t(z))=\Bc^1_{z^{\pm}}(\varphi^t(z))=\pm t$. This proves (1).

    (2) Let $v=(a_1,\dots,a_{d+1})^T\in\Rb^{d+1}$ and $\phi=(b_1,\dots,b_{d+1})\in\Rb^{d+1^*}$ such that $x=([v],[\phi])$. Then $a_1$, $b_1$, $a_{d+1}$, $b_{d+1}$ are all non-zero. By a straightforward computation, \[\mathcal{B}^1_x(\varphi^t(z))= -\frac{1}{4}\log\frac{A(t)}{4\norm{\phi}^2\norm{v}^2}~,\] where \[A(t)=2(a_{d+1} b_1)^2 e^{4t}+2(a_1 b_{d+1})^2e^{-4t} + 2 a_1 b_1 a_{d+1} b_{d+1} + 2(a_1 b_1 + a_{d+1} b_{d+1})^2~.\] Using standard arguments from one variable calculus, one deduces that $A$ has a unique critical point at $\frac{1}{4}\log\frac{a_1b_{d+1}}{a_{d+1}b_1}$, which is a local minimum. At the same time, one can verify that \[\frac{\Delta(x,z^-)}{\Delta(x,z^+)}=\frac{a_1 b_{d+1}}{a_{d+1} b_1} \quad \text{and} \quad \log\frac{\Delta(\pi(z),z^-)}{\Delta(\pi(z),z^+)}=0~.\] so the map $t\mapsto \mathcal{B}^1_x(\varphi^t(z))$ has a unique critical point at $t_0$, which is a local maximum. Since the function $t\mapsto\Bc_x^2(\varphi^t(z))$ is constant as a function of $t$, (2) follows.
\end{proof}

\Cref{prop: unimodal}(2) implies that if $z\in\WH(\Lambda)$ and $x\in\Lambda\setminus\{z^+,z^-\}$, then the map $t\mapsto \Bc_x(\varphi^t(z))$ is either increasing on $(-\infty,0]$ or decreasing on $[0,+\infty)$. As a consequence of \Cref{prop: unimodal}(1), we deduce the following corollary, which we will use repeatedly later in the paper.

\begin{corollary}\label{bounded converges to unbounded}
Let $(z_n)$ be a sequence in $\WH(\Lambda)$ such that the following hold:
\begin{enumerate}
    \item the sequences $(z_n^+)$ and $(z_n^-)$ converge to distinct points $x$ and $y$ in $\Lambda$ respectively,
    \item for all $n\in\Nb$, the map $t\mapsto \Bc_x(\varphi^t(z_n))$ is increasing on $(-\infty,0]$,
    \item there exists $C>0$ such that $\Bc_x(z_n)<C$ for all $n\in\Nb$. 
\end{enumerate} 
Then no subsequence of $(\pi(z_n))$ can converge to $x$.
\end{corollary}

\begin{proof}
Suppose otherwise. Then by taking a subsequence, we may assume that $\pi(z_n)\to x$ as $n\to\infty$. Let 
\[R_n:=\{\varphi^t(z_n):t\in(-\infty,0]\}~,\]
and note that the sequence $(R_n)$ converges to a flowline $g$ in $\WH(\Lambda)$ that joins $x$ and $y$. \Cref{prop: unimodal}(1) implies that along $g$, the value of $\Bc_x(\cdot)$ is not bounded above, so there is some $w\in g$ such that $\mathcal{B}_x (w)>C$. Let $(w_n)$ be a sequence in $\WH(\Lambda)$ such that $w_n\in R_n$ for all $n$, and $w_n\to w$ as $n\to+\infty$. By the continuity of $\Bc_x$, for sufficiently large $n$, $\Bc_x(w_n)>C$. However, for all $n\in\Nb$, since 
$t\mapsto\Bc_x(\varphi^t(z_n))$ is increasing on $(-\infty,0]$, we have
\[\Bc_x(w_n)\leqslant \Bc_x(z_n)\leqslant C\]
which is a contradiction.
\end{proof}

\begin{remark}\label{remark: horofunctions for general Lie groups}
Our construction of the horofunctions is natural in the general Lie group setting. Let $\mathsf G$ be a connected semisimple real Lie group with finite center, fix an Iwasawa decomposition $\mathsf G=\mathsf K\mathsf A\mathsf N$, and denote by $\mathfrak a$ the Lie algebra of $\mathsf A$. As higher rank analogues of the classical Busemann cocycle and Gromov product, Quint \cite{quint2002mesures} introduced the $\mathfrak a$-valued Iwasawa cocycle $\beta:\mathsf G\times\Fc\to\mathfrak a$ on the full flag manifold $\Fc$ using the Iwasawa decomposition of $\mathsf G$, and Sambarino \cite{sambarino2015orbital} introduced the associated $\mathfrak a$-valued Gromov product $\Gc:\Fc^{(2)}\to\mathfrak a$, where $\Fc^{(2)}$ denotes the set of transverse pairs in $\Fc\times\Fc$.

We denote by $\mathfrak g$ the Lie algebra of $\mathsf G$.
Suppose that $\sigma:\mathfrak g\to\mathfrak g$ is an involution that commutes with the Cartan involution associated with $\mathsf K$, and that $\mathsf H\subset\mathsf G$ is a subgroup whose Lie algebra is $\mathfrak g^\sigma:=\{X\in\mathfrak g:\sigma(X)=X\}$. Then we have a decomposition 
\[\mathfrak g=\mathfrak h\oplus(\mathfrak q\cap\mathfrak k)\oplus(\mathfrak q\cap\mathfrak p),\] 
where $\mathfrak k$ and $\mathfrak p$ are the eigenspaces of the above Cartan involution on $\mathfrak g$ corresponding to eigenvalues $1$ and $-1$ respectively, and $\mathfrak h$ and $\mathfrak q$ are the eigenspaces of $\sigma$ with eigenvalues $1$ and $-1$ respectively. In the pseudo-Riemannian symmetric space $\mathsf G/\mathsf H$, the tangent space to the identity coset is identified with $\mathfrak q$, and $\mathfrak q\cap\mathfrak p\subset\mathfrak q$ is the subspace of spacelike tangent vectors. Up to conjugation, we may assume that $\mathfrak a$ is $\sigma$-invariant and that $\mathfrak a_0:=\mathfrak a\cap\mathfrak q$ is a maximal abelian subspace of $\mathfrak q\cap\mathfrak p$. See \cite[Chapter~X]{helgason1979differential} for more details.

Assume that $\dim\mathfrak a_0=1$, or equivalently, that $\mathsf G/\mathsf H$ is a rank-one. Let $X_0\in\mathfrak a_0$ be a unit vector. Choose a set of simple roots $\Delta$ for $\mathsf G$ on $\mathfrak a$, and let $\mathfrak a^+\subset\mathfrak a$ denote the associated positive Weyl chamber. By conjugating, we may assume that $X_0\in\overline{\mathfrak a^+}$. Set \[\theta := \{\alpha\in\Delta:\alpha(X_0)>0\} = \{\alpha\in\Delta:\alpha(X_0)\neq0\}~.\] 
Denote by $\mathsf P_\theta$ the associated parabolic subgroup and by $\Fc_\theta:=\mathsf G/\mathsf P_\theta$ the corresponding partial flag manifold. Also set \[\mathsf L = \Stab_{\mathsf H}(X_0) = \{h\in\mathsf H:\Ad(h)X_0=X_0\}~.\] Then $\mathsf G/\mathsf L$ is naturally identified with the $\mathsf G$-orbit in the spacelike unit tangent bundle of $\mathsf G/\mathsf H$ determined by $X_0$. If $\mathsf H$ acts transitively on the unit spacelike vectors in $\mathfrak q$, this orbit is the spacelike unit tangent bundle of $\mathsf G/\mathsf H$. If $z=g\mathsf L$, its forward and backward endpoints are \[z^+=g\mathsf P_\theta \quad\text{and}\quad z^-=gw_0\mathsf P_{\iota(\theta)} \quad \text{respectively},\] where $w_0\in N_{\mathsf K}(\mathsf A)$ represents the longest Weyl group element and $\iota$ denotes the opposition involution.

Since \[\mathfrak a = (\mathfrak a\cap\mathfrak h)\oplus\Rb X_0~,\] there is a unique linear form $\psi\in\mathfrak a^*$ such that $\psi(X_0)=1$ and $\psi\vert_{\mathfrak a\cap\mathfrak h}=0$. We define $\Bc^0: \Fc_\theta\times \mathsf G/\mathsf L \to (-\infty,+\infty]$ to be a potential for $\psi\circ\beta$, in the sense that \[\Bc^0_{gx}(gz)-\Bc^0_x(z) = \psi(\beta(g,x_0))\] for any $x\in\Fc_\theta$, $z\in \mathsf G/\mathsf L$ and $g\in\mathsf G$, where $x_0\in\Fc$ is a (any) lift of $x$ in $\Fc$. There is a natural Hopf map \[\mathsf G/\mathsf L \to \Fc_\theta^{(2)}\times\mathfrak a_0~,\] where $\Fc_\theta^{(2)}$ denotes the set of transverse pairs in $\Fc_\theta\times\Fc_{\iota(\theta)}$. The fiber of this Hopf map is compact and we define $\Bc^1: \Fc_\theta\times \mathsf G/\mathsf L \to (-\infty,+\infty]$ to be the logarithmic symmetrization of $\Bc^0$ over the fibers. We then define $\Bc^2: \Fc_\theta\times \mathsf G/\mathsf L \to (-\infty,+\infty]$ by \[\Bc_x^2(z) = \psi(\Gc(x_0,z_0^+) + \Gc(x_0,z_0^-) - \Gc(z_0^+,z_0^-))~,\] where $x_0,z_0^+,z_0^-\in\Fc$ are some (any) lifts of $x,z^+,z^-$ in $\Fc$ respectively.

One can verify directly that when $\mathsf G=\PGL(d+1,\Rb)$ and $\mathsf H=\mathsf P(\Rb^*\times\GL(d,\Rb))$, these constructions agree with our previous definitions.
\end{remark}

\subsection{Good flow spaces for discrete subgroups of $\PGL(d+1,\Rb)$}\label{rfejkrflek}

Let $\Gamma\subset\PGL(d+1,\Rb)$ be a discrete, infinite subgroup. We now define good flow spaces for $\Gamma$, discuss some of its basic properties, and define the notion of a thick-thin decomposition for good flow spaces.

\begin{definition}
    For a flow-invariant subset $U\subset \T^1\Hb_\tau^d$, the \emph{endpoint set of $U$} is \[\Lambda_U:=\{z^+:z\in U\}\cup\{z^-:z\in U\}\subset\partial_\infty^{\Pb}\Hb_\tau^d\cong\Fc_{1,d}~.\] We say that $U$ is \emph{convex} if for any distinct points $x,y\in\Lambda_U$, there exists $z\in U$ such that $z^+=x$ and $z^-=y$. We then say that $U$ is a \emph{good flow space for $\Gamma$} if it is closed, non-empty, $\Gamma$-invariant, negation-invariant, $\opp$-invariant, convex, and the $\Gamma$-action on $U$ is properly discontinuous.
\end{definition}

We now prove that when $\Gamma$ is projective transverse, then $\WH(\Lambda_\Gamma)$ is an example of a good flow space for $\Gamma$. Observe that 
\[\Dc_\Gamma := \{z \in \T^1\Hb_\tau^d : \forall \eta \in \Lambda_{\Gamma}, \text{ at least one of } z^+, z^- \text{ is transverse to } \eta\}~,\]
is an open, non-empty, $\Gamma$-invariant, flow-invariant, negation-invariant, ${\opp}$-invariant subset of $\T^1\Hb_\tau^d$ that contains $\WH(\Lambda_\Gamma)$.

\begin{proposition}\label{proposition: domain of discontinuity}
    If $\Gamma\subset\PGL(d+1,\Rb)$ is a projective transverse subgroup, then the $\Gamma$-action on $\Dc_{\Gamma}$ is properly discontinuous. In particular, $\WH(\Lambda_\Gamma)$ is a good flow space for $\Gamma$. 
\end{proposition}

\begin{proof}
Suppose that the first claim fails. Then there exists an escaping sequence $(\gamma_n)$ in $\Gamma$ and a convergent sequence $(z_n)$ in $\Dc_{\Gamma}$ such that \[z_n\to z\quad\text{ and }\quad w_n:=\gamma_n z_n \to w \quad \text{as} \quad n \to \infty\] for some $w,z\in \Dc_{\Gamma}$. Since $\Gamma$ is projective divergent, \Cref{lemma: projective north-south dynamics} implies that by taking a subsequence, we may assume that the sequence $(\gamma_n)$ has an attracting and repelling point in $\Lambda_\Gamma$, denoted $x$ and $y$ respectively. By the definition of $\Dc_{\Gamma}$, at least one of $z^+$ or $z^-$ is transverse to $y$. By negating, we may assume without loss of generality that $z^+$ is transverse to $y$. Since $\lim_{n\to\infty}z_n^+=z^+$, by \Cref{lemma: projective north-south dynamics}, 
\[w^+=\lim_{n\to\infty}\gamma_n z_n^+= x~.\] 
Notice that $(\gamma_n^{-1})$ has $y$ and $x$ as its attracting and repelling points respectively. Since $w^-$ and $w^+=x$ are transverse and $\lim_{n\to\infty}w_n^-=w^-$, we have
\[z^-=\lim_{n\to\infty}z_n^-=\lim_{n\to\infty}\gamma_n^{-1}w_n^-=y~.\]

Recall that as $\PGL(d+1,\Rb)$-spaces, we have identified
\[\Hb_\tau^d\cup\partial_\infty^{\Pb}\Hb_\tau^d\cong\mathbb{RP}^d\times\mathbb{RP}^{d*}~,\]
see \Cref{real projective}. With this identification,  \Cref{observation: a spacelike geodesic is totally transverse}(1) implies that $\pi(z)\in\Hb_\tau^d\subset\mathbb{RP}^d\times\mathbb{RP}^{d*}$ is transverse to $z^-=y$, so there is an open neighborhood $O\subset\Hb_\tau^d$ of $\pi(z)$ such that every point in its closure is transverse to $y$. Since $\pi(z_n)\in O$ for $n$ large enough, \Cref{lemma: projective north-south dynamics} implies that $\pi(w_n)=\gamma_n \pi(z_n)\to x=w^+$. However, by assumption, $\pi(w_n)\to \pi(w)$, and again by \Cref{observation: a spacelike geodesic is totally transverse}(1), $\pi(w)$ is transverse to $w^+$. This is a contradiction, so the first claim holds.

Observe that $\WH(\Lambda_\Gamma)$ is a closed, non-empty, $\Gamma$-invariant, flow-invariant, negation-invariant, ${\opp}$-invariant, convex subset of $\T^1\Hb_\tau^d$. Since $\WH(\Lambda_\Gamma)$ is a closed subset of $\Dc_\Gamma$, it follows from the first claim that the $\Gamma$-action on $\WH(\Lambda_\Gamma)$ is properly discontinuous, so $\WH(\Lambda_\Gamma)$ is a good flow space for $\Gamma$.
\end{proof}

Let $U$ be a good flow space for $\Gamma$. By \Cref{prop: spacelike geodesics}, $\Lambda_U$ is transverse and $U=\WH(\Lambda_U)$. Also, $U$ admits a natural $\Gamma\times\Rb\times\Zb_2$-action, where the $\Rb$-action is given by the geodesic flow and the $\Zb_2$-action is induced by the involution ${\opp}$. The following proposition gives some basic properties of the $\Gamma$-action on $\Lambda_U$.

\begin{proposition}\label{proposition: convergence action}
    \begin{enumerate}
        \item $\Gamma$ acts on $\Lambda_U$ as a convergence group.
        \item Suppose that the $\Gamma$-action on $\Lambda_U$ is non-elementary and the $\Gamma\times\Rb\times \Zb_2$-action on $U$ is topologically transitive. Then the limit set of the $\Gamma$-action on $\Lambda_U$ is all of $\Lambda_U$. In particular, the $\Gamma$-action on $\Lambda_U$ is minimal and $\Lambda_U$ is perfect.
    \end{enumerate} 
\end{proposition}

\begin{proof}
    (1). Since $\Lambda_U\subset \Fc_{1,d}$ is closed, it is compact.
    By \Cref{Thm: bowditch} it suffices to show that the $\Gamma$-action on \[\Lambda_U^{(3)} = \{(a,b,c)\in\Lambda_U^3: a,b,c\text{ are mutually distinct}\}\] is properly discontinuous. Suppose for the purpose of contradiction that this were not the case. Then $\Lambda_U$ contains at least $3$ points, and there exists an escaping sequence $(\gamma_n)$ in $\Gamma$ and a sequence $((a_n,b_n,c_n))$ in $ \Lambda_U^{(3)}$ such that $(a_n,b_n,c_n)\to (a,b,c)\in \Lambda_U^{(3)}$ and $\gamma_n(a_n,b_n,c_n)\to (a',b',c')\in \Lambda_U^{(3)}$. 

    For each $n\in \Nb$, let $\ell_{n,ab}$ be a unit speed spacelike geodesic with $a_n$ and $b_n$ as its backward and forward endpoints in $\partial_\infty^{\Pb}\Hb_\tau^d\cong\Fc_{1,d}$ respectively. By taking a subsequence, we may assume that as subsets of $\Hb_\tau^d$, the sequence $(\ell_{n,ab}(\Rb))$ converges to $\ell_{ab}(\Rb)$, where $\ell_{ab}$ is a unit speed spacelike geodesic with backward and forward endpoints $a$ and $b$ respectively. Since $\gamma_n (a_n,b_n)\to (a',b')$, by taking another subsequence, we may also assume that $(\gamma_n\ell_{n,ab}(\Rb))$ converges to $\ell_{a'b'}(\Rb)$, where $\ell_{a'b'}$ is a unit speed spacelike geodesic with $a'$ and $b'$ as its backward and forward endpoints respectively. 

    Let $v\in\T^1\Hb_\tau^d$ be a tangent vector to $\ell_{ab}$. Then for each $n\in\Nb$, let $v_n\in\T^1\Hb_\tau^d$ be a tangent vector to $\ell_{n,ab}$ such that $v_n$ converges to $v$. Notice that $(v_n)$ and $v$ lie in $U$. Since the $\Gamma$-action on $U$ is properly discontinuous, $(\gamma_nv_n)$ does not converge to a vector along $\ell_{a'b'}$. Thus, by taking a subsequence, we may assume that $(\pi(\gamma_nv_n))$ converges to either $a'$ or $b'$. 

    We now only consider the case when the sequence $(\pi(\gamma_nv_n))$ converges to $a'$; the other case when $(\pi(\gamma_nv_n))$ converges to $b'$ is similar. Then for each $n\in\Nb$, let $\ell_{n,cb}$ be a unit speed spacelike geodesic with $c_n$ and $b_n$ as its backward and forward endpoints (in $\partial_\infty^\mathbb{P}\Hb_\tau^d$) respectively, and whose forward endpoint in $\partial_\infty^\mathbb{S}\Hb_\tau^d$ agrees with that of $\ell_{n,ab}$. Taking a further subsequence, we may assume that the sequence $(\ell_{n,cb}(\Rb))$ converges to $\ell_{cb}(\Rb)$, where $\ell_{cb}$ is a unit speed spacelike geodesic with $c$ and $b$ as its backward and forward endpoints respectively. Observe that $\ell_{cb}$ and $\ell_{ab}$ share a forward endpoint $\widetilde b\in \partial_\infty^\mathbb{S}\Hb_\tau^d$. By \Cref{observation: joining by spacelike is an open condition} there are open sets $W_{\widetilde b}$ and $W_{\pi(v)}$ in $\Hb_\tau^d\cup\partial_\infty^{\mathbb S}\Hb_\tau^d$ containing $\widetilde b$ and $\pi(v)$ respectively, such that any point in $W_{\widetilde b}$ and any point in $W_{\pi(v)}$ are joined by a spacelike geodesic in $\Hb_\tau^d$. 
    
    Let $w\in\T^1\Hb_\tau^d$ be a  tangent vector to $\ell_{cb}$, such that $\pi(w)\in W_{\widetilde b}$. For each $n\in\Nb$, let $w_n$ be a vector tangent to $\ell_{n,cb}$ such that $(w_n)$ converges to $w$. By taking the tail of the sequences, we may assume that $\pi(w_n)\in W_{\widetilde b}$ and $\pi(v_n)\in W_{\pi(v)}$ for all $n$. Hence, there are spacelike geodesic segments $s_n$ from $\pi(v_n)$ to $\pi(w_n)$ and a spacelike geodesic segment $s$ from $\pi(v)$ to $\pi(w)$. As before, the proper discontinuity of the $\Gamma$-action on $U$ implies that by taking a further subsequence, we may assume that the sequence $(\pi(\gamma_nw_n))$ converges to either $b'$ or $c'$.

    Since $(s_n)$ converges to $s$, the length of $s_n$ is uniformly bounded as we vary $n\in\Nb$. However, since $\pi(\gamma_n v_n)\to a'$ and $(\pi(\gamma_nw_n))$ converges to either $b'$ or $c'$, it follows that the length of $\gamma_ns_n$ grows to infinity. This is a contradiction, because $\gamma_n$ preserves lengths. Hence, the first claim holds.

    (2). We first show that $\Lambda_U$ is a perfect set. If this were not the case, then there exists an isolated point $x\in \Lambda_U$. For any non-empty open subsets $O_1, O_2\subset \Lambda_U\setminus \{x\}$ (these exist because $\Lambda_U$ has at least two points), observe that \[O'_1: = \{x\}\times O_1\quad\text{and}\quad O'_2: = \{x\}\times O_2\] are non-empty open subsets of the set $\Lambda_U^{(2)}$ of distinct pairs of points in $\Lambda_U$. Since the $\Gamma\times\Rb\times \Zb_2$-action on $U$ is topologically transitive, so is the diagonal $\Gamma$-action on $\Lambda_U^{(2)}$. Hence, there exists $\gamma\in \Gamma$ such that $\gamma O'_1 \cap O'_2$ is non-empty, which in turn implies that $\gamma x = x$ and $\gamma O_1\cap O_2$ is non-empty. Since $O_1,O_2$ are arbitrary, it follows that the  $\mathsf{Stab}_{\Gamma}(x)$-action on $\Lambda_U\setminus \{x\}$ is topologically transitive. Since the $\Gamma$-action on $\Lambda_U$ is non-elementary, $\Lambda_U$ is infinite. Hence, $x$ is the attracting point of some collapsing sequence in $\mathsf{Stab}_{\Gamma}(x)$. In particular, $x$ is not isolated, which is a contradiction.

    Now, let $L$ be the limit set of the $\Gamma$-action on $\Lambda_U$ and suppose that $\Omega:=\Lambda_U\setminus L$ is non-empty. Since $\Lambda_U$ is perfect and $\Omega\subset\Lambda_U$ is open, we may find mutually disjoint pre-compact open subsets $V_1,V_2, W_1, W_2\subset \Omega$. Let \[X_1:= V_1\times W_1\quad\text{and}\quad X_2 := V_2\times W_2~,\] and let $R: = \{\gamma\in \Gamma \mid \gamma X_1\cap X_2\not = \emptyset\}$. Notice that $\gamma\in R$ if and only if $\gamma V_1\cap V_2$ and $\gamma W_1\cap W_2$ are both non-empty. We observed in \Cref{sec: convergence groups} that the $\Gamma$-action on $\Omega$ is properly discontinuous, so $R$ is finite. Hence, the fact that $\Lambda_U$ is perfect implies that we can find non-empty open subsets $W'_1\subset W_1$ and $W'_2\subset W_2$ such that for all $\gamma\in R$, $\gamma W'_1\cap W'_2$ is empty. Now, if we set $X'_1:= V_1\times W'_1$ and $X'_2:= V_2\times W'_2$, then note $X_1'$ and $X_2'$ are non-empty open subsets of $\Lambda_U^{(2)}$, but \[\left(\bigcup_{\gamma\in\Gamma}\gamma X'_1\right)\cap X'_2 = \emptyset~,\] which contradicts the topological transitivity of the $\Gamma$-action on $\Lambda_U^{(2)}$. 
\end{proof}

Next, we define thick-thin decompositions for $U$. Recall that for any $x\in\Lambda_U$ and any $C>0$, $H_x(C)$ denotes the horoball of $U$ centered at $x$ of radius $C$, i.e.
\[H_x(C):=\{z\in U:\Bc_x(z)>C\}~.\]
Then for a finite subset $\Pi\subset\Lambda_U$ consisting of points in pairwise distinct $\Gamma$-orbits and a constant $C\in\Rb$, we set
\[U^{\rm thin}(\Pi,C):=\bigcup_{\gamma\in \Gamma,\ p\in \Pi}\gamma H_p(C)\quad\text{and}\quad U^{\rm thick}(\Pi,C):=U\setminus U^{\rm thin}(\Pi,C)~.\]

\begin{definition}\label{definition: thick-thin decomposition geometric}
    Given a finite subset $\Pi$ contained in the limit set of the $\Gamma$-action on $\Lambda_U$ and a constant $C\in\Rb$, we say that the decomposition \[U=U^{\rm thin}(\Pi,C)\cup U^{\rm thick}(\Pi,C)\] is a \emph{$(\Pi,C)$-thick-thin decomposition of $U$} if the following hold, 
    \begin{itemize}
        \item[(I)] the $\Gamma$-invariant collection of closed horoballs \[\{\gamma \overline{H_p(C)}: \gamma \in \Gamma, \ p\in \Pi\}\] is pairwise disjoint and \emph{locally finite}, that is, any compact subset in $U$ intersects finitely many members in the collection.
        \item[(II)] the $\Gamma$-action on $U^{\rm thick}(\Pi,C)$ is cocompact.
    \end{itemize}
    In this case, we refer to $U^{\rm thin}(\Pi,C)$ and $U^{\rm thick}(\Pi,C)$ as the \emph{$(\Pi,C)$-thin part} and \emph{$(\Pi,C)$-thick part} of $U$ respectively.
\end{definition}

\begin{remark}
Later, we will see that if $U$ has a $(\Pi,C)$-thick-thin decomposition, then every point in $\Pi$ is necessarily a parabolic point of $\Lambda_U$, see \Cref{proposition: gf convergence action}.
\end{remark}

\subsection{Weak hulls of transverse groups and good flow spaces}\label{section: geometric finiteness via relative Anosovness} Now we prove the following, which implies the forward directions of all three parts of \Cref{theoremalpha: characterizations}. 

\begin{theorem}\label{theoremalpha: characterizations1}
Let $\Gamma\subset\PGL(d+1,\Rb)$ be a discrete, infinite subgroup.
\begin{enumerate}
    \item If $\Gamma$ is projective transverse and $\Lambda_\Gamma$ contains at least two points, then $\WH(\Lambda_\Gamma)$ is a good flow space for $\Gamma$. Further, if $\Gamma$ is non-elementary, then the $\Gamma\times\Rb\times \Zb_2$-action on $\WH(\Lambda_\Gamma)$ is topologically transitive.
    \item If $\Gamma$ is relatively projective Anosov and $\Lambda_\Gamma$ contains at least two points, then $\WH(\Lambda_\Gamma)$ admits a thick-thin decomposition.
    \item If $\Gamma$ is projective Anosov, then the $\Gamma$-action on $\WH(\Lambda_\Gamma)$ is cocompact.
\end{enumerate} 
\end{theorem}

The bulk of the work in the proof of \Cref{theoremalpha: characterizations1} is the proof of part (2). We first prove parts (1) and (3) of \Cref{theoremalpha: characterizations1} under the assumption that part (2) holds.

\begin{proof}[Proof of \Cref{theoremalpha: characterizations1} parts (1) and (3).]
    We previously proved the first statement of part (1) as \Cref{proposition: domain of discontinuity}. For the second statement of part (1), note that since $\Gamma$ is non-elementary, the $\Gamma$-action on $\Lambda_\Gamma$ is minimal. Then by \Cref{prop: topological transitive} the $\Gamma$-action on $\Lambda_\Gamma^{(2)}$ is topologically transitive, so the same holds true for the $\Gamma\times\Rb\times \Zb_2$-action on $\WH(\Lambda_\Gamma)$.

    By definition, projective Anosov subgroups are the relatively projective Anosov subgroups where $\WH(\Lambda_\Gamma)^{\rm thick}=\WH(\Lambda_\Gamma)$ (and so $\Lambda_\Gamma$ automatically has at least two points). Thus, part (2) implies part (3). 
\end{proof}

We now focus on the proof of part (2) of \Cref{theoremalpha: characterizations1}. As such, we assume for the remainder of this subsection that $\Gamma\subset\PGL(d+1,\Rb)$ is relatively projective Anosov. Choose a representative in each orbit of parabolic points in $\Lambda_\Gamma$, and let $\Pi$ denote the collection of such choices. Then set \[\Pc := \{\Stab_\Gamma(p) : p \in  \Pi\}\] and \[\Pc^\Gamma := \{\gamma P\gamma^{-1}:\gamma\in \Gamma,\ P\in \Pc\} = \{\Stab_\Gamma(p) : p \in  \Gamma \Pi\}~.\] 
For each $p\in\Pi$ and each $C\in\Rb$, let $H_p(C)$ denote the horoball in $\WH(\Lambda_\Gamma)$ centered at $p$ of radius $C$, i.e.
\[H_p(C):=\{z\in\WH(\Lambda_\Gamma):\Bc_p(z) > C\}~.\]

The proof of part (2) of \Cref{theoremalpha: characterizations1} has two main steps. For the first step, we will prove that for sufficiently large $C\in\Rb$, (I) of \Cref{definition: thick-thin decomposition geometric} holds for $\WH(\Lambda_\Gamma)$.

\begin{lemma}\label{proposition: disjointness collection of horoballs}
    There exists $C\in\Rb$ such that 
    \[\mathcal{H} = \{\gamma \overline{H_{p}(C)}: \gamma\in \Gamma,\ p \in \Pi\}~\] is a locally finite collection of mutually disjoint horoballs.
\end{lemma}

For the second step, we prove that for any $C\in\Rb$, (II) of \Cref{definition: thick-thin decomposition geometric} holds for $\WH(\Lambda_\Gamma)$.

\begin{lemma}\label{lemma: cocompactness of thick part of WH}
For all $C\in\Rb$, the $\Gamma$-action on $\WH(\Lambda_\Gamma)^{\rm thick}(\Pi,C)$ is cocompact.
\end{lemma}

This finishes the proof of \Cref{theoremalpha: characterizations1}~(2), so it remains to prove \Cref{proposition: disjointness collection of horoballs} and \Cref{lemma: cocompactness of thick part of WH}.

\subsubsection{The proof of \Cref{proposition: disjointness collection of horoballs}} Our proof is an adaptation of an argument due to Bowditch \cite{bowditch2012relatively}.
As a preliminary step, we prove the following condition for collapsing sequences in $\Gamma$. Recall that $\pi:\T^1\Hb_\tau^d\to\Hb_\tau^d$ is the projection map.

\begin{lemma}\label{collapsing condition}
    Let $(z_n)$ be a sequence in $\WH(\Lambda_\Gamma)$ converging to $z\in\WH(\Lambda_\Gamma)$ and let $(\gamma_n)$ be a sequence in $\Gamma$ such that $\gamma_nz_n^+\to x\in \Lambda_\Gamma$, $\gamma_nz_n^-\to y\in \Lambda_\Gamma$ as $n\to\infty$ and $x\ne y$. If $\gamma_n \pi(z_n)\to x$ as $n\to\infty$, then $(\gamma_n)$ is a collapsing sequence with $x$ its attracting point and $z^-$ its repelling point.
\end{lemma}

\begin{proof}
    We prove the lemma by showing that every subsequence of $(\gamma_n)$ admits a further collapsing subsequence that has attracting point $x$ and  repelling point $z^-$.

    The assumption that the sequences $(\gamma_n z_n)$ and $(\gamma_n z_n^+)$ both converge to $x$ implies that $(\gamma_n)$ is an escaping sequence. Since $\Gamma$ is projective transverse, by \Cref{lemma: projective north-south dynamics}, we assume that up to subsequence $(\gamma_n)$ is a collapsing sequence with attracting point $a^+\in \Fc_{1,d}$ and repelling point $a^-\in \Fc_{1,d}$. Since $\gamma_nz_n^+\to x$ and $\gamma_nz_n^-\to y$ as $n\to\infty$ and $x\ne y$, then either $(a^+,a^-)=(x,z^-)$ or $(a^+,a^-)=(y,z^+)$. Since $z_n\to z$ as $n\to \infty$, we may assume that the sequence $(\pi(z_n))$ is uniformly transverse to both $z^+$ and $z^-$ by \Cref{observation: a spacelike geodesic is totally transverse} part (1). Since $\gamma_n \pi(z_n)\to x$ as $n\to\infty$, it follows that $(\gamma_n)$ has attracting point $x$ and  repelling point $z^-$.
\end{proof}

We then prove that $\Gamma$-orbits of a point in $\WH(\Lambda_\Gamma)$ cannot get arbitrarily deep into the horoballs.

\begin{lemma}\label{lemma: Busemann upper bounds on orbits}
    For any compact subset $K\subset \WH(\Lambda_\Gamma)$ and $p\in \Pi$, there exists $C\in\Rb$ such that $\Bc_p(\gamma z) \leqslant C$ for any $\gamma\in \Gamma$ and $z\in K$.
\end{lemma}

\begin{proof}
    Suppose that there exists a sequence $(\eta_n)$ in $\Gamma$ and a sequence $(z_n)$ in $K$ such that $\Bc_p(\eta_n z_n)\to +\infty$ as $n\to \infty$. Up to subsequence, we may assume $z_n\to z\in K$ as $n\to \infty$. We show that $p$ is conical, which contradicts that every parabolic point is not conical (\cite[Proposition~6.1]{bowditch2012relatively}).

    Since $p$ is a bounded parabolic point, up to subsequence and negation if necessary, we assume there is a sequence $(\gamma_n)$ in $\Stab_\Gamma(p)$ such that $\gamma_n \eta_n z_n^-\to q \in\Lambda_\Gamma\setminus \{p\}$ as $n\to \infty$. By \Cref{proposition: parabolic subgroups are weakly unipotent} and \Cref{unipotents preserve horofunctions}, \[\Bc_p(\gamma_n \eta_n z_n)= \Bc_p(\eta_n z_n)\to +\infty\quad\text{as}\quad n\to\infty~,\] then by \Cref{proposition: two of three}, both $(\gamma_n \eta_n \pi( z_n))$ and  $(\gamma_n \eta_n z_n^+)$ converge to $p$. Hence, by \Cref{collapsing condition}, $(\gamma_n\eta_n)$  is collapsing with attracting point $p$ and repelling point $z^-$. Then $((\gamma_n\eta_n)^{-1})$ is a collapsing sequence with attracting point $z^-$ and repelling point $p$.

{\bf Case 1.} If the sequence $(\gamma_n\eta_nz_n^+)$ admits a constant subsequence, then by taking a subsequence, we may assume that $\gamma_n\eta_nz_n^+ = p$ for all $n\in\Nb$. Hence,
\[(\gamma_n\eta_n)^{-1} p=(\gamma_n \eta_n)^{-1} (\gamma_n\eta_n z_n^+) = z_n^+ \to z^+\neq z^- \quad \text{as} \quad n\to\infty~.\]
Then $p$ is conical by definition.

{\bf Case 2.} If the sequence $(\gamma_n\eta_nz_n^+)$ admits no constant subsequence, then up to subsequence, we may assume that $(\gamma_n \eta_n z_n^+)$ has pairwise distinct terms which are all distinct from $p$. Again as $p$ is a bounded parabolic point,  there exists a sequence $(\omega_n)$ in $\Stab_\Gamma(p)$ such that $(\omega_n \gamma_n \eta_n z_n^+)$ converges to some $q'\in\Lambda_\Gamma\setminus\{ p\}$. By \Cref{proposition: parabolic subgroups are weakly unipotent} and \Cref{unipotents preserve horofunctions}, $\Bc_p(\omega_n\gamma_n \eta_n z_n) = \Bc_p(\eta_n z_n)\to +\infty$ as $n\to\infty$, and by \Cref{proposition: two of three}, both sequences $(\omega_n\gamma_n \eta_n \pi(z_n))$ and $(\omega_n\gamma_n \eta_n z_n^-)$ converge to $p$. Then by \Cref{collapsing condition}, the attracting point and repelling point of $\omega_n\gamma_n\eta_n$ are $p$ and $z^+$ respectively.

    Up to further subsequence, $((\gamma_n\eta_n)^{-1} p)$ converges to some $b\in \Lambda_\Gamma$. Then both of the following hold,
    \begin{enumerate}
        \item $(\gamma_n\eta_n)^{-1} p \to b$ and $(\gamma_n\eta_n)^{-1} x \to z^-$ for any $x\in \Lambda_\Gamma \setminus \{p\}$ as $n\to \infty$.
        \item $(\omega_n\gamma_n\eta_n)^{-1} p  = (\gamma_n\eta_n)^{-1}\omega_n^{-1} p = (\gamma_n\eta_n)^{-1} p \to b$ and $(\omega_n \gamma_n\eta_n)^{-1} x \to z^+$ for any $x\in \Lambda_\Gamma \setminus \{p\}$ as $n\to \infty$.
    \end{enumerate}
    As $z^+\ne z^-$, either $b\ne z^-$ or $b\ne z^+$. This proves that $p$ is conical.
\end{proof}

Next, we prove that the stabilizer of any parabolic point $p\in\Lambda_\Gamma$ acts cocompactly on the thickened horospheres centered at $p$.

\begin{lemma}\label{lemma: cocompactness of horospheres}
    For any $p \in \Pi$ and real constants $C'\leqslant C$, the $\Stab_\Gamma(p)$-action on $\{z \in \WH(\Lambda_\Gamma): C' \leqslant \Bc_p(z) \leqslant C\}$ is cocompact.
\end{lemma}

\begin{proof}
    Let $(z_n)$ be a sequence in $\{z \in \WH(\Lambda_\Gamma): C'\leqslant \Bc_p(z) \leqslant C\}$. Since the $\Stab_\Gamma(p)$-action on $\Lambda_\Gamma\setminus \{p\}$ is cocompact, by applying \Cref{prop: unimodal}, there exists a sequence $(\gamma_n)$ in $\Stab_\Gamma(p)$, such that up to subsequence and negation, $\gamma_n z_n^-\to b \in\Lambda_\Gamma\setminus\{p\}$ as $n\to\infty$ and $t\mapsto\Bc_p(\varphi^t(\gamma_nz_n))$ is increasing on $(-\infty,0]$. We show that $(\gamma_n z_n)$ admits a subsequence that converges in $\WH(\Lambda_\Gamma)$. 

    Up to subsequence, we assume that $\gamma_n z_n^+ \to c\in\Lambda_\Gamma$ and $\pi(\gamma_nz_n)\to a\in\Hb_\tau^d\cup\partial_\infty^{\Pb}\Hb_\tau^d$ as $n\to \infty$. Note that for all $n\in\Nb$, $\Bc^2_p(z_n) \geqslant \Bc_p(z_n) \geqslant C'$ by \eqref{horofunction component bound}. This implies that $c \ne b$, otherwise \Cref{proposition: parabolic subgroups are weakly unipotent} and \Cref{unipotents preserve horofunctions} imply that 
    \[\mathcal{B}^2_p (z_n)=\mathcal{B}^2_p (\gamma_n z_n)\to\frac{1}{4}\log \frac{\Delta(b,b)}{\Delta(p,b)^2}=-\infty~.\]
    Hence one of the following holds,
    \begin{enumerate}
        \item $a=b$;
        \item $a=c \ne p$;
        \item $a=c = p$;
        \item $a$ lies along a spacelike geodesic between $b$ and $c$.
    \end{enumerate}
    It now suffices to rule out Cases (1), (2), and (3), as Case (4) implies that the sequence $(\gamma_n z_n)$ converges in $\WH(\Lambda_\Gamma)$.
        
    In Cases (1) and (2), observe that \Cref{proposition: parabolic subgroups are weakly unipotent} and \Cref{unipotents preserve horofunctions} give 
    \[\mathcal{B}^0_p(z_n)=\mathcal{B}^0_p(\gamma_nz_n) \to\frac{1}{2}\log \frac{\sqrt{\Delta(a,a)}}{\Delta(p,a)}=-\infty~.\] However, \eqref{horofunction component bound} and \eqref{horofunction component bound2} imply that for all $n\in\Nb$, $\Bc^0_p(z_n)\geqslant\Bc^1_p(z_n)\geqslant \Bc_p(z_n)\geqslant C'$, which is a contradiction. Case (3) is also impossible by \Cref{bounded converges to unbounded}.
\end{proof}

For each $p\in\Pi$, \Cref{proposition: parabolic subgroups are weakly unipotent} and \Cref{unipotents preserve horofunctions} imply that $\Bc_p$ is invariant under $\Stab_\Gamma(p)$, and so for each $q\in\Gamma p$, we set
\[\wh \Bc_q(\cdot) := \Bc_p (\gamma^{-1}\cdot)\quad\text{and}\quad \wh H_q(C) := \gamma H_p(C)=\{z\in\WH(\Lambda_\Gamma): \wh\Bc_q(z)\geqslant C\}~,\]
where $\gamma$ is some (any) element in $\Gamma$ such that $\gamma p=q$. Combining \Cref{lemma: Busemann upper bounds on orbits} and \Cref{lemma: cocompactness of horospheres}, we prove that for any $\Gamma$-invariant collection of horoballs centered at the parabolic points in $\Lambda_\Gamma$ and any compact subset in $\WH(\Lambda_\Gamma)$, the number of horoballs in this collection that intersect a compact subset is finite.

\begin{lemma}\label{lemma: mostly negatively infinite}
    For any compact subset $K \subset \WH(\Lambda_\Gamma)$ and $C'\in \Rb$, the set \[\{q \in \Gamma \Pi:  \sup_{z\in K}\wh \Bc_q(z)\geqslant C'\}\] is finite.
\end{lemma}

\begin{proof}
    Since $\Pi$ is finite, it suffices to prove that for each $p\in \Pi$ the set \[A:=\{q\in\Gamma p : \sup_{z\in K} \wh \Bc_q(z)\geqslant C'\}\] is finite. By \Cref{lemma: Busemann upper bounds on orbits}, there exists $C\in \Rb$ such that $\sup_{z\in K}\Bc_p(\gamma^{-1}z)\leqslant C$ for all $\gamma\in \Gamma$. As such, if $C'> C$, then $A$ is empty, and if $C'\leqslant C$, then \[\# A=\#\{\gamma\Stab_\Gamma(p)\in\Gamma/\Stab_\Gamma(p) : C'\leqslant \sup_{z\in K}\Bc_p(\gamma^{-1}z)\leqslant C\}~,\] where $\#$ denotes cardinality. By \Cref{lemma: cocompactness of horospheres}, there is a compact set $Q \subset \WH(\Lambda_\Gamma)$ such that \[\Stab_\Gamma(p) Q = \{z \in \WH(\Lambda_\Gamma): C' \leqslant \Bc_p(z) \leqslant C\}~.\] Hence, \[\#A\leqslant \#\{\gamma\in\Gamma: \gamma^{-1}K\cap Q\not = \emptyset\}~,\] which is a finite set because the $\Gamma$-action on $\WH(\Lambda_\Gamma)$ is properly discontinuous.
\end{proof}

Finally, we deduce \Cref{proposition: disjointness collection of horoballs} from \Cref{lemma: mostly negatively infinite}.

\begin{proof}[Proof of \Cref{proposition: disjointness collection of horoballs}]
    We prove the disjointness. Once this is done, local finiteness follows directly from \Cref{lemma: mostly negatively infinite}. 
    
    Suppose otherwise. Then there exist a sequence $(h_n)$ in $\Rb$ that grows to $\infty$, sequences $(p_n)$ and $(q_n)$ in $\Pi$, and sequences $(\gamma_n)$ and $(\eta_n)$ in $\Gamma$ such that for all $n$, $\gamma_n p_n\neq\eta_n q_n$ and $\wh H_{\gamma_n p_n}(h_n)\cap\wh H_{\eta_n q_n}(h_n)$ is non-empty. We may assume up to subsequence that there is some $p,q\in\Pi$ such that for all $n\in\Nb$, $p_n = p$ and $q_n =q$, in which case $\gamma_n^{-1}\eta_n q\ne p$. Since the $\Stab_\Gamma(p)$-action on $\Lambda_\Gamma \setminus \{p\}$ is cocompact, there exists a sequence $(\omega_n)$ in $\Stab_\Gamma(p)$ such that up to a further subsequence, the sequence $(\omega_n \gamma_n^{-1}\eta_n q)$ converges to some $q'\in\Lambda_\Gamma\setminus\{p\}$. For all $n\in\Nb$, set 
    \[\alpha_n := \omega_n \gamma_n^{-1}\eta_n~,\] 
    and we have that $\widehat H_p(h_n)\cap\widehat H_{\alpha_nq}(h_n)=\omega_n\gamma_n^{-1}\big(\widehat H_{\gamma_np}(h_n)\cap\widehat H_{\eta_nq}(h_n)\big)$ is non-empty.
    
    If $(\alpha_n q)$ admits a constant subsequence where every element is $q'$, then $q' = \gamma q$ for some $\gamma\in \Gamma$. Hence, for all $n\in\Nb$, $H_p(h_n) = \wh H_p(h_n)$ intersects $\wh H_{q'}(h_n) = H_{q'}(h_n+ k)$ for some constant $k$ that does not depend on $n$, see \Cref{constant difference}. In particular, for all $n\in\Nb$, $H_p(\min\{h_n,h_n+k\})$ and $H_{q'}(\min\{h_n,h_n+k\})$ intersect, which contradicts \Cref{lemma: disjointness of two horoballs}. 
    
    Thus, by taking a subsequence, we may assume that sequence $(\alpha_n q)$ has mutually distinct terms, and that $\alpha_nq\to q'\in\Lambda_\Gamma$ as $n\to\infty$ Again by \Cref{constant difference}, for each $n\in\Nb$, there exists $k_n\in \Rb$ such that for all $z\in\WH(\Lambda_\Gamma)$,
    \[\Bc_{\alpha_n q}(z) - \wh \Bc_{\alpha_nq}(z) = k_n~.\] 
    
    We now prove that $k_n \geqslant 0$ for all but finitely many $n$. Suppose that this were not the case. Then by taking a subsequence, we may assume that $k_n<0$ for all $n\in\Nb$. Pick any $z\in\WH(\Lambda_\Gamma)$. Since $\Bc_{\alpha_n q}(z)$ converges to $\Bc_{q'}(z)$, it follows that as one varies over all $n\in\Nb$, $\Bc_{\alpha_n q}(z)$ has a uniform lower bound, and so the same holds for \[\wh \Bc_{\alpha_nq}=\Bc_{\alpha_n q}(z) -k_n~.\] This contradicts \Cref{lemma: mostly negatively infinite} since $(\alpha_n q)$ is an infinite sequence.
    
    By taking a subsequence, we may now assume that for all $n\in\Nb$, we have $k_n\geqslant 0$, and so 
    \[\wh H_{\alpha_n q}(h_n) = H_{\alpha_n q}(h_n+k_n) \subset H_{\alpha_n q}(h_n)~.\] 
    This implies that for all $n\in\Nb$, $H_p(h_n)=\widehat{H}_p(h_n)$ and $H_{\alpha_nq}(h_n)$ intersect, which again contradicts \Cref{lemma: disjointness of two horoballs}. 
\end{proof}

\subsubsection{The proof of \Cref{lemma: cocompactness of thick part of WH}}

Let $d_\angle$ denote the angle metrics on $\mathbb{RP}^d$ and $\mathbb{RP}^{d*}$ induced by the standard inner products on $\Rb^{d+1}$ and $\Rb^{d+1*}$. Using the identification  $\Hb_\tau^d\cup\partial_\infty^{\Pb}\Hb_\tau^d\cong\mathbb{RP}^d\times\mathbb{RP}^{d*}$, define $D:\WH(\Lambda_\Gamma) \to \Rb^+$ to be the function given by \[D(z)=\min\left\{d_\angle(\pi(y)^a,(y^\dagger)^a):y\in\{z,{\opp}(z)\},\,\dagger\in\{+,-\},\,a\in\{1,d\}\right\}~.\] Then for any $\epsilon>0$, let \[K(\epsilon):=\{v\in \WH(\Lambda_\Gamma) : D(v)\geqslant \epsilon \}~.\] 
We first prove the following compactness result for $K(\epsilon)$.

\begin{lemma}\label{lemma: K compact}
For any $\epsilon>0$, $K(\epsilon)\subset \WH(\Lambda_\Gamma)$ is compact. 
\end{lemma}

\begin{proof}
    Let $(z_n)$ be an arbitrary sequence in $ K(\epsilon)$. We show it admits a convergent subsequence. We assume up to subsequence that there is some $z^+,z^-\in\Lambda_\Gamma$ such that $z_n^+\to z^+$ and $ z_n^-\to z^-$ as $n\to\infty$. 
     
    We first prove that $z^+\ne z^-$. Suppose otherwise. Let $S_n$ and $S_n^{\opp}$ denote the connected components of $\mathbb{RP}^d\setminus((z_n^+)^{d}\cup (z_n^-)^{d})$ such that $\pi(z_n)^1\in S_n$ and $\pi({\opp}(z_n))^1\in S_n^{\opp}$. Notice that $S_n\neq S_n^{\opp}$ and both $\pi(z_n)^{d}\cap S_n$ and $\pi({\opp}(z_n))^{d}\cap S_n^{\opp}$ are empty. Since the sequences $((z_n^+)^{d})$ and $((z_n^-)^{d})$ both converge to \[H:=(z^+)^{d}=(z^-)^{d}~,\] either the sequence $(S_n)$ or the sequence $(S_n^{\opp})$ necessarily converges to $\mathbb{RP}^d\setminus H$, so either the sequence $(\pi(z_n)^{d})$ or the sequence $(\pi({\opp}(z_n))^{d})$) converges to $H$. This contradicts the fact that $z_n\in K(\epsilon)$ for all $n\in\Nb$. Hence $z^+\neq z^-$.
    
    The transversality of $\Gamma$ implies that $z^+$ and $z^-$ are transverse. Since $z_n^\pm\to z^\pm$ as $n\to\infty$, there is a sequence $(g_n)$ in $\PGL(d+1,\Rb)$ that converges to some $g\in\PGL(d+1,\Rb)$, such that \[g_n\cdot z_1^\pm=z_n^\pm\to z^\pm=g\cdot z_1^\pm~.\] For each $n\in\Nb$, let $\ell_n$ be a geodesic tangent to $z_n$. Let $\ell:=g\ell_1$, and note that $z^+$ and $z^-$ are the forward and backward endpoints in $\partial_\infty^{\Pb}\Hb_\tau^d$ of $\ell$. By taking a subsequence, we may ensure that $\ell_n(\Rb)=g_n\ell_1(\Rb)\to g\ell_1(\Rb)=\ell(\Rb)$ as $n\to\infty$. Since $z_n\in K(\epsilon)$ for all $n\in\Nb$, the sequence $(z_n)$ converges to some $z\in \T^1\Hb_\tau^d$ that is tangent to $\ell$ and lies in $K(\epsilon)$.
\end{proof}
 
\begin{remark}
    Notice that unlike the other lemmas used in the proof of part (2) of \Cref{theoremalpha: characterizations1}, \Cref{lemma: K compact} holds not only for relatively projective Anosov subgroups of $\PGL(d+1,\Rb)$, but also for all projective transverse subgroups of $\PGL(d+1,\Rb)$. 
\end{remark}

Using \Cref{lemma: K compact}, we now prove \Cref{lemma: cocompactness of thick part of WH}.

\begin{proof}[Proof of \Cref{lemma: cocompactness of thick part of WH}]
    We first observe that for any $z\in\WH(\Lambda_\Gamma)$, there is a point $z'\in\Gamma\cdot z$ such that $D(z')=\sup D(\Gamma \cdot z)$. Choose any positive $\epsilon<D(z)$. By \Cref{lemma: K compact}, $K(\epsilon)$ is compact. Hence \Cref{proposition: domain of discontinuity} implies that $K(\epsilon)\cap\Gamma\cdot z$ is finite (and non-empty), and thus contains a point $z'$ such that \[D(z')=\max\{D(\eta z):\eta\in\Gamma\text{ and }\eta z\in K(\epsilon)\}=\sup D(\Gamma\cdot z)~,\] where the last equality follows from the definition of $K(\epsilon)$. 
    
    Observe that the set 
    \[\Sc = \{z\in \WH(\Lambda_\Gamma)^{\mathrm{thick}} : D(z)\geqslant D(\gamma\cdot z)\text{ for all }\gamma\in\Gamma \}\]
    has the property that $\bigcup_{\gamma\in\Gamma}\gamma\Sc=\WH(\Lambda_\Gamma)^{\mathrm{thick}}$. Hence, by \Cref{lemma: K compact}, it suffices to show there exists $\epsilon>0$ such that $\Sc\subset K(\epsilon)$. We argue by contradiction. Suppose that there exists a sequence $(z_n)$ in $\Sc$ such that $D(z_n)\to 0$ as $n\to\infty$. Up to dualizing, negating, and taking a subsequence, we may assume that $z_n^+ \to p \in \Lambda_{\Gamma}$ and $z_n^- \to q \in \Lambda_{\Gamma}$ and $\pi(z_n)^1\to p^1$ as $n\to \infty$. We proceed the proof in two cases.
    
{\bf Case 1. $p$ is conical.} By \Cref{lemma: projective expanding}~(1), there exists 
\begin{itemize}
    \item an open neighborhood $O_{1}$ in $\mathbb{RP}^d$ of $p^1$,
    \item an open neighborhood $O_{d}$ in $\mathbb{RP}^{d*}$ of $p^d$,
    \item an element $\gamma\in\Gamma$, and
    \item a constant $C>1$
\end{itemize} 
such that for all  $x^1,y^1\in O_{1}$ and $x^d,y^d\in O_{d}$, we have
\[d_\angle(\gamma x^1, \gamma y^1) \geqslant C d_\angle(x^1, y^1)\quad\text{and}\quad d_\angle(\gamma x^d, \gamma y^d) \geqslant C d_\angle(x^d, y^d)~.\] We derive a contradiction by proving that for sufficiently large $n$, \[D(\gamma z_n)\geqslant CD(z_n)~.\] We split the proof into two further cases depending on whether $p$ and $q$ agree.

{\bf Case 1.1. $p=q$.} Let $C_0>0$ such that $\gamma$ acts as a $C_0$-biLipschitz automorphism on both $\mathbb{RP}^d$ and $\mathbb{RP}^{d*}$ (with respect to $d_\angle$). Let $O_1'\subset\mathbb{RP}^d$ and $O_d'\subset\mathbb{RP}^{d*}$ be open neighborhoods of $p^1=q^1$ and $p^d=q^d$ respectively, such that $\overline{O_1'}\subset O_1$ and $\overline{O_d'}\subset O_d$. Then there exists a constant $M>0$ such that for all $x^1\in\mathbb{RP}^d\setminus O_1$ and $x^d\in\mathbb{RP}^{d*}\setminus O_d$, we have $d_\angle(x^1,O_1')\geqslant M$ and $d_\angle(x^d,O_d')\geqslant M$. 
    
Pick $n$ large enough so that $\big( (z_n^{\pm})^1, (z_n^{\pm})^d\big)\in O_1'\times O_d'$ and $D(z_n) \leqslant \frac{M}{C_0C}$. To prove that $D(\gamma z_n)\geqslant CD(z_n)$, we verify by definition that for all $y_n\in\{z_n,{\opp}(z_n)\}$, $\dagger\in\{+,-\}$, and $a\in\{1,d\}$, we have
\[d_\angle(\pi(\gamma y_n)^a,(\gamma y_n^\dagger)^a)\geqslant C D(z_n)~.\]
If $\pi(y_n)^a\in O_a$, then \[ d_\angle(\pi(\gamma y_n)^a,(\gamma y_n^\dagger)^a)\geqslant C d_\angle(\pi(y_n)^a,(y_n^\dagger)^a) \geqslant C D(z_n)~.\]
Otherwise, $d_\angle(\pi(y_n)^a, (y_n^\dagger)^a)\geqslant M$ and hence \[d_\angle (\pi(\gamma y_n)^a,(\gamma y_n^\dagger)^a)\geqslant \frac{M}{C_0} \geqslant C D(z_n)~.\] 

{\bf Case 1.2. $p\ne q$.} For all $n\in\Nb$, let $\ell_n$ be the unit speed spacelike geodesic tangent to $z_n$. Up to subsequence, we assume that $(\ell_n(\Rb))$ converges to $\ell(\Rb)$, where $\ell$ is a unit speed spacelike geodesic with forward endpoint $p$ and backward endpoint $q$ in $\partial_\infty^{\Pb}\Hb_\tau^d$. The assumption $\pi(z_n)^1\to p^1$ implies that $(\pi(z_n))$ converges to $p$. 

We again verify that for sufficiently large $n$, all $y_n\in\{z_n,{\opp}(z_n)\}$, $\dagger\in\{+,-\}$, and $a\in\{1,d\}$, we have \[d_\angle(\pi(\gamma y_n)^a,(\gamma y_n^\dagger)^a)\geqslant C D(z_n)~.\]

To get the required inequality when $y_n\in\{z_n,{\opp}(z_n)\}$, $a\in\{1,d\}$ and $\dagger = -$, let $C_0>0$ such that $\gamma$ acts as $C_0$-bilipschitz automorphisms on $\mathbb{RP}^d$ and $\mathbb{RP}^{d*}$. Choose $n$ large enough so that    
\[d_\angle(\pi(y_n)^a,(y_n^-)^a)\geqslant \frac{1}{2} \min\{d_\angle(p^1,q^1), d_\angle(p^d,q^d)\}=: M_0~,\]
and $D(y_n)\leqslant \frac{M_0}{C_0C}$. Then
\[d_\angle(\pi(\gamma y_n)^a,(\gamma y_n^-)^a)\geqslant \frac{M_0}{C_0}\geqslant CD(z_n)~.\] For the case when $y_n\in\{z_n,{\opp}(z_n)\}$, $a\in\{1,d\}$ and $\dagger = +$, choose $n$ large enough so that $\pi(y_n)^a\in O_{a}$. Then \[d_\angle(\pi(\gamma y_n)^a,(\gamma y_n^+)^a)\geqslant C d_\angle(\pi(y_n)^a,(y_n^+)^a)\geqslant CD(z_n)~.\] 
This completes the proof for Case 1.

{\bf Case 2. $p$ is parabolic.}  \Cref{prop: unimodal} implies that for each $n\in\Nb$, the map $t\mapsto\Bc_p(\varphi^t(z_n))$ is either increasing on $(-\infty,0]$ or decreasing on $[0,+\infty)$. Since we assumed $\pi(z_n)^1\to p^1$ as $n\to\infty$, we have the following two possibilities.
\begin{itemize}
	\item $p = q$, in which case the role between $p$ and $q$ are equivalent, thus by applying negation if necessary, we can assume $t\mapsto\Bc_p(\varphi^t(z_n))$ is decreasing on $[0,\infty)$ without breaking the assumption that $\pi(z_n)^1\to p^1$ as $n\to\infty$.
	\item $p \ne q$, in which case $z_n\not \in \wh H_p(C)$, then for sufficiently large $n$, $t\mapsto\Bc_p(\varphi^t(z_n))$ is decreasing on $[0,\infty)$. Otherwise, there exists an infinite sequence $n_k \in \mathbb \Nb$ such that $t\mapsto\Bc_p(\varphi^{t}(z_{n_k}))$ is strictly increasing on $(-\infty,0]$. However \Cref{bounded converges to unbounded} implies that no subsequence of $\pi(z_{n_k} )$ converges to $p$, which gives a contradiction.
\end{itemize}

Thus, for each $n\in\Nb$, we assume the map $t\mapsto\Bc_p(\varphi^t(z_n))$ is decreasing on $[0,\infty)$, which implies that $z_n^+\neq p$. Since $p$ is a bounded parabolic point, up to subsequence, there is a sequence $(\gamma_n)$ in $\Stab_\Gamma(p)$ such that \[\gamma_nz_n^+\to b\in\Lambda_\Gamma\setminus\{p\}\quad\text{and}\quad\gamma_nz_n^-\to c\in\Lambda_\Gamma~.\]
Since $z_n^+\to p$, $\gamma_n$ is an escaping sequence. Let $y_n\in \{\gamma_n z_n,\opp(\gamma_n z_n)\}$, $\dagger_n\in\{+,-\}$ and $a_n\in \{1,d\}$ be sequences satisfying
\[D(\gamma_n z_n) = d_{\angle}(\pi(y_n)^{a_n}, (y_n^{\dagger_n})^{a_n})~.\]

The proof now proceeds in two cases, depending on whether $c$ and $p$ agree. In both cases, we will find compact sets $K^1\subset\mathbb{RP}^d$ transverse to $p^d$ and $K^d\subset\mathbb{RP}^{d*}$ transverse to $p^1$ such that for sufficiently large $n$, both $\pi(y_n)^{a_n}$ and $(y_n^{\dagger_n})^{a_n}$ lie in $K^1$ or $K^d$

Indeed, once we do so, then we may apply \Cref{lemma: projective expanding}~(2) to deduce that for sufficiently large $n$,
\begin{align*}
    D(z_n)&\leqslant  d_\angle(\gamma_n^{-1}\pi(y_n)^{a_n}, \gamma_n^{-1}(y_n^{\dagger_n})^{a_n})\\
    &< d_\angle(\pi( y_n)^{a_n}, ( y_n^{\dagger_n})^{a_n})\\
    &= D(y_n)~.
\end{align*}
This gives a contradiction since $D(\gamma_nz_n)>D(z_n)$, but $D(z_n)\geqslant D(\gamma_n z_n)$ as $z_n\in\mathcal S$.

{\bf Case 2.1. $c\ne p$.} In this case there is an open subset $U^1\subset\mathbb{RP}^d$ that contains both $b^1$ and $c^1$ and an open subset $U^d\subset\mathbb{RP}^{d*}$ that contains both $b^d$ and $c^d$, such that all the points in $\overline{U^1}$ are transverse to $p^d$ and all the points in $\overline{U^d}$ are transverse to $p^1$. So for sufficiently large $n$, both $\pi(y_n)^{a_n}$ and $(y_n^{\dagger_n})^{a_n}$ lie in $U^{a_n}$. $K^1:=\overline{U^1}$, $K^d:=\overline{U^d}$ are the compact subsets required above.

{\bf Case 2.2. $c= p$.} 
In this case $b\neq c$, $\Bc_p(y_n)=\Bc_p(z_n)$ has a uniform upper bound as we vary $n$, and the map $t\mapsto\Bc_p(\varphi^t(y_n))=\Bc_p(\varphi^t(z_n))$ is decreasing on $[0,+\infty)$. On the one hand by \Cref{bounded converges to unbounded}, $(\pi(y_n))$ admits no subsequence converging to $p$. On the other hand, since $D(y_n)\leqslant D(z_n)\to 0$ as $n\to \infty$, we have $\dagger_n = +$ for large enough $n$ and $\pi(y_n)\to b$ as $n\to\infty$.

Let $U^1\subset\mathbb{RP}^d$, $U^d\subset\mathbb{RP}^{d*}$ be neighborhoods of $b^1,b^d$ respectively, such that all the points in $\overline{U^1}$ are transverse to $p^d$ and all the points in $\overline{U^d}$ are transverse to $p^1$. Then for sufficiently large $n$, both $\pi(\gamma_n z_n)^{a_n}$ and $(\gamma_n z_n^+)^{a_n}$ lie in $U^{a_n}$. Hence $K^1:=\overline{U^1}$ and $K^d:=\overline{U^d}$ are the compact subsets required above.
\end{proof}

\subsection{Existence of good flow spaces implies group is transverse}\label{section: existence of good flow spaces implies group is transverse}
Let $\Gamma\subset\PGL(d+1,\Rb)$ be a discrete, infinite subgroup for which there is a good flow space $U\subset\T^1\Hb_\tau^d$ for $\Gamma$. We now prove the backward direction of all three parts of \Cref{theoremalpha: characterizations}, i.e., we prove the following proposition.

\begin{proposition}\label{criterion}Suppose that $\Gamma$ is irreducible.
\begin{enumerate}
        \item If the $\Gamma\times\Rb\times \Zb_2$-action on $U$ is topologically transitive, then $\Gamma$ is a projective transverse subgroup. 
        \item If $U$ admits a thick-thin decomposition, then $\Gamma$ is relatively projective Anosov.
        \item If the $\Gamma$-action on $U$ is cocompact, then $\Gamma$ is projective Anosov. 
    \end{enumerate}
    In all cases, $U=\WH(\Lambda_\Gamma)$.
\end{proposition}

To prove \Cref{criterion}, we will use the pair of lemmas. The first gives conditions under which $\Gamma$ is projective transverse.

\begin{lemma}\label{strongly dynamics preserving}
    Suppose that $\Gamma$ is irreducible. If there is a closed, non-empty, $\Gamma$-invariant, perfect, transverse subset $\Lambda\subset\Fc_{1,d}$ on which $\Gamma$ acts as a convergence group, then $\Gamma$ is projective transverse and $\Lambda_\Gamma\subset\Lambda$.  
\end{lemma}

The second ensures the geometric finiteness of the $\Gamma$-action on $\Lambda_U$ if $U$ admits a $(\Pi,C)$-thick-thin decomposition.

\begin{lemma}\label{proposition: gf convergence action}
    If $U$ admits a $(\Pi,C)$-thick-thin decomposition, then $\Gamma$ acts on $\Lambda_U$ as a geometrically finite convergence group, and $\Gamma\Pi$ is the set of parabolic points.
\end{lemma}

Assuming these two lemmas, we prove \Cref{criterion}.

\begin{proof}[Proof of \Cref{criterion}]
(1). Notice that since $\Gamma$ is irreducible, the $\Gamma$-action on $\Lambda_U$ is non-elementary. Then \Cref{proposition: convergence action} implies that $\Gamma$ acts on $\Lambda_U$ as a minimal convergence group. In particular, $\Lambda_U$ is a perfect set. Thus, we may apply \Cref{strongly dynamics preserving} with $\Lambda=\Lambda_U$ to deduce that $\Gamma$ is projective transverse and $\Lambda_\Gamma=\Lambda_U$.

(2). If $U$ admits a thick-thin decomposition, then \Cref{proposition: gf convergence action} implies that $\Gamma$ acts on $\Lambda_U$ as a geometrically finite convergence group, so the limit set of this $\Gamma$-action is all of $\Lambda_U$. The irreducibility of $\Gamma$ implies that this $\Gamma$-action is non-elementary. Hence, $\Gamma$ acts minimally on $\Lambda_U$, and $\Lambda_U$ is perfect.   By \Cref{strongly dynamics preserving}, $\Lambda_\Gamma=\Lambda_U$. Therefore, $\Gamma$ is relatively projective Anosov.

(3). This follows from the same argument as (2), with the additional observation that if the thin part of $U$ is empty, then \Cref{proposition: gf convergence action} implies that $\Lambda_\Gamma$ has no parabolic points.
\end{proof}

It remains to prove \Cref{strongly dynamics preserving,proposition: gf convergence action}. We will start with \Cref{strongly dynamics preserving}.

\begin{proof}[Proof of \Cref{strongly dynamics preserving}]
We first prove that $\Gamma$ is projective divergent. Suppose otherwise that there is a constant $C\geqslant 1$ and an escaping sequence $(\gamma_n)$ in $\Gamma$ such that $1\leqslant \frac{\sigma_1(\gamma_n)}{\sigma_2(\gamma_n)}\leqslant C$. For each $n\in\Nb$, let $\gamma_n=k_n a(\gamma_n) l_n$ be a singular value decomposition of $\gamma_n$, where $k_n, l_n\in\PO(d+1)$ and $a(\gamma_n)=\diag(\sigma_1(\gamma_n),\dots,\sigma_{d+1}(\gamma_n))$. We assume up to subsequence that $k_n\to k$ and $l_n\to l$ as $n\to \infty$.

Set $V_1:=l^{-1}\Span(e_1,e_2)$, $V_2:=l^{-1}\Span(e_3,\dots,e_{d+1})$, $W_1:=k\Span(e_1,e_2)$, and $W_2:=k\Span(e_3,\dots,e_{d+1})$ where $(e_1,\dots,e_{d+1})$ is the standard basis of $\Rb^{d+1}$. Let
\[p_V:\Rb^{d+1}\to V_1\quad\text{and}\quad p_W:\Rb^{d+1}\to W_1\] be the projections whose kernels are $V_2$ and $W_2$ respectively. Observe that $p_V$ and $p_W$ restrict and descend to maps 
\[p_V:\mathbb{RP}^d\setminus \Pb(V_2)\to\Pb(V_1)\quad\text{and}\quad p_W:\mathbb{RP}^d\setminus \Pb(W_2)\to\Pb(W_1)~.\]

As $(\gamma_n)$ is an escaping sequence, there is a positive integer $i$ such that $\frac{\sigma_i(\gamma_n)}{\sigma_{i+1}(\gamma_n)}\to+\infty$ as $n\to\infty$. We assume $i$ is the minimal one with this property, and by our assumption, $i\geqslant 2$. Let $U_1:=\Span(e_1,...,e_i)$ and $U_2:=\Span(e_{i+1},\dots,e_{d+1})$.  The sequence of maps $a(\gamma_n):\mathbb{RP}^d\setminus \Pb(U_2)\to\mathbb{RP}^d$ converges uniformly on any compact subsets to a map $g:\mathbb{RP}^d\setminus \Pb(U_2)\to\Pb(U_1)$ up to subsequence.

Since $\Gamma$ is irreducible, there are points $([v_1],[\phi_1]),([v_2],[\phi_2])\in\Lambda_\Gamma$ such that $(p_V(v_1),p_V(v_2))$ is a basis of $V_1$. Let $q:\Fc_{1,d}\to\mathbb{RP}^d$ be the projection to the first coordinate, and let $O_1,O_2\subset\Fc_{1,d}$ be open neighborhoods of $([v_1],[\phi_1])$ and $([v_2],[\phi_2])$ respectively with disjoint closures, such that $q(O_1),q(O_2)\subset\mathbb{RP}^d\setminus \Pb(V_2)$, and $p_V\circ q(O_1)$ and $p_V\circ q(O_2)$ are disjoint open subsets of $\Pb(V_1)$. Since $\Lambda$ is perfect, for each $j=1,2$, take a point \[([v_j'],[\phi_j'])\in (O_i\cap\Lambda)\setminus\{([v_j],[\phi_j])\}~.\]

By definition, $U_2\subset l(V_2)$, so $l q(O_1),l q(O_2)\subset  \mathbb{RP}^d\setminus \Pb(U_2)$. Thus,
\[p_W(\gamma_n\,[v_1])=p_W(k_na(\gamma_n)l_n\,[v_1])\to p_W(kgl\,[v_1])=kgl\,p_V([v_1])\in kgl\,p_V\circ q(O_1)~.\]
Similarly, 
\[p_W(\gamma_n\,[v_1'])\in kgl\,p_V\circ q(O_1)\quad\text{and}\quad p_W(\gamma_n\,[v_2]),p_W(\gamma_n\,[v_2'])\in kgl\,p_V\circ q(O_2)~.\]
This implies that $\lim_{n\to\infty}\gamma_n([v_1],[\phi_1])$ and $\lim_{n\to\infty}\gamma_n([v_1'],[\phi_1'])$ are distinct from $\lim_{n\to\infty}\gamma_n([v_2],[\phi_2])$, while $\lim_{n\to\infty}\gamma_n([v_2],[\phi_2])$ and $\lim_{n\to\infty}\gamma_n([v_2'],[\phi_2'])$ are distinct from $\lim_{n\to\infty}\gamma_n([v_1],[\phi_1])$. This contradicts the assumption that $\Gamma$ acts on $\Lambda$ as a convergence group. Hence, $\Gamma$ is projective divergent.

Since $\Lambda$ is transverse, it now suffices to verify $\Lambda_\Gamma\subset\Lambda$. Pick any point $x\in\Lambda_\Gamma$, and let $(\gamma_n)$ be a sequence in $\Gamma$ such that $(U_1(\gamma_n),U_d(\gamma_n))\to x$ as $n\to\infty$. By taking a subsequence, we may assume that $(U_1(\gamma_n^{-1}),U_d(\gamma_n^{-1}))\to y$ as $n\to\infty$ for some $y\in\Lambda_\Gamma$. The irreducibility of $\Gamma$ then implies that there is some point $z\in\Lambda$ such that $z^1$ is transverse to $y^d$. Then by \Cref{lemma: projective north-south dynamics}, $\gamma_n z^1\to x^1$ as $n\to\infty$, so there is some $x'\in\Lambda$ such that $(x')^1=x^1$. Similarly, there is some $x''\in\Lambda$ such that $(x'')^{d}=x^{d}$. In particular, $x'$ and $x''$ are not transverse, so the fact that $\Lambda$ is transverse implies that $x=x'=x''\in\Lambda$. Since $x\in\Lambda_\Gamma$ was arbitrary, $\Lambda_\Gamma\subset\Lambda$.
\end{proof}

Next, we turn our attention to the proof of \Cref{proposition: gf convergence action}. We will regard $\Pi$ and $C$ as fixed in the remainder of this section, so we simplify notation by denoting
\[U^{\rm thin}:=U^{\rm thin}(\Pi,C)\quad\text{and}\quad U^{\rm thick}:=U^{\rm thick}(\Pi,C)~.\]
Also, for every $x\in\Lambda_U$, we denote
\[H_x:=H_x(C)\quad\text{and}\quad \partial H_x:=\{z\in U:\Bc_x(z)=C\}\subset U^{\mathrm{thick}}~.\] 

We begin with the following compactness result for every $p\in\Pi$. 

\begin{lemma}\label{horoball cocompactness}
    For every $p\in\Pi$, $\Stab_\Gamma(\partial H_p)$ acts cocompactly on $\partial H_p$.
\end{lemma}

\begin{proof}
Since the $\Gamma$-action on $U$ is properly discontinuous and the collection of horoballs $\{\gamma \overline{H_q}: \gamma \in \Gamma, \ q\in \Pi\}$ is locally finite, notice that the disjoint union
\[A:=\bigcup_{\gamma \in \Gamma}\bigcup_{q\in\Pi} \gamma\partial H_q\]
is the boundary of the open subset $U^{\rm thin}\subset U$, and is thus a $\Gamma$-invariant, locally finite closed subset of $U$.  Since $A\subset U^{\rm thick}$ and $\Gamma$ acts cocompactly on $U^{\rm thick}$, $\Gamma$ also acts cocompactly on $A$. So there exists a compact $K\subset A$ such that $\Gamma \cdot K$ covers $A$. 

Again, by local finiteness of the horoballs, there exists finitely many $\gamma_1,\dots ,\gamma_k\in \Gamma$ representing distinct cosets in $\Gamma /\Stab_{\Gamma}(\partial H_p)$ such that $\gamma _i\partial H_p\cap  K$ is non-empty. Note that for any $x\in \partial H_p$, there exists $\gamma\in \Gamma$ and $y\in K$ such that $x = \gamma y$, which means $\gamma^{-1}\partial H_p\cap K$ is nonempty, so there exists $1\leqslant j\leqslant k$ such that $\gamma\gamma_j \in \Stab_{\Gamma}(\partial H_p)$. As a result, $\partial H_p$ can be covered by the $\Stab_{\Gamma}(\partial H_p)$-orbit of $\bigcup_{j=1}^k \gamma_j^{-1}K$, which is compact.
\end{proof}

We are now ready to prove \Cref{proposition: gf convergence action}.

\begin{proof}[Proof of \Cref{proposition: gf convergence action}]
We divide the proof into two steps. First, we prove that every point in $\Lambda_U\setminus\Gamma\Pi$ is conical. Second, we prove that every point in $\Pi$ is a bounded parabolic point. 

{\bf Step 1.} Pick any $z^-\in\Lambda_U\setminus \Gamma\Pi$. We now prove that $z^-$ is a conical point. Let $z^+\in\Lambda_U\setminus \Gamma\Pi$ be distinct from $z^-$. Since $U$ is convex, there exists $z\in \T^1\Hb_\tau^d$ with $z^+$ and $z^-$ as its forward and backward endpoint respectively. Let $\ell$ be the geodesic in $\Hb_\tau^d$ tangent to $z$. By \eqref{glt formula}, there is a biproximal, diagonal $g\in \PGL(d+1,\Rb)$ that preserves the image of $\ell$, has $z^+$ and $z^-$ as its attractor and repeller respectively, and $g z = \varphi^1(z)$. By \Cref{prop: unimodal}(2) and the disjointness of horoballs, there is a sequence $(n_k)$ of real numbers such that $n_k \to +\infty$ as $k\to\infty$ and $\varphi^{-n_k}(z) = g^{-n_k}z\in U^{\mathrm{thick}}$. Since the $\Gamma$-action on $U^{\mathrm{thick}}$ is cocompact, there exists an escaping sequence $(\gamma_k)$ in $\Gamma$ such that $\gamma_k g^{-n_k}z$ remains in a compact set $K \subset U^{\mathrm{thick}}$ such that $\Gamma K = U^{\mathrm{thick}}$. Hence up to subsequence, we may assume that $(\gamma_k g^{-n_k}z)$ converges to some $w\in U$. Then $\gamma_k z^+\to w^+$ and $\gamma_k z^- \to w^-$ as $k\to \infty$, where $w^+,w^-\in\Lambda_U$ are the forward and backward endpoints of the spacelike geodesic tangent to $w$.

Since $\Gamma$ acts on $\Lambda_U$ as a convergence group, by taking a further subsequence, we know that there is some $a,b\in\Lambda_U$ such that $\gamma_k\vert_{\Lambda_U\setminus a}$ converges uniformly on compact sets to $b$. Thus, either $b=w^+$, in which case $a=z^-$, or $b=w^-$, in which case $a=z^+$. It suffices to prove that the former holds.

Suppose for the purpose of contradiction that the latter holds, i.e. $b=w^-$ and $a=z^+$. Pick a point $ z_*^+\in \Lambda_U\setminus \{z^-,z^+\}$. Then $z_*^+$ and $z^-$ are transverse, so \Cref{prop: spacelike geodesics} implies that there is a spacelike geodesic $\ell_*$ with $z_*^+$ and $z^-$ as its forward and backward endpoints in $\partial_\infty^{\Pb}\Hb_\tau^d$ respectively, such that $\ell$ and $\ell_*$ share a backward endpoint $\widetilde z^-\in\partial_\infty^\mathbb{S}\Hb_\tau^d$. Let $\widetilde{z}_*^+\in\partial_\infty^\mathbb{S}\Hb_\tau^d$ denote the forward endpoint of $\ell_*$. By \Cref{observation: joining by spacelike is an open condition}, there are open sets $W_{\widetilde z_*^+}$ and $W_{\widetilde z^-}$ in $\Hb_\tau^d\cup\partial_\infty^{\mathbb S}\Hb_\tau^d$ that contain $\widetilde z_*^+$ and $\widetilde z^-$ respectively, so that any point in $W_{\widetilde z_*^+}$ and any point in $W_{\widetilde z^-}$ is joined by a spacelike geodesic in $\Hb_\tau^d$. 

Since $z_*^+\neq z^+$, our assumption implies that $\gamma_k z_*^+\to w^-$ as $k\to\infty$. Hence, there is a sequence of points $(p_k)$ along $\ell_*$ that converges to $z_*^+$ such that $\gamma_k(p_k)\to w^-$ (we just need to ensure that $(p_k)$ converges to $z_*^+$ fast enough). By taking the tail of the sequences, we may assume that $\pi(g^{-n_k}z)\in W_{\widetilde z^-}$ and $p_k\in W_{\widetilde z_*^+}$ for all $k\in\Nb$, so there is a spacelike geodesic $\ell_{*k}$ in $\Hb_\tau^d$ that passes through $\pi(g^{-n_k}z)$ and $p_k$ in that order. For each $k\in\Nb$, let $z_{*k}$ denote the tangent vector to $\ell_{*k}$ at $\pi(g^{-n_k}z)$, and let $z_{*k}^+$ denote the forward endpoint of $\ell_{*k}$ in $\partial_\infty^{\Pb}\Hb_\tau^d$. 

Notice that the sequence of spacelike geodesics $(\ell_{*k})$ converges to $\ell_*$, so the sequence $(z_{*k}^+)$ converges to $z_*^+$. Thus, there is a compact set $K\subset\Fc_{1,d}$ consisting of flags that are transverse to $z^-$, such that $z_{*k}^+\in K$ for large enough $k$. Since $z^+$ and $z^-$ are the attractor and repeller of $g$ respectively, we have $g^{n_k} z_{*k}^+ \to z^+$ as $k\to\infty$. Also, 
\[\pi(g^{n_k}z_{*k})=g^{n_k}\pi(z_{*k})=g^{n_k}\pi(g^{-n_k}z)=\pi(z)~.\] 
Hence, the sequence of geodesics $(g^{n_k}\ell_{*k})$ converges to the geodesic with forward endpoint $z^+$ and which passes through $\pi(z)$, which is exactly $\ell$. Therefore $g^{n_k}z_{*k}$ converges to $z$, and so \[(z,g^{n_k} z_{*k})_q \to 1\text{ as }k\to \infty~.\] 

On the other hand, since $\pi (\gamma_kg^{-n_k} z)\to \pi(w)$ and $\gamma_k p_k\to w^-$, the sequence of geodesics $(\gamma_k\ell_{*k})$ converges to the geodesic with forward endpoint $w^-$ and which passes through $\pi(w)$. In particular, $\gamma_k z_{*k} \to -w$ as $k\to \infty$. It follows that 
\[(z,g^{n_k} z_{*k})_q = (g^{-n_k}z,z_{*k})_q = (\gamma_kg^{-n_k} z,\gamma_k z_{*k})_q\to (w,-w)_q= -1\text{ as }k\to \infty~,\]
which is a contradiction.

{\bf Step 2.} Pick $p\in \Pi$. We now prove that $p$ is a bounded parabolic point. To do so, we will need to prove that
\begin{enumerate}
    \item [(1)] $\Stab_\Gamma(p)$ acts cocompactly on $\Lambda_U\setminus\{p\}$, 
    \item [(2)] $\Stab_\Gamma(p)$ is infinite, and
    \item [(3)] $\Stab_\Gamma(p)$ does not contain any loxodromic elements.
\end{enumerate}

First, we verify that $\Stab_\Gamma(\partial H_p)\subset\Stab_\Gamma(p)$. Let 
\[f:\partial H_p\to(\mathbb{RP}^d\times\mathbb{RP}^{d*})^2\]
be the $\Stab_\Gamma(\partial H_p)$-equivariant embedding given by $f:z\mapsto(\pi(z),z^+)$, and let 
\[\Delta:=\{(x,y)\in(\mathbb{RP}^d\times\mathbb{RP}^{d*})^2:x=y\}~.\]
Notice that $\{(p,p)\}=\overline{f(\partial H_p)}\cap\Delta$. Indeed, if $(z_n)$ is a sequence in $\partial H_p$ such that $f(z_n)\to (q,q)\in\Delta$ for some $q\in(\mathbb{RP}^d\times\mathbb{RP}^{d*})\setminus\{p\}$, then \eqref{horofunction component bound} implies that $\mathcal{B}^2_p(z_n)\ge C$ for all $n\in\Nb$. Thus, by taking a subsequence we may assume that the sequence $(z_n^-)$ converges to some $q' \in\Fc_{1,d}$ that is transverse to $q$. Then by \Cref{prop: unimodal}, $\mathcal{B}_p(z_n)\to -\infty$ as $n\to\infty$, which is a contradiction. Hence, if $\gamma\in\Stab_\Gamma(\partial H_p)$, then
\[\{(\gamma p,\gamma p)\}=\overline{ f(\gamma\partial H_p)}\cap\gamma\Delta=\overline{f(\partial H_p)}\cap\Delta=\{(p,p)\}~,\]
so $\gamma\in\Stab_\Gamma(p)$.

Since $\Stab_\Gamma(\partial H_p)\subset\Stab_\Gamma(p)$, we see that
\[W^{\rm ss}_p:=\{z \in U: z^+ = p, \ \Bc_{p} (z) = C\}\subset\partial H_p\] 
is closed and $\Stab_\Gamma(\partial H_p)$-invariant. Then \Cref{horoball cocompactness} implies that $\Stab_\Gamma(\partial H_p)$ also acts cocompactly on $W^{\rm ss}_p$. Since the map 
\[W^{\rm ss}_p \to \Lambda_U\setminus\{p\}\quad\text{given by}\quad z\mapsto z^-\] is $\Stab_\Gamma(\partial H_p)$-equivariant, continuous and two-to-one, we have that $\Stab_\Gamma(\partial H_p)$, and hence $\Stab_\Gamma(p)$, acts cocompactly on $\Lambda_U\setminus\{p\}$. This proves (1). 

Since any point in $\Lambda_U\setminus \Gamma\Pi$ is conical and $\Gamma\Pi$ is contained in the limit set, then the whole $\Lambda_U$ is the limit set of the $\Gamma$-action on itself.
If the $\Gamma$-action on $\Lambda_U$ is elementary, $\Lambda_U$ has two points. Since $\Gamma$ is infinite, it follows that the stabilizer of every point in $\Lambda_U$ (and hence in $\Pi$) is infinite. On the other hand, if the $\Gamma$-action on $\Lambda_U$ is non-elementary, $\Lambda_U$ is a perfect set, so $\Lambda_U\setminus \{p\}$ is non-compact. The fact that $\Stab_\Gamma(p)$ acts cocompactly on $\Lambda_U\setminus\{p\}$ then implies that $\Stab_\Gamma(p)$ is infinite. Thus, (2) holds.

Next, suppose for the purpose of contradiction that there is a loxodromic element $\gamma\in \Stab_\Gamma(p)$. Then $\gamma$ has a fixed point  $q\in\Lambda_U\setminus\{p\}$. By replacing $\gamma$ with $\gamma^2$, we may assume that $\gamma$ preserves a flow line in $\T^1\Hb_\tau^d$, call it $\ell$, whose projection to $\Hb_\tau^d$ is a spacelike geodesic that  joins $q$ and $p$. \Cref{constant difference} implies that one of $H_p$ and $\gamma H_p$ contains the other, so by the disjointness assumption of horoballs, we have $H_p = \gamma H_p$. Thus, $\gamma$ fixes the point of intersection $\ell(\Rb)\cap \partial H_p$. However, since $\gamma$ has infinite order, this contradicts fact that the $\Gamma$-action on $U$ is properly discontinuous. We have verified (3).
\end{proof}

\section{Hyperbolicity of the flow spaces}\label{section: axiom A flows}

In this section, we prove \Cref{theoremalpha: nonwandering} and \Cref{theoremalpha: contraction} from the Introduction. We recall some definitions from hyperbolic dynamics that were used in the statement of these two theorems.

Let $\varphi^t$ be a smooth flow on a smooth manifold $M$. A point $x\in M$ is \emph{non-wandering} if for every open neighborhood $U\subset M$ of $x$ and every $T>0$, there is some $t>T$ such that $\varphi^t(U)\cap U$ is non-empty. The \emph{non-wandering set} of $(M,\varphi^t)$, denoted $\NW(M)$,  is the set of non-wandering points in $M$. Notice that $\NW(M)$ is $\varphi^t$-invariant and closed. 

A $\varphi^t$-invariant subset $H\subset M$ is a \emph{hyperbolic set} if $\T M|_H$ admits a continuous Riemannian metric whose norm we denote $\norm{\cdot}$, as well as a continuous, $\varphi^t$-invariant splitting
\[\T M\vert_H = E_s \oplus E_0 \oplus E_u\]
such that the following hold. 
\begin{enumerate}
\item The line bundle $E_0$ is spanned by the tangent direction to the flow.
\item There are constants $C>0$ and $0<\lambda<1$ such that for all $t>0$, and all $v\in E_s$,
\[\norm{D\varphi^t v}\leqslant C\lambda^t\norm{v}~.\]
\item There are constants $C>0$ and $0<\lambda<1$ such that for all $t>0$, and all $v\in E_u$,
\[\norm{D\varphi^{-t} v}\leqslant C\lambda^t\norm{v}~.\]
\end{enumerate}

In \Cref{section: the non-wandering sets}, we give the proof of \Cref{theoremalpha: nonwandering}. The proof of \Cref{theoremalpha: contraction} is more involved: for any relatively projective Anosov group $\Gamma$, we have to construct a continuous and $\varphi^t$-invariant splitting of the vector bundle $\T(\T^1\Hb_\tau^d)|_{\WH(\Lambda_\Gamma)}/\Gamma$, as well as a continuous Riemannian metric on said bundle, for which the splitting we constructed satisfy the required contraction properties in the definition of a hyperbolic set. In \Cref{section: splitting} and \Cref{section: thick-thin decompositions of flow spaces}, we will construct the required splitting and Riemannian metric respectively, before proving that the required contraction properties hold in \Cref{section: Hyperbolic flows of relatively Anosov subgroups}. Finally, in \Cref{section: DMS}, we describe the relationship between our results in the Anosov case and those of Delarue, Monclair, and Sanders \cite{delarue2025locally}.

\subsection{The weak hull is non-wandering}\label{section: the non-wandering sets}

Let $\Gamma\subset\PGL(d+1,\Rb)$ be a projective transverse, non-elementary, and torsion free subgroup. Recall that  $\Lambda_\Gamma\subset \Fc_{1,d}$ denotes the limit set of $\Gamma$, \[\Dc_\Gamma := \{z \in \T^1\Hb_\tau^d : \forall \eta \in \Lambda_{\Gamma}, \text{ at least one of } z^+, z^- \text{ is transverse to } \eta\}~,\] $\varphi^t$ denotes the geodesic flow on $\Dc_\Gamma$, and $\WH(\Lambda_\Gamma)\subset\Dc_\Gamma$ denotes the weak hull of $\Lambda_\Gamma$. By \Cref{proposition: domain of discontinuity}, the $\Gamma$-action on $\Dc_\Gamma$ is properly discontinuous. We again denote by \[\varphi^t: \Dc_\Gamma/\Gamma\to \Dc_\Gamma/\Gamma\] the flow that descends from the geodesic flow on $\Dc_\Gamma$. Notice that the quotient map $\Dc_\Gamma\to\Dc_\Gamma/\Gamma$ is a covering map, $\Dc_\Gamma/\Gamma$ is a smooth manifold, and the flow $\varphi^t$ on $\Dc_\Gamma/\Gamma$ is smooth. 

We now prove \Cref{theoremalpha: nonwandering}, i.e. that the non-wandering set of $(\Dc_\Gamma/\Gamma,\varphi^t)$ is $\WH(\Lambda_\Gamma)/\Gamma$.

\begin{proof}[Proof of \Cref{theoremalpha: nonwandering}]
    Since $\Gamma$ is non-elementary, the minimality of the $\Gamma$-action on $\Lambda_\Gamma$ and \Cref{dense pairs} implies that the set of closed orbits in $\WH(\Lambda_\Gamma)/\Gamma$ is dense in $\WH(\Lambda_\Gamma)/\Gamma$. At the same time, $\NW(\Dc_\Gamma/\Gamma)\subset\Dc_\Gamma/\Gamma$ is a closed subset that contains all closed orbits, so $\WH(\Lambda_\Gamma)/\Gamma\subset\NW(\Dc_\Gamma/\Gamma)$. 
    
    To prove the reverse inclusion, we pick an arbitrary point $\overline z\in \NW(\Dc_\Gamma/\Gamma)$ and prove that it lies in $\WH(\Lambda_\Gamma)$. Choose a lift $z\in\Dc_\Gamma$ of $\overline z$ and choose a sufficiently small neighborhood $U_1\subset\Dc_\Gamma$ of $z$ such that $\overline{U_1}$ is compact and the covering map $\Dc_\Gamma\to\Dc_\Gamma/\Gamma$ is injective on $U_1$. Let $(U_n)$ be a sequence of neighborhoods of $z$, with $U_n\subset U_{n-1}$ for all $n\in\Nb$ and $\bigcap U_n=\{z\}$.
    
    Since $\overline z$ is non-wandering, there exists an increasing sequence $(t_n)$ in $\Rb_{\geqslant 0}$ and a sequence $(\gamma_n)$ in $\Gamma$ such that $t_n \to +\infty$ as $n \to +\infty$ and for all $n\in\Nb$, \[\varphi^{t_n}(U_n)\cap \gamma_n U_n \ne \emptyset~.\] Observe that $(\gamma_n)$ is escaping because $U_1$ has compact closure.

    Since $\Gamma$ is projective transverse, after passing to a subsequence we may assume that $(\gamma_n)$ is a collapsing sequence with attracting point $a$ and repelling point $b$ in $\Lambda_\Gamma$. It suffices to establish that $\{z^+,z^-\}=\{a,b\}$; indeed, this implies that $z^+$ and $z^-$ both lie in $\Lambda_\Gamma$, which implies that $\overline z\in\WH(\Lambda_\Gamma)/\Gamma$.
    
    To do so, observe first that by the definition of $\Dc_\Gamma$, at least one of $z^+$ and $z^-$ is transverse to $b$. We only give the proof in the case when $z^+$ is transverse to $b$; the other case is identical with the roles of $z^+$ and $z^-$ interchanged. In this case, for every sufficiently large $n\in\Nb$, every point in $\overline{U_n^+}$ are transverse to $b$, where \[U_n^+:=\{u^+:u\in U_n\}\subset \partial_\infty^{\Pb}\Hb_\tau^d~.\] Fix $N\in\Nb$. By \Cref{lemma: projective north-south dynamics}, $\gamma_n U_N^+\to a$ as $n\to\infty$. Since \[U_N^+\cap(\gamma_n U_N^+)=\varphi^{t_n}(U_N)^+\cap(\gamma_n U_N^+)\supset\varphi^{t_n}(U_n)^+\cap(\gamma_n U_n^+) \ne \emptyset\] for all $n\geqslant N$, we have $a\in\overline{U_N^+}$. Since $U_N^+\to z^+$ as $N\to\infty$, it follows that $z^+=a$. In particular, $z^-$ transverse to $a$. Therefore we may apply the same argument to the escaping sequence $(\gamma_n^{-1})$ to get $z^-=b$. Hence, $\{z^+,z^-\}=\{a,b\}$ as required.
\end{proof}

\subsection{The splitting of $\T(\T^1\Hb_\tau^d)|_{\WH(\Lambda_\Gamma)}$}\label{section: splitting}

Suppose henceforth that $\Gamma\subset\PGL(d+1,\Rb)$ is relatively projective transverse, non-elementary, and torsion free subgroup. We now construct the splitting of the vector bundle \[\widetilde E:=\T(\T^1\Hb_\tau^d )|_{\WH(\Lambda_\Gamma)}\] that is required in the proof of \Cref{theoremalpha: contraction}.

To do so, it is convenient to use the following algebraic description of $\T^1\Hb_\tau^d$ and the geodesic flow on it. Let $e_1,\dots,e_{d+1}$ be the standard basis on $\Rb^{d+1}$. By \Cref{prop: spacelike geodesics}, the map $\ell:\Rb\to\Hb_\tau^d$ given by 
\[\ell(t)=\left[e_+\left(\frac{e^t}{\sqrt{2}}e_1+\frac{e^{-t}}{\sqrt{2}}e_{d+1}\right)+e_-\left(-\frac{e^t}{\sqrt{2}}e_1-\frac{e^{-t}}{\sqrt{2}}e_{d+1}\right)\right]\]
is a unit speed spacelike geodesic in $\Hb_\tau^d$ with $[e_+e_1-e_-e_1]$ and $[e_+e_{d+1}-e_-e_{d+1}]$ in $\partial_\infty^{\mathbb S}\Hb_\tau^d$ as its forward and backward endpoints. In the standard basis, let 
\[\mathsf L:=\left\{\begin{bmatrix}
    a&0&0\\
    0&M&0\\
    0&0&a \end{bmatrix}\in\PGL(d+1,\Rb):a\in\Rb\setminus\{0\}\text{ and } M\in\GL(d-1,\Rb)\right\}~,\]
and for any $s\in\Rb$, let \[\alpha(s)=\begin{bmatrix}
    e^s&0&0\\
    0&\Id&0\\
    0&0&e^{-s} \end{bmatrix}\in\PGL(d+1,\Rb)~.\] 

\begin{proposition}\label{algebraic flow description}
The map
\[\T^1\Hb_\tau^d\to\PGL(d+1,\Rb)/\mathsf L\quad\text{given by}\quad g\cdot\ell'(0)\mapsto g\cdot\mathsf L\]
is an isomorphism of $\PGL(d+1,\Rb)$-spaces. Via this isomorphism, the geodesic flow on $\T^1\Hb_\tau^d$ is given by $\varphi^s(g\cdot \mathsf L)=g\alpha(s)\cdot\mathsf L$.
\end{proposition}

\begin{proof}The subgroup of $\PGL(d+1,\Rb)$ that fixes both endpoints of $\ell$ in $\partial_\infty^{\mathbb S}\Hb_\tau^d$ is 
\[\left\{\begin{bmatrix}
a&0&0\\
0&M&0\\
0&0&b \end{bmatrix}:ab>0\text{ and }M\in\GL(d-1,\Rb)\right\}~,\]
and we saw by \eqref{glt formula} that for any $t\in\Rb$ and any 
\[g=\begin{bmatrix}
a&0&0\\
0&M&0\\
0&0&b \end{bmatrix}\]
in this subgroup, we have $g(\ell(t))=\ell\left(t+\frac{1}{2}\log\frac{a}{b}\right)$. It now follows that the stabilizer in $\PGL(d+1,\Rb)$ of $\ell'(0)$ is $\mathsf L$, because the said stabilizer is the subgroup that fixes the endpoints of $\ell$ and also fixes $\ell(0)$. By \Cref{vectors}, $\PGL(d+1,\Rb)$ acts transitively on $\T^1\Hb_\tau^d$, so the first claim follows. The second claim follows since the $\PGL(d+1,\Rb)$-action on $\T^1\Hb_\tau^d$ commutes with the geodesic flow, and $\alpha(s)\ell(t)=\ell(s+t)$ for all $s,t\in\Rb$. 
\end{proof}

Let $\mathfrak{sl}(d+1,\Rb)$ denote the Lie algebra of $\PGL(d+1,\Rb)$, and let $\mathfrak h$ and $\mathfrak d$ denote the subalgebras that correspond to the Lie subgroups $\mathsf L$ and $\{\alpha(s):s\in\Rb\}$ of $\PGL(d+1,\Rb)$. Also, let $\mathsf P^+$ and $\mathsf P^-$ denote the stabilizers in $\PGL(d+1,\Rb)$ of the flags $([e_1],[e_{d+1}^*])$ and $([e_{d+1}],[e_1^*])$ respectively, and let $\mathfrak n^+$ and $\mathfrak n^-$ respectively denote the Lie algebras of the unipotent radicals in $\mathsf P^+$ and $\mathsf P^-$. Explicitly, $\mathfrak{sl}(d+1,\Rb)$ is the vector space of traceless, $(d+1)\times (d+1)$, real valued matrices equipped with the bracket $[A,B]:=AB-BA$, 
\begin{align*}
    \mathfrak h&=\left\{\begin{pmatrix} 
-\frac{1}{2}\tr(M)&0&0\\
0&M&0\\
0&0&-\frac{1}{2}\tr(M)
\end{pmatrix}:\begin{array}{c}
M\text{ is a }(d-1)\times(d-1)\\
\text{real valued matrix}
\end{array}\right\},\\
\mathfrak d&=\left\{\begin{pmatrix}
s&0&0\\
0&0&0\\
0&0&-s \end{pmatrix}:s\in\Rb\right\},\\
\mathfrak n^+&=\left\{\begin{pmatrix}
0&v^T&r\\
0&0&w\\
0&0&0 \end{pmatrix}:v,w\in\Rb^{d-1}\text{ and }r\in\Rb\right\}\\
\mathfrak n^-&=\left\{\begin{pmatrix}
0&0&0\\
w&0&0\\
r&v^T&0 \end{pmatrix}:v,w\in\Rb^{d-1}\text{ and }r\in\Rb\right\}.
\end{align*}

Notice that we have the direct sum decomposition 
\[\mathfrak{sl}(d+1,\Rb)=\mathfrak h\oplus\mathfrak n^-\oplus\mathfrak d\oplus\mathfrak n^+~,\] 
and that the adjoint action of $\mathsf L$ preserves this decomposition. As such, have a $\mathsf L$-invariant splitting
\[\T_{\mathsf L}\left(\PGL(d+1,\Rb)/{\mathsf L}\right)\cong\mathfrak{sl}(d+1,\Rb)/\mathfrak h\cong\mathfrak n^-\oplus\mathfrak d\oplus\mathfrak n^+~.\]
Then by \Cref{algebraic flow description}, we may extend this to a $\PGL(d+1,\Rb)$-invariant splitting of the tangent bundle \[\T(\T^1\Hb_\tau^d)\cong\T\left(\PGL(d+1,\Rb)/{\mathsf L}\right)=\widetilde E_u\oplus \widetilde E_0\oplus \widetilde E_s~\] where $\widetilde E_u\vert_{\mathsf L}=\mathfrak n^-$, $\widetilde E_0\vert_{\mathsf L}=\mathfrak d$ and $\widetilde E_s\vert_{\mathsf L}=\mathfrak n^+$. Note that for all $s\in\Rb$, $\Ad(\alpha(s))$ preserves $\mathfrak n^-$, $\mathfrak d$ and $\mathfrak n^+$. Therefore \Cref{algebraic flow description} also implies that this splitting is invariant under $\D\varphi^t$. Note that the involution ${\rm opp}$ preserves this splitting, while negation switches $\widetilde E_u$ and $\widetilde E_s$ and preserves $\widetilde E_0$.

Recall that for $x\in\Fc_{1,d}$, $\Bc_x:\T^1\Hb_\tau^d\to\Rb$ denotes the horofunction defined in \Cref{section: the horofunctions}. For any $z\in\T^1\Hb_\tau^d$, we define \[M_u(z):=\{w \in \T^1\Hb_\tau^d: w^- = z^-\text{ and } \Bc_{z^-} (w) = \Bc_{z^-} (z)\}\] and \[M_s(z):=\{w \in \T^1\Hb_\tau^d: w^+ = z^+\text{ and } \Bc_{z^+} (w) = \Bc_{z^+} (z)\}~.\] Also, define the maps \[\beta^\pm:\T^1\Hb_\tau^d\to\partial_\infty^{\Pb}\Hb_\tau^d\cong\Fc_{1,d}\] by $\beta^\pm(z)=z^\pm$, and notice that they are smooth. The next proposition relates the splitting of $\T(\T^1\Hb_\tau^d)$ to the horofunctions.

\begin{proposition}\label{proposition: differential of endpoint map}
Let $z\in\T^1\Hb_\tau^d$. Then the following hold,
\begin{enumerate}
    \item $\widetilde E_u|_z=\T_zM_u(z)$, and the map \[(\D\beta^+)_z:\T(\T^1\Hb_\tau^d)|_z\to\T_{z^+}\Fc_{1,d}\] when restricted to $\widetilde E_u|_z$ is a linear isomorphism onto $\T_{z^+}\Fc_{1,d}$.
    \item $\widetilde E_s|_z=\T_zM_s(z)$, and the map \[(\D\beta^-)_{z}:\T(\T^1\Hb_\tau^d)|_{z}\to\T_{z^-}\Fc_{1,d}\] when restricted to $\widetilde E_s|_z$ is a linear isomorphism onto $\T_{z^-}\Fc_{1,d}$.
\end{enumerate}
\end{proposition} 

\begin{proof}
First, we prove that $\widetilde E_u|_z=\T_zM_u(z)$ and $\widetilde E_s|_z=\T_zM_s(z)$. By the homogeneity of the $\PGL(d+1,\Rb)$-action on $\T^1\Hb_\tau^d\cong\PGL(d+1,\Rb)/{\mathsf L}$, we need only to verify this at $z=\mathsf L\in \PGL(d+1,\Rb)/{\mathsf L}$. Observe from \Cref{prop: unimodal}~(1) that $M_u(\mathsf{L})\cap M_s(\mathsf{L})=\{\mathsf{L},{\rm opp}(\mathsf{L})\}$, so $\T_{\mathsf{L}}M_u(\mathsf{L})$ and $\T_{\mathsf{L}}M_s(\mathsf{L})$ are transverse subspaces of $\T_{\mathsf{L}}(\T^1\Hb_\tau^d)$ that do not contain $\widetilde E_0|_{\mathsf{L}}$. It thus suffices to show that $\widetilde E_u|_{\mathsf{L}}\subset\T_{\mathsf{L}}M_u(\mathsf{L})$ and $\widetilde E_s|_{\mathsf{L}}\subset\T_{\mathsf{L}}M_s(\mathsf{L})$. We will only verify the former; the latter is similar. Via the identification $\T_{\mathsf L}(\T^1\Hb_\tau^d)\cong\mathfrak n^-\oplus\mathfrak d\oplus\mathfrak n^+$, $\tilde{E}_u|_{\mathsf{L}}$ is identified with $\mathfrak n^-$, so $\tilde{E}_u|_{\mathsf{L}}$ is tangent to $\exp(\mathfrak n^-)\cdot \mathsf L\subset\T^1\Hb_\tau^d$. By \Cref{unipotents preserve horofunctions}, $\exp(\mathfrak n^-)\cdot \mathsf L\subset M_u(\mathsf L)$, so $\widetilde E_u|_{\mathsf{L}}\subset\T_{\mathsf{L}}M_u(\mathsf{L})$.

By \Cref{points along spacelike geodesics v2} and \Cref{prop: unimodal}~(1), $\beta^+|_{M_u(\mathsf L)}$ is a bijection between $M_u(\mathsf L)$ and the set of flags in $\Fc_{1,d}$ that are transverse to $z^-$. Since $\beta^+$ is smooth and $M_u({\mathsf L})$ is $\exp(\mathfrak n^-)$-homogeneous, it follows that $(\D\beta^+)_z$ when restricted to $\widetilde E_u|_z$ is a linear isomorphism onto $\T_{z^+}\Fc_{1,d}$. This finishes the proof of (1). We can similarly finish the proof of (2). 
\end{proof}

By restricting the vector bundle $\T(\T^1\Hb_\tau^d)$ and its splitting to $\WH(\Lambda_\Gamma)\subset\T^1\Hb_\tau^d$, we obtain the required continuous, $\Gamma$-invariant, ${\rm opp}$-invariant, $\varphi^t$-invariant splitting of 
\begin{align}\label{sjgnhslkjnf}
\widetilde E=\widetilde E_u|_{\WH(\Lambda_\Gamma)}\oplus \widetilde E_0|_{\WH(\Lambda_\Gamma)}\oplus \widetilde E_s|_{\WH(\Lambda_\Gamma)}~.
\end{align}
We refer to this splitting as the \emph{canonical splitting} of $\widetilde E$. 

\subsection{The Riemannian metric on $\T(\T^1\Hb_\tau^d)|_{\WH(\Lambda_\Gamma)}$}\label{section: thick-thin decompositions of flow spaces}
Our next goal is to give a careful construction of the Riemannian metric on $\widetilde E$ used in the proof of \Cref{theoremalpha: contraction}. Our construction is inspired by an analogous construction from \cite{zhu2022relatively} (see also \cite[Section~6]{wang2023notions}). First, we prove that if the flow line of a point in $\WH(\Lambda_\Gamma)$ passes through a horoball, then the time that this flow line enters and exits the horoball varies continuously with the point. Using this fact, we then explicit construct a class of $\Gamma$-invariant, continuous Riemannian metrics on $\widetilde E$, called \emph{$\lambda$-extensions}.

\subsubsection*{Continuous entry time and exit time}

Let $\Lambda\subset\Fc_{1,d}$ be a transverse set. For any $x\in\Lambda$ and any $A>0$, let \[\mathsf{touch}_{x,A}:=\{z\in\WH(\Lambda):\Bc_x(\varphi^t(z))\geqslant A\text{ for some }t\in\Rb\}~.\] Since $\Bc_x$ is continuous, \[\interior(\mathsf{touch}_{x,A}):=\{z\in\mathsf{touch}_{x,A}:\Bc_x(\varphi^t(z))> A\text{ for some }t\in\Rb\}\] and \[\partial(\mathsf{touch}_{x,A}):=\mathsf{touch}_{x,A}\setminus\interior(\mathsf{touch}_{x,A})\] are the interior and topological boundary of $\mathsf{touch}_{x,A}$ respectively, both of which are non-empty. Then define the maps \[\mathsf{in}_{x,A}:\mathsf{touch}_{x,A}\to[-\infty,+\infty)\quad\text{by}\quad\mathsf{in}_{x,A}(z):=\inf\{t\in\Rb:\Bc_x(\varphi^t(z)) \geqslant A\}\] and \[\mathsf{out}_{x,A}:\mathsf{touch}_{x,A}\to(-\infty,+\infty]\quad\text{by}\quad\mathsf{out}_{x,A}(z):=\sup\{t\in\Rb:\Bc_x(\varphi^t(z))\geqslant A\}~.\] Using \Cref{prop: unimodal}, we deduce several properties of the maps $\mathsf{in}_{x,A}$ and $\mathsf{out}_{x,A}$.

\begin{proposition}\label{corollary: continuity of in and out}
    Let $x\in\Lambda$ and $A>0$.
    \begin{enumerate}
        \item For every $z\in\interior(\mathsf{touch}_{x,A})$, $\mathsf{in}_{x,A}(z)<\mathsf{out}_{x,A}(z)$, and at least one of $\mathsf{in}_{x,A}(z)$ or $\mathsf{out}_{x,A}(z)$ is a real number.
        \item For every $z\in\partial(\mathsf{touch}_{x,A})$, $\mathsf{in}_{x,A}(z)= \mathsf{out}_{x,A}(z)$.
        \item For every $z\in\mathsf{touch}_{x,A}$, $-z\in\mathsf{touch}_{x,A}$ and $\mathsf{in}_{x,A}(z)=-\mathsf{out}_{x,A}(-z)$.
        \item The maps $\mathsf{in}_{x,A}$ and $\mathsf{out}_{x,A}$ are both continuous.
    \end{enumerate}
\end{proposition}

\begin{proof}
Since $x$ and $A$ are fixed, in this proof, we denote $\mathsf{touch}_{x,A}$, $\mathsf{in}_{x,A}$, and $\mathsf{out}_{x,A}$ simply by $\mathsf{touch}$, $\mathsf{in}$, and $\mathsf{out}$ respectively.

(1). Since $z\in\interior(\mathsf{touch})$, there is some $t_0\in\Rb$ such that $\Bc_x(\varphi^{t_0}(z))>A$. The continuity of both $\Bc_x$ and the flow then imply that $\mathsf{in}(z)<t_0<\mathsf{out}(z)$. The fact that at least one of $\mathsf{in}(z)$ or $\mathsf{out}(z)$ is a real number follows from \Cref{prop: unimodal}.

(2). If $z\in\partial(\mathsf{touch})$, then by \Cref{prop: unimodal}, there is some $t_0\in\Rb$ such that $\Bc_x(\varphi^{t_0}(z))=A$ and $\Bc_x(\varphi^t(z))<A$ for all $t\in\Rb\setminus\{t_0\}$. Necessarily, $t_0 =\mathsf{in}(z) = \mathsf{out}(z)$.

(3). Notice that for any $z\in\WH(\Lambda)$ and any $t\in\Rb$, we have 
\[\Bc_x(\varphi^t(-z))=\Bc_x(-\varphi^{-t}(z))=\Bc_x(\varphi^{-t}(z))~.\]
As such, if $z\in\mathsf{touch}$, then $-z\in\mathsf{touch}$ and
\[-\mathsf{out}(-z)=-\sup\{t\in\Rb:\Bc_x(\varphi^{-t}(z))\geqslant A\}=\inf\{-t\in\Rb:\Bc_x(\varphi^{-t}(z))\geqslant A\}=\mathsf{in}(z)~.\]

(4). Suppose that we can show that $\mathsf{in}$ is lower semicontinuous on $\mathsf{touch}$ and upper semicontinuous on $\interior(\mathsf{touch})$. Then by (3), $\mathsf{out}$ is upper semicontinuous on $\mathsf{touch}$ and lower semicontinuous on $\interior(\mathsf{touch})$. Thus, $\mathsf{in}$ and $\mathsf{out}$ are continuous on $\interior(\mathsf{touch})$. The continuity of $\mathsf{in}$ and $\mathsf{out}$ on $\partial(\mathsf{touch})$ follow from the fact that for all $z_0\in\partial(\mathsf{touch})$, we have 
\[\mathsf{in}(z_0) \leqslant \liminf_{z\to z_0}\mathsf{in}(z) \leqslant \limsup_{z\to z_0}\mathsf{out}(z) \leqslant \mathsf{out}(z_0)=\mathsf{in}(z_0)\] 
where first and third inequality follows from the semicontinuity properties of $\mathsf{in}$ and $\mathsf{out}$, while the equality is part (2).

It remains to prove that $\mathsf{in}$ is lower semicontinuous on ${\rm touch}$ and upper semicontinuous on ${\rm int}(\mathsf{touch})$. If $z_0 \in \interior(\mathsf{touch})$, then (1) implies that there are $a,b\in\Rb$ such that $[a,b]\subset(\mathsf{in}(z_0),\mathsf{out}(z_0))$. Then by \Cref{prop: unimodal} there is some $\epsilon>0$ such that  
\[\Bc_x(\varphi^{[a,b]}(z_0))\subset (A+\epsilon,+\infty)~.\]
Since $\Bc_x$ and the flow are both continuous, there is some open neighborhood $U\subset\interior(\mathsf{touch})$ of $z_0$ such that for all $z\in U$,
\[\Bc_x(\varphi^{[a,b]}(z))\subset (A+\epsilon,+\infty)~.\]
This implies that $\mathsf{in}(z)<a$ for all $z\in U$. Since the interval $[a,b]\subset(\mathsf{in}(z_0),\mathsf{out}(z_0))$ is arbitrary, we have that 
\[\limsup_{z\to z_0}\mathsf{in}(z)\leqslant \mathsf{in}(z_0)~.\]

We now assume $z_0$ is an arbitrary tangent vector in $\mathsf{touch}$ rather than only in its interior. If $\mathsf{in}(z_0)\in\Rb$, then by using intervals of the form $(-\infty,c]\subset(-\infty,\mathsf{in}(z_0))$ in place of intervals of the form $[a,b]\subset(\mathsf{in}(z_0),\mathsf{out}(z_0))$, a similar argument implies that
\[\liminf_{z\to z_0}\mathsf{in}(z)\geqslant \mathsf{in}(z_0)~.\]
If $\mathsf{in}(z_0)=-\infty$, then this is also holds trivially.
\end{proof}

We now specialize to the case when the transverse set $\Lambda$ is the limit set $\Lambda_\Gamma$ (recall here that $\Gamma\subset\PGL(d+1,\Rb)$ is a relatively projective Anosov subgroup). By \Cref{theoremalpha: characterizations1}, there is some finite subset $\Pi\subset\Lambda_\Gamma$ and constant $A\in\Rb$ such that 
\[\WH(\Lambda_\Gamma)=\WH(\Lambda_\Gamma)^{\rm thin}(\Pi,A)\cup\WH(\Lambda_\Gamma)^{\rm thick}(\Pi,A) \]
is a thick-thin decomposition of $\WH(\Lambda_\Gamma)$. For the remainder of this section, we will largely regard $\Pi$ and $A$ are fixed, so we drop them from the notation and denote
\[\WH(\Lambda_\Gamma)^{\rm thin}:=\WH(\Lambda_\Gamma)^{\rm thin}(\Pi,A)\quad\text{and}\quad\WH(\Lambda_\Gamma)^{\rm thick}:=\WH(\Lambda_\Gamma)^{\rm thick}(\Pi,A)~.\]

Notice that for each $p\in\Pi$, $\Bc_p$ is invariant under $\Stab_\Gamma(p)$, so 
$\mathsf{touch}_{p,A}$ is $\Gamma$-invariant, and both $ \mathsf{in}_{p,A}$ and $ \mathsf{out}_{p,A}$ are invariant under pre-composition by elements in $\Stab_\Gamma(p)$. As such, for any $q\in\Gamma\Pi$, we may define 
\[\mathsf{touch}_q:=\gamma\mathsf{touch}_{p,A}~,\]
\[\mathsf{in}_q:= \mathsf{in}_{p,A}\circ\gamma^{-1}:\mathsf{touch}_q\to[-\infty,+\infty)~,\]
and
\[\mathsf{out}_q =: \mathsf{out}_{p,A}\circ\gamma^{-1}:\mathsf{touch}_q\to(-\infty,+\infty]~,\]
for any $p\in\Pi$ and $\gamma\in\Gamma$ that satisfy $q=\gamma p$. 

\subsubsection*{$\lambda$-extensions on $E$} 
Fix a continuous, $\Gamma$-invariant  Riemannian metric $h$ on $\widetilde{E}|_{\WH(\Lambda_\Gamma)^{\rm thick}}$ for which that the splitting \eqref{sjgnhslkjnf} is orthogonal (such $h$ can be shown to exist using a standard partitions of unity argument). For any $\lambda>0$, we now construct the \emph{$\lambda$-extension of $h$}, denoted $h^\lambda$, which is a continuous, $\Gamma$-invariant Riemannian metric on $\widetilde E$ that restricts to $h$ on $\widetilde{E}|_{\WH(\Lambda_\Gamma)^{\rm thick}}$.

To do so, we will use the following lemma, which tells us that the weighted geometric mean of two inner products on a finite dimensional vector space is a well-defined inner product.

\begin{lemma}\cite[Proposition~3.14]{zhu2022relatively}\label{lemma: interpolated metric}
    Let $V$ be a finite-dimensional real vector space, and let $h_0,h_1$ be two inner products on $V$. Then there is a basis $\{v_1,v_2,...,v_d\}$ of $V$ that is orthogonal with respect to both $h_0,h_1$. Furthermore, for all $\eta\in[0,1]$, the inner product $h_0^{1-\eta}h_1^\eta$ with the defining property that 
    \[h_0^{1-\eta}h_1^\eta(v_i,v_j):= h_0(v_i,v_j)^{1-\eta} h_1(v_i,v_j)^\eta\] for any $i,j \in \{1,2,...,d\}$ does not depend on the choice of the basis $\{v_1,v_2,...,v_d\}$. 
\end{lemma}

Now, pick any $z\in \overline{\WH(\Lambda_\Gamma)^{\rm thin}}$. Recall that for all $p\in\Pi$, the horoball centered at $p$ of radius $A$ is \[H_p(A):=\{z\in\WH(\Lambda_\Gamma):\Bc_p(z)> A\}~,\] and that for all $q=\gamma p\in\Gamma\Pi$, we denote \[\widehat{H}_q(A):=\gamma H_p(A)~.\] Since
\begin{align}\label{equation: thin part closure}
    \overline{\WH(\Lambda_\Gamma)^{\rm thin}}=\bigcup_{q\in\Gamma\Pi}\overline{\widehat{H}_q(A)}
\end{align}
is a disjoint union, there is a unique $q\in\Gamma\Pi$ such that $z\in\overline{\widehat{H}_q(A)}$ (in particular, $z\in\mathsf {touch}_q$). To simplify notation, we set 
\[t_1:=\mathsf{in}_q(z)\in[-\infty,0),\quad t_2:=\mathsf{out}_q(z)\in(0,+\infty],\quad\text{and}\quad L:= t_2-t_1\in[0,+\infty]~.\]
If $t_1\in\Rb$, set $z_1:=\varphi^{t_1}(z)$, and if $t_2\in\Rb$, set $z_2:=\varphi^{t_2}(z)$. Also, if $L<+\infty$ (equivalently, both $t_1$ and $t_2$ lie in $\Rb$), set \[t':=t_1+\frac{1}{3}L,\quad t'':=t_2-\frac{1}{3}L,\quad z':=\varphi^{t'}(z),\quad\text{and}\quad z'':=\varphi^{t''}(z)~,\] and notice that $\varphi^{\frac{1}{3}L}(z_1)=z'$, $\varphi^{\frac{1}{3}L}(z')=z''$, and $\varphi^{\frac{1}{3}L}(z'')=z_2$. Finally, if $-t_1,t_2\geqslant\frac{1}{3}L$, set
\[\eta:=-\frac{3t'}{L}~,\]
and notice that $\eta\in[0,1]$ and $1-\eta=\frac{3t''}{L}$. \Cref{corollary: continuity of in and out}~(1) implies that one of the following conditions hold for $z$:
\begin{enumerate}[label=(\alph*)]
    \item \label{item1} $t_1\in\Rb$ and $-t_1\leqslant \frac{1}{3}L$,
    \item \label{item2} $t_2\in\Rb$ and $t_2\leqslant \frac{1}{3}L$,
    \item \label{item3} $t_1,t_2\in\Rb$ and $-t_1,t_2\geqslant\frac{1}{3}L$.
\end{enumerate}

We now define the inner product $h_z^\lambda$ on $E|_z$ in the above three cases separately. For any $v\in E|_z$, we decompose \[v=v_u+v_0+v_s~,\] where $v_u\in\widetilde E_u|_z$, $v_0\in\widetilde E_0|_z$, and $v_s\in \widetilde E_s|_z$. If condition (a) holds, then define
\begin{align}\label{first chunk}
    h_z^\lambda(v,w):=e^{-2\lambda t_1}h_{z_1}(v^1_u,w^1_u)+h_{z_1}(v^1_0,w^1_0)+e^{2\lambda t_1}h_{z_1}(v^1_s,w^1_s)
\end{align}
for all $v,w\in E|_z$, where $v^1:=\D \varphi^{t_1}(v)$ and $w^1:=\D \varphi^{t_1}(w)$ are vectors in $E|_{z_1}$. If condition (b) holds, then define
\begin{align}\label{last chunk}
    h_z^\lambda(v,w):=e^{-2\lambda t_2}h_{z_2}(v^2_u,w^2_u)+h_{z_2}(v^2_0,w^2_0)+e^{2\lambda t_2}h_{z_2}(v^2_s,w^2_s)
\end{align}
for all $v,w\in E|_z$, where $v^2:=\D \varphi^{t_2}(v)$ and $w^2:=\D \varphi^{t_2}(w)$ are vectors in $E|_{z_2}$. If condition (c) holds, then observe that the pushforwards $\D\varphi^{-t'}(h^\lambda_{z'})$ and $\D\varphi^{-t''}(h^\lambda_{z''})$ are inner products on $E|_z$, so we may define
\begin{align}\label{middle chunk}
    h^\lambda_z:=\D\varphi^{-t'}(h^\lambda_{z'})^{1-\eta}\D\varphi^{-t''}(h^\lambda_{z''})^{\eta}~.
\end{align}

Notice that if $z\in\overline{\WH(\Lambda_\Gamma)^{\rm thin}}\cap \WH(\Lambda_\Gamma)^{\rm thick}$, then either $t_1=0$ and $z=z_1$, or $t_2=0$ and $z=z_2$. In both cases, we see from the definition that $h^\lambda_z=h_z$. Also, notice that conditions (a) and (b) are mutually exclusive. Furthermore, if conditions (a) (respectively, (b)) and (c) hold simultaneously, then $t'=0$, $\eta=0$, and $z=z'$ (respectively, $t''=0$, $\eta=1$ and $z=z''$), so the two definitions of $h^\lambda_z$ agree. We have thus verified that $h^\lambda$ is indeed well-defined.

\begin{proposition}
    For any $\lambda>0$ and any $\Gamma$-invariant, continuous Riemannian metric $h$ on $E|_{\WH(\Lambda_\Gamma)^{\rm thick}}$, $h^\lambda$ is a continuous Riemannian metric on $E$.
\end{proposition}

\begin{proof}
    Since $h^\lambda$ is an extension of $h$, it is continuous on $E|_{\WH(\Lambda_\Gamma)^{\rm thick}}$, so it suffices to verify that $h^\lambda$ is continuous on $E|_{\overline{\WH(\Lambda_\Gamma)^{\rm thin}}}$. The fact that \eqref{equation: thin part closure} is a disjoint union that implies that it suffices to pick $q\in\Gamma\Pi$ and verify that $h^\lambda$ is continuous on $E|_{\overline{\widehat{H}_q(A)}}$. 

    By \Cref{corollary: continuity of in and out}, $\mathsf{in}_q(z)$ and $\mathsf{out}_q(z)$, and hence $L$, vary continuously with $z\in \overline{\widehat{H}_q(A)}$. It follows that for each (i) from \ref{item1}, \ref{item2}, \ref{item3}, the subset \[X_{(i)}:=\left\{z\in\overline{\widehat{H}_q(A)}:\text{condition (i) holds}\right\}\subset\overline{\widehat{H}_q(A)}\] is closed, and $h^\lambda$ is continuous on $E|_{X_{(i)}}$. Since $\overline{\widehat{H}_q(A)}=X_{\text{(a)}}\cup X_{\text{(b)}}\cup X_{\text{(c)}}$, we have shown that $h^\lambda$ is continuous on $E|_{\overline{\widehat{H}_q(A)}}$. 
\end{proof}

\begin{remark}
    If we fix $\lambda>0$, then different choices of a thick-thin decomposition of $\WH(\Lambda_\Gamma)$ and different choices of a Riemannian metric $h$ on $\WH(\Lambda_\Gamma)^{\rm thick}$ yield biLipschitz $\lambda$-extensions. Hence, for the required contraction properties to hold for $h^\lambda$, the choice of $\lambda$ is important, but the choices of thick-thin decomposition and of $h$ are irrelevant.
\end{remark}

\subsection{Hyperbolic flows of relatively Anosov subgroups}\label{section: Hyperbolic flows of relatively Anosov subgroups}
Recall that we have fixed a thick-thin decomposition of $\WH(\Lambda_\Gamma)$ and a continuous, $\Gamma$-invariant Riemannian metric $h$ on $\WH(\Lambda_\Gamma)^{\rm thick}$. Our goal now is to prove that there is some $\lambda>0$ so that \Cref{theoremalpha: contraction} holds, where the splitting of $\widetilde E$ is the canonical splitting and the Riemannian metric is the $\lambda$-extension of $h$. 

To do so, we require two different estimates of the contraction rates of $E_s$ and $E_u$ under the flow; one for the part of the flow with endpoints in $\WH(\Lambda_\Gamma)^{\rm thick}$, and one for the part of the flow that lies in $\WH(\Lambda_\Gamma)^{\rm thin}$.

\subsubsection*{Thick part estimate}

Let $\norm{\cdot}$ denote the norm of the Riemannian metric $h$. The following proposition states the required contraction estimate for the part of the flow with endpoints in $\WH(\Lambda_\Gamma)^{\rm thick}$.

\begin{proposition}\label{proposition: thick estimate}
    There are constants $C,c>0$ such that if $z\in \WH(\Lambda_\Gamma)^{\rm thick}$, and $t\geqslant 0$ satisfies  and $\varphi^t(z)\in \WH(\Lambda_\Gamma)^{\rm thick}$, then \[\norm{\D\varphi^t(v)}_{\varphi^t(z)} \leqslant C e^{-ct}\norm{v}_z \quad \text{for all}\quad v\in \widetilde E_s\vert_z\] \[\text{and} \quad \norm{\D\varphi^{-t}(v)}_z \leqslant C e^{-ct}\norm{v}_{\varphi^t(z)} \quad \text{for all}\quad v\in \widetilde E_u\vert_{\varphi^t(z)}~.\]
\end{proposition}

The proof of \Cref{proposition: thick estimate} uses the following lemmas. 

Since the $\Gamma$-action on  $\WH(\Lambda_\Gamma)^{\rm thick}$ is cocompact, we may fix a compact subset $K\subset\WH(\Lambda_\Gamma)^{\rm thick}$ whose $\Gamma$-translates cover $\WH(\Lambda_\Gamma)^{\rm thick}$. The following lemma describes the attracting point and repelling point of collapsing sequences that track flowlines starting from points in $K$.

\begin{lemma}\label{lemma: classical flow contraction}
    Suppose that $(z_n)$ is a sequence in $K$, $(\gamma_n)$ is a collapsing sequence in $\Gamma$, and $(t_n)$ is a sequence in $\Rb_{\geqslant 0}$ such that $w_n:=\gamma_n^{-1}\varphi^{t_n}(z_n)\in K$ for all $n\in\Nb$. Then the sequences $(z_n^+)$ and $(w_n^-)$ in $\Fc_{1,d}$ converge respectively to the attracting point and repelling point of $(\gamma_n)$.
\end{lemma}

\begin{proof}
Let $a$ and $b$ be the repelling point and attracting point of $(\gamma_n)$ respectively. By taking a subsequence, we may assume that there is some $z_\infty, w_\infty\in K$ such that $z_n\to z_\infty$ and $w_n\to w_\infty$ as $n\to\infty$. 

Notice that $a$ is transverse to either $w_\infty^-$ and $w_\infty^+$. Thus, by picking a small $\epsilon>0$, replacing $K$ with $\varphi^{[-\epsilon,0]}(K)$, and replacing each $z_n$ and $t_n$ with $\varphi^{-\frac{\epsilon}{2}}(z_n)$ and $t_n+\frac{\epsilon}{2}$ respectively if necessary, we may assume that $w_\infty$ is transverse to $a$. Then by taking the tail, we may assume that the sequence $(w_n)$ is uniformly transverse to $a$. Hence by \Cref{lemma: projective north-south dynamics}, 
\[b=\lim_{n\to\infty}\gamma_nw_n=\lim_{n\to\infty}\varphi^{t_n}(z_n)=z_\infty^+=\lim_{n\to\infty}z_n^+~.\]
In particular, the sequence $(z_n^-)$ is uniformly transverse to $b$, which is the repelling point of the collapsing sequence $(\gamma_n^{-1})$. Thus
\[a=\lim_{n\to\infty}\gamma_n^{-1}z_n^-=\lim_{n\to\infty}w_n^-~.\qedhere\]
\end{proof}

For convenience, for every $g\in\PGL(d+1,\Rb)$, we denote 
\[\rho(g):= \frac{\sigma_2}{\sigma_1}(g) + \frac{\sigma_{d+1}}{\sigma_d}(g) \quad \text{and} \quad R(g):=\frac{\sigma_1}{\sigma_{d+1}}(g)~.\]
The next lemma bounds, for any $v\in K$, the quantity $\norm{\D\varphi^t(v)}_{\varphi^t(z)}$ in terms of the singular values of the element $\gamma\in\Gamma$ that tracks the flowline through $v$.

\begin{lemma}\label{lemma: stable singular estimate}
    There exists a constant $A_0>0$ such that
    whenever $z\in K$, $v\in \widetilde E_s|_z$, $t\geqslant 0$, and $\gamma \in \Gamma$ satisfies $\varphi^t(z)\in \gamma K$, one has \[\norm{\D\varphi^t(v)}_{\varphi^t(z)}
    \leqslant A_0\rho(\gamma)\norm{v}_z~.\]
\end{lemma}

\begin{proof}
    Let $\norm{\cdot}'$ denote the norm on $\T\Fc_{1,d}$ induced by the angular metric $d_\angle$ on $\mathcal F_{1,d}$. \Cref{proposition: differential of endpoint map}~(2) implies that for every $z\in K$, the derivative $(\D\beta^-)_z$ restricted to $\widetilde E_s\vert_z$ is a linear isomorphism onto $\T_{z^-}\mathcal F_{1,d}$. Since $K$ is compact, there is a constant $A_1>1$ such that, for every $z\in K$ and every $v\in \widetilde E_s\vert_z$,
    \begin{equation}\label{equation: endpoint-comparison} A_1^{-1}\norm{v}_z \leqslant \norm{\D\beta^-(v)}'_{z^-}\leqslant A_1\norm{v}_z~. \end{equation}

    Suppose that the lemma fails. Let $(k_n)$ be a sequence in $\Nb$ that grows to $\infty$. Then for all $n\in\Nb$, there exists $z_n\in K$, a unit (in the norm $\norm{\cdot}$) vector $v_n\in\widetilde E_s|_{z_n}$, $t_n\geqslant 0$, and $\gamma_n\in\Gamma$ such that $\varphi^{t_n}(z_n)\in\gamma_n K$ and $\norm{\D\varphi^{t_n}(v_n)}_{\varphi^{t_n}(z_n)}\geqslant k_n\rho(\gamma_n)$. By taking a subsequence, we may ensure that as $n\to\infty$, we have that $z_n\to z_\infty$ and $v_n\to v_\infty$ for some $z_\infty\in K$ and some unit vector $v_\infty\in\widetilde E_s|_{z_\infty}$. Notice that the sequence $(t_n)$ is bounded above if and only if the sequence $(\gamma_n)$ does not escape in $\Gamma$. When this happens, by taking a subsequence, we may assume that $t_n\to t_\infty\in\Rb$ and $\gamma_n\to\gamma_\infty\in\Gamma$ as $n\to\infty$, in which case $\norm{\D\varphi^{t_\infty}(v_\infty)}_{\varphi^{t_\infty}(z_\infty)}=\infty$, which is impossible. Thus, by further refinement, we may assume that the sequence $(t_n)$ increases to $\infty$ and the sequence $(\gamma_n)$ is collapsing. By \Cref{lemma: classical flow contraction}, $z_\infty^+$ is the attracting point of $(\gamma_n)$, or equivalently, the repelling point of $(\gamma_n^{-1})$. Since $(z_n^-)$ is uniformly transverse to $z_\infty^+$, we may now apply \Cref{lemma: attracting contraction} to deduce that there is a constant $A_2>0$ such that \begin{align}\label{equation: equivariant shrinking}
        \norm{\D\gamma_n^{-1}\circ\D\beta^-(v_n)}'_{\gamma_n^{-1}z_n^-} \leqslant A_2 \rho(\gamma_n)\norm{\D\beta^-(v_n)}'_{z_n^-}~.
    \end{align}

    Combining all these, we see that for all $n\in\Nb$, we have
    \begin{align*}
        k_n&\leqslant \frac{1}{\rho(\gamma_n)}\norm{\D(\gamma_n^{-1}\circ\varphi^{t_n})(v_n)}_{\gamma_n^{-1}\varphi^{t_n}(z_n)}\\
        &\leqslant A_1\frac{1}{\rho(\gamma_n)}\norm{\D\beta^-\circ\D\gamma_n^{-1}(v_n)}_{\gamma_n^{-1}z_n^-}'\\
        &=A_1\frac{1}{\rho(\gamma_n)}\norm{\D\gamma_n^{-1}\circ\D\beta^-(v_n)}_{\gamma_n^{-1}z_n^-}'\\
        &\leqslant A_1A_2\norm{\D\beta^-(v_n)}_{z_n^-}'\\
        &\leqslant A_1^2A_2.
    \end{align*}
    Above, the first inequality follows from the defining property of the sequences $(z_n)$, $(v_n)$, $(t_n)$ and $(\gamma_n)$, together with the $\Gamma$-invariance of $\norm{\cdot}$. The second inequality is the first inequality of \eqref{equation: endpoint-comparison}, together with the fact that $\varphi^t$ preserves endpoints for all $t\in\Rb$. The equality holds because $\beta^\pm$ are equivariant with respect to the $\Gamma$-actions. Finally, the third and fourth inequalities follow from \eqref{equation: equivariant shrinking} and the second inequality of \eqref{equation: endpoint-comparison} respectively. Since $(k_n)$ is a sequence that grows to $\infty$, we have arrived at a contradiction.
\end{proof}

The next lemma tells us that if $\gamma$ tracks a flow line through $K$, then $R(\gamma)$ grows exponentially with the length of the said flowline. 

\begin{lemma}\label{wrejfnwjl}
    There exists a constant $B_0 \geqslant 1$ such that whenever $z\in K$, $t\geqslant 0$ and $\gamma \in \Gamma$ satisfies $\varphi^t(z)\in \gamma K$, we have $B_0^{-1} e^{2t} \leqslant R(\gamma) \leqslant B_0 e^{2t}$.
\end{lemma}

\begin{proof}
    We write $w = \gamma^{-1}\varphi^t(z)\in K$. By \Cref{algebraic flow description}, we may identify $\T^1\Hb_\tau^d\cong \PGL(d+1,\Rb)/\mathsf L$. Since $K$ is compact, we may choose a compact subset $\widehat K\subset \PGL(d+1,\Rb)$ such that $K\subset \widehat K \mathsf L$. From  Bochi--Potrie--Sambarino~\cite[Lemma~A.2]{bochi2019anosov}, there exists $B_1 \geqslant 1$ such that for any $f,k\in \widehat K$, $ g \in \PGL(d+1,\Rb)$ and $\mathcal{S}\in \{\frac{\sigma_1}{\sigma_2}, \frac{\sigma_d}{\sigma_{d+1}} , \frac{\sigma_1}{\sigma_{d+1}}\}$,
    \begin{equation}\label{equation: singular-control}
        B_1^{-1} \mathcal{S}(g) \leqslant \mathcal{S}(f g k^{-1}),\ \mathcal{S}(g k^{-1}),\ \mathcal{S}(f g) \leqslant B_1 \mathcal{S}(g)~.
    \end{equation}

    Let $f,k\in \widehat K$ such that $z = f \mathsf L$ and $w = k \mathsf L$. Then \[\gamma k \mathsf L=\gamma w=\varphi^t(z)= f \alpha(t)\mathsf L ~.\] In particular, there exists $h \in \mathsf L$ such that $\gamma=f\alpha(t) h k^{-1}$. By the definition of $\mathsf L$, we may represent $h$ by a block matrix of the form
    \[h= \begin{pmatrix}
        1 & 0 & 0\\
        0 & h' & 0\\
        0 & 0 & 1 \end{pmatrix}~,\] 
    where $h'\in \mathsf{GL}(d-1,\Rb)$. Then
    \[\alpha(t) h = \begin{pmatrix}
        e^t & 0 & 0\\
        0 & h' & 0\\
        0 & 0 & e^{-t} \end{pmatrix}~,\]
    and so $\frac{\sigma_1}{\sigma_{d+1}}(\alpha(t)h)\geqslant e^{2t}$. By \eqref{equation: singular-control}, we have \[R(\gamma) = \frac{\sigma_1}{\sigma_{d+1}}(f \alpha(t)h k^{-1}) \geqslant B_1^{-1} \frac{\sigma_1}{\sigma_{d+1}}(\alpha(t)h)\geqslant B_1^{-1} e^{2t}~. \] We have thus proven the first inequality for any $B_0\geqslant B_1$. 

    We now prove the other inequality by contradiction. Suppose that it fails. Let $(k_n)$ be a sequence in $\Nb$ that grows to $\infty$. Then for all $n\in\Nb$, there exists $\gamma_n\in\Gamma$, $t_n\geqslant 0$, and $z_n\in K$ such that $\varphi^{t_n}(z_n)\in \gamma_n K$ and $R(\gamma_n)>k_ne^{2t_n}$. By taking a subsequence, we may assume that $(\gamma_n)$ is a collapsing sequence and the sequence $(t_n)$ increases to $\infty$.
    
    For each $n\in\Nb$, set \[w_n:=\gamma_n^{-1}\varphi^{t_n}(z_n)\in K~,\] and choose $f_n,k_n\in\widehat K$ such that $z_n=f_n\mathsf L$ and $w_n=k_n\mathsf L$. By the same argument as we used above, we see that there is some $h_n\in\mathsf L$ such that $\gamma_n = f_n \alpha(t_n) h_n k_n^{-1}$, and \[\alpha(t_n)h_n = \begin{bmatrix}
    e^{t_n}&0&0\\
    0&h_n'&0\\
    0&0&e^{-t_n} \end{bmatrix}~\]
    for some $h_n'\in \mathsf{GL}(d-1,\Rb)$. By \eqref{equation: singular-control}, to obtain a contradiction, it suffices to show that $R(\alpha(t_n)h_n)e^{-2t_n}$ has a uniform upper bound. We will in fact verify that for $n\in\Nb$ large enough, $R(\alpha(t_n)h_n)e^{-2t_n}=1$.

Observe that if we set 
\[F_+^0 = (\Span\{e_1\},\Span\{e_1,e_2,\cdots, e_d\})\] and \[F_-^0 = (\Span\{e_{d+1}\},\Span\{e_2,\cdots, e_d,e_{d+1}\})~,\]
then $z_n^+=f_nF_+^0$, $w_n^-=k_nF_-^0$. Up to subsequence, we assume that there exists $f,k\in\widehat K$ such that $f_n\to f$ and $k_n\to k$ as $n \to \infty$. Set 
\[z_\infty:= f \mathsf L\quad\text{and}\quad w_\infty := k \mathsf L~.\] 
By \Cref{lemma: classical flow contraction}, $z_\infty^+= f F_+^0$ and $w_\infty^-=kF_-^0$ are respectively the attracting point and repelling point of $(\gamma_n)$. Since $\Gamma$ is projective divergent, \Cref{lemma: projective north-south dynamics} implies that
\[(U_1(\gamma_n),U_d(\gamma_n))\to  f F_+^0~, \quad (U_1(\gamma_n^{-1}),U_d(\gamma_n^{-1}))\to k F_-^0~.\]
As $\frac{\sigma_1}{\sigma_2}(\gamma_n), \frac{\sigma_d}{\sigma_{d+1}}(\gamma_n)\to \infty$, \eqref{equation: singular-control} implies that
\[\frac{\sigma_1}{\sigma_2}(\alpha(t_n)h_n), \frac{\sigma_d}{\sigma_{d+1}}(\alpha(t_n)h_n)\to \infty~,\] and so by \Cref{slmsfsvjsgfhb}, we have
\[(U_1(\alpha(t_n)h_n),U_d(\alpha(t_n)h_n))\to F_+^0~, \quad (U_1(h_n^{-1}\alpha(t_n)^{-1} ),U_d(h_n^{-1}\alpha(t_n)^{-1} ))\to F_-^0~.\]
Since $\alpha(t_n)h_n$ is block diagonal, $U_1( \alpha(t_n)h_n)\to [e_1]$ and $\frac{\sigma_1}{\sigma_2}( \alpha(t_n)h_n)\to +\infty$ as $n\to \infty$, we see that $\sigma_1(\alpha(t_n)h_n )=e^{t_n}$. Similarly $\sigma_{d+1}( \alpha(t_n)h_n)=e^{-t_n}$. As such, 
\[ R(\alpha(t_n)h_n)e^{-2t_n} = 1~.\qedhere\]
\end{proof}

Finally, we combine \Cref{lemma: stable singular estimate} and \Cref{wrejfnwjl} to prove \Cref{proposition: thick estimate}.

\begin{proof}[Proof of \Cref{proposition: thick estimate}]

First, we observe that the second inequality follows from the first. Let 
\[\mathsf{ne}:\T^1\Hb_\tau^d\to\T^1\Hb_\tau^d\] denote negation, i.e. $\mathsf{ne}(z)=-z$. We observed earlier that the differential $\D\mathsf{ne}$ on $\widetilde E$ switches the sub-bundles $\widetilde E_s$ and $\widetilde E_u$, while preserving $\widetilde E_0$. Also, since the $\Gamma$-action on $\WH(\Lambda_\Gamma)^{\rm thick}$ is cocompact and $h$ is $\Gamma$-invariant, the Riemannian metrics $\D\mathsf{ne}(h)$ an $h$ are $C_2$-biLipschitz for some $C_2\geqslant 1$. As such, if $v\in \widetilde E_u\vert_{\varphi^t(z)}$, then $\D\mathsf{ne}(v)\in \widetilde E_s\vert_{\varphi^{-t}(-z)}$, and so
\begin{align*}
    \norm{\D\varphi^{-t}(v)}_z&\leqslant C_2\norm{\D\mathsf{ne}\circ\D\varphi^{-t}(v)}_{-z}\\
    &=C_2\norm{\D\varphi^t\circ\D\mathsf{ne}(v)}_{-z}\\
    &\leqslant C_2C e^{-ct}\norm{\D\mathsf{ne}(v)}_{\varphi^{-t}(-z)}\\
    &\leqslant C_2^2C e^{-ct}\norm{v}_{\varphi^t(z)}
\end{align*}
as required (we enlarge the constant $C$ to $C_2^2C$).

Given \Cref{lemma: stable singular estimate} and \Cref{wrejfnwjl}, to prove the first inequality, it remains to verify that there exists a constant $C_0,c_0> 0$, such that $\rho(\gamma) \leqslant C_0 R(\gamma)^{-c_0}$ for any $\gamma \in R$. By the uniform regularity properties of relatively Anosov subgroups introduced by \cite{kapovich2018relativizing} (see \cite[Theorem~1.7, Proposition~1.13, Proposition~4.2]{zhu2022relatively} for explicit statements), we know that there are constants $c_1 \geqslant 0$ and $C_1\geqslant 0$ such that for all $\gamma\in\Gamma$,
\[\log \frac{\sigma_1}{\sigma_2}(\gamma) \geqslant c_1 \log\frac{\sigma_1}{\sigma_d}(\gamma) - C_1~.\]
Since $\frac{\sigma_1}{\sigma_d}(g)=\frac{\sigma_1}{\sigma_d}(g^{-1})$ for all $g\in\PGL(d+1,\Rb)$, the required inequality now holds with $C_0:=2e^{C_1}$ and $c_0:=c_1$.
\end{proof}

\subsubsection*{Thin part estimate}
Next, we prove the required contraction estimate for the part of the flow that lies in $\WH(\Lambda_\Gamma)^{\rm thin}$. Its proof follows the argument of \cite[Lemma~9.6]{zhu2022relatively} and \cite[Lemma~6.10]{wang2023notions}. 

Let $C,c>0$ be the constants in \Cref{proposition: thick estimate}. By enlarging $C$ if necessary, we may assume $C\geqslant 1$. Let $\norm{\cdot}$ denote the norm of $h^c$, the $c$-extension of $h$ (on $\widetilde E|_{\Lambda_\Gamma^{\rm thick}}$, $\norm{\cdot}$ is the norm used in \Cref{proposition: thick estimate}).

\begin{proposition}\label{proposition: thin estimate}
If $z\in\WH(\Lambda_\Gamma)^{\rm thin}$ and $t\geqslant 0$ such that $\varphi^{t'}(z)\in\WH(\Lambda_\Gamma)^{\rm thin}$ for all $0\leqslant t'\leqslant t$, then 
\[\norm{\D\varphi^t(v)}_{\varphi^t(z)} \leqslant C e^{-ct}\norm{v}_z \quad \text{for all}\quad v\in \widetilde E_s\vert_z~,\] 
\[\text{and} \quad \norm{\D\varphi^{-t}(v)}_z \leqslant C e^{-c t}\norm{v}_{\varphi^t(z)} \quad \text{for all}\quad v\in \widetilde E_u\vert_{\varphi^t(z)}~.\]
\end{proposition}

\begin{proof}
We saw in the proof of \Cref{proposition: thick estimate} that $h$ and its pushforward under negation are biLipschitz. Then by construction, the same is true for $h^c$. Therefore, for the same reasons as in the proof of \Cref{proposition: thick estimate}, the second inequality of the proposition reduces to the first by negation.

To prove the first inequality, let $q\in\Gamma\Pi$ be the unique point such that $z\in\widehat{H}_q(C)$, and let 
\[t_1:=\mathsf{in}_q(z),\quad t_2:=\mathsf{out}_q(z),\quad\text{and}\quad L:=t_2-t_1~.\]
Observe that 
\[r_1:=t_1-t=\mathsf{in}_q(\varphi^t(z))\quad\text{and}\quad r_2:=t_2-t=\mathsf{out}_q(\varphi^t(z))~,\] 
and that $r_1\leqslant t_1\leqslant 0$ and $0\leqslant r_2\leqslant t_2$.

We will first argue that it suffices to prove the following:
\begin{enumerate}
\item [(I)] If $t_1,r_1\in\Rb$ and $-t_1,-r_1\leqslant \frac{1}{3}L$, then for all $v\in \widetilde E_s\vert_z$,
\[\norm{\D\varphi^t(v)}_{\varphi^t(z)}=e^{-c t}\norm{v}_{z}~.\]
\item [(II)] If $t_2,r_2\in\Rb$ and $t_2,r_2\leqslant \frac{1}{3}L$, then for all $v\in \widetilde E_s\vert_z$,
\[\norm{\D\varphi^t(v)}_{\varphi^t(z)}=e^{-c t}\norm{v}_{z}~.\]
\item [(III)] If $t_1,r_1,r_2,r_2\in\Rb$ and $-t_1,-r_1,t_2,r_2\geqslant\frac{1}{3}L$, then for all $v\in \widetilde E_s\vert_z$,
\[\norm{\D\varphi^t(v)}_{\varphi^t(z)} \leqslant C e^{-c t}\norm{v}_z~.\]
\end{enumerate}
Notice that when $L=\infty$, then either (I) or (II) holds. On the other hand, if $L<\infty$, then $I:=\varphi^{[t_1,t_2]}(z)$ is a compact segment of a flowline, and (I) (respectively, (II) and (III)) holds if and only if both $z$ and $\varphi^t(z)$ lies along the first (respectively, last and middle) third of $I$. Hence, if $L<\infty$, we can decompose $t=t(1)+t(2)+t(3)$, where $t(1),t(2),t(3)\geqslant 0$, so that if we set $z(1):=z$, $z(2):=\varphi^{t(1)}(z)$, and $z(3):=\varphi^{t(2)}(z)$, then for each $j=1,2,3$, either $t(j)=0$ or both $z(j)$ and $\varphi^{t(j)}(z(j))$ lie along the closed- $j$-th third of $I$. Then by (I), (II), and (II), we have that for all $v\in \widetilde E_s\vert_z$, \[\norm{\D\varphi^t(v)}_{\varphi^t(z)}=\norm{\D\varphi^{t(3)}\circ\D\varphi^{t(2)}\circ\D\varphi^{t(1)}(v)}_{\varphi^t(z)}\leqslant Ce^{-c t}\norm{v}_z\] as required.

Now, we verify (I) and (II). Suppose first that $t_1,r_1\in\Rb$ and $-t_1,-r_1\leqslant \frac{1}{3}L$. By definition, see \eqref{first chunk}, for all $v\in \widetilde E_s\vert_z$ we have
\begin{align}\label{wjiwelknekfw}\begin{split}
\norm{\D\varphi^t(v)}_{\varphi^t(z)}&=e^{cr_1}\norm{\D\varphi^{r_1}\circ\D\varphi^t (v)}_{\varphi^{r_1}\circ\varphi^{t}(z)}\\
&=e^{c(t_1-t)}\norm{\D\varphi^{t_1}( v)}_{\varphi^{t_1}(z)}\\
&=e^{-ct}\norm{v}_{z}.
\end{split}
\end{align}
Similarly, if $t_2,r_2\in\Rb$ and $t_2,r_2\leqslant \frac{1}{3}L$, then using \eqref{last chunk} in place of \eqref{first chunk} in the argument above, we see that for all $v\in \widetilde E_s\vert_z$ we have
\begin{align}\label{equation: precise exponential contraction}
    \norm{\D\varphi^t(v)}_{\varphi^t(z)} = e^{-ct}\norm{v}_{z}.
\end{align} 

Finally, we verify (III). If $t_1,r_1,t_2,r_2\in\Rb$ and $-t_1,-r_1,t_2,r_2\geqslant\frac{1}{3}L$, set 
\[t':=t_1+\frac{1}{3}L,\quad r':=r_1+\frac{1}{3}L,\quad t'':=t_2-\frac{1}{3}L,\quad r'':=r_2-\frac{1}{3}L~,\]
\[\eta:=-\frac{3t'}{L},\quad\text{and}\quad\theta:=-\frac{3r'}{L}~.\]
Also denote
\[z_1:=\varphi^{t_1}(z)=\varphi^{r_1}(\varphi^t(z)),\quad z_2:=\varphi^{t_2}(z)=\varphi^{r_2}(\varphi^t(z))~,\]
\[z':=\varphi^{t'}(z)=\varphi^{r'}(\varphi^t(z)),\quad\text{and}\quad z'':=\varphi^{t''}(z)=\varphi^{r''}(\varphi^t(z))~.\]
By \Cref{lemma: interpolated metric}, there is a basis $(v_1,\dots,v_k)$ of $E_s\vert_z$ that is orthogonal with respect to both inner products $\D\varphi^{-t'}(h^\lambda_{z'})$ and $\D\varphi^{-t''}(h^\lambda_{z''})$. Then for each $1\leqslant i\leqslant k$,
\begin{align}\label{lekjrlksglk}
\begin{split}
\norm{\D\varphi^{t'}(v_i)}_{z'}&=e^{-c L/3}\norm{\D\varphi^{t_1}(v_i)}_{z_1}\\
&\geqslant \frac{1}{C}e^{2cL/3}\norm{\D\varphi^{t_2}(v_i)}_{z_2}\\
&=\frac{1}{C}e^{cL/3}\norm{\D\varphi^{t''}(v_i)}_{z''}
\end{split}
\end{align}
where the first and third equality holds by \eqref{wjiwelknekfw} and \eqref{equation: precise exponential contraction} respectively, and the inequality holds by \Cref{proposition: thick estimate}. Therefore, 
\begin{align*}
\norm{\D\varphi^t(v_i)}_{\varphi^t(z)}&=\norm{\D\varphi^{r'}\circ\D\varphi^t(v_i)}_{z'}^{1-\theta}\norm{\D\varphi^{r''}\circ\D\varphi^t(v_i)}_{z''}^{\theta}\\
&=\norm{\D\varphi^{t'}(v_i)}_{z'}^{1-\eta}\norm{\D\varphi^{t''}(v_i)}_{z''}^{\eta}\norm{\D\varphi^{t'}(v_i)}_{z'}^{\eta-\theta}\norm{\D\varphi^{t''}(v_i)}_{z''}^{\theta-\eta}\\
&=\norm{v_i}_z\norm{\D\varphi^{t'}(v_i)}_{z'}^{\eta-\theta}\norm{\D\varphi^{t''}(v_i)}_{z''}^{\theta-\eta}\\
&\leqslant Ce^{-c t}\norm{v_i}_z
\end{align*}
Here, the first and third equalities hold by definition, see \eqref{middle chunk}. The second equality follows from the definitions of $r'$, $t'$, $r''$, and $t''$. The inequality holds by \eqref{lekjrlksglk}, the assumption that $C\geqslant 1$, and the observation that $\eta-\theta=-\frac{3t}{L}\in[-1,0]$. We may now conclude that if $v=\sum_{i=1}^ka_iv_i\in \widetilde E_s\vert_z$, then
\[\norm{\D\varphi^t (v)}_{\varphi^t(z)}=\sqrt{\sum_{i=1}^ka_i^2\norm{\D\varphi^t(v_i)}_{\varphi^t(z)}^2}\leqslant Ce^{-c t}\sqrt{\sum_{i=1}^ka_i^2\norm{v_i}_z^2}\leqslant Ce^{-ct}\norm{v}_z~.\qedhere\]
\end{proof}

We now have all the ingredients to finish the  proof of \Cref{theoremalpha: contraction}.

\begin{proof}[Proof of \Cref{theoremalpha: contraction}]
If $\varphi^{[0,t]}(z)\subset\WH(\Lambda_\Gamma)^{\rm thin}$, then the theorem holds by \Cref{proposition: thin estimate}. Otherwise, we may find $0\leqslant a\leqslant b\leqslant t$ such that both $\varphi^a(z)$ and $\varphi^b(z)$ lie in $\WH(\Lambda_\Gamma)^{\rm thick}$, $\varphi^{[0,a)}\subset\WH(\Lambda_\Gamma)^{\rm thin}$, and $\varphi^{(b,t]}\subset\WH(\Lambda_\Gamma)^{\rm thin}$. Then for all $v\in \widetilde E_s\vert_z$,
\[\norm{\D\varphi^t(v)}_{\varphi^t(z)}=\norm{\D\varphi^{t-b}\circ\D\varphi^{b-a}\circ\D\varphi^a(v)}_{\varphi^t(z)}
\leqslant C^3e^{-c t}\norm{v}_z~,\]
where the inequality holds by \Cref{proposition: thin estimate} and \Cref{proposition: thick estimate}. Similarly, for all $v\in \widetilde E_u\vert_{\varphi^t(z)}$, we have
\[\norm{\D\varphi^{-t}(v)}_z\leqslant C^3e^{-ct}\norm{v}_{\varphi^t(z)}~.\]
The theorem now holds with $B:=C^3$ and $b=c$.
\end{proof}

\subsection{Relation to the Delarue-Monclair-Sanders flow}\label{section: DMS} Let $\Gamma$ be a non-elementary, torsion-free, projective Anosov subgroup $\Gamma\subset\PGL(d+1,\Rb)$.
Delarue, Monclair, and Sanders \cite{delarue2025locally} constructed a flow space $(\mathcal{M}_{\mathsf{Ad}(\Gamma)}, \phi^t)$ that satisfies Axiom A. In this subsection, we construct a two-to-one flow-equivariant morphism 
\[\pi_{\mathrm{DMS}}:(\mathcal{D}_{\Gamma}/\Gamma,\varphi^t)\to(\mathcal{M}_{\mathsf{Ad}(\Gamma)}, \phi^{t}).\] 
We will also show that in this case, $\WH(\Lambda_{\Gamma})/\Gamma$ is mapped onto the hyperbolic set $\mathcal{K}_{\mathsf{Ad}(\Gamma)}$ of $\mathcal{M}_{\mathsf{Ad}(\Gamma)}$.  Via $\pi_{\rm DMS}$, \Cref{theoremalpha: nonwandering} and \Cref{theoremalpha: contraction} can be viewed as a generalization of \cite[Theorem~A]{delarue2025locally}.

Let $\mathsf{End}(\Rb^{d+1})$ denote the space of endomorphisms of $\Rb^{d+1}$. We have the canonical isomorphisms
\[\mathsf{End}(\Rb^{d+1})\cong\mathfrak{sl}(d+1,\Rb)\cong \mathbb{R}^{d+1}\otimes \mathbb{R}^{d+1*}\quad\text{and}\quad\mathsf{End}(\Rb^{d+1})^*\cong\mathbb{R}^{d+1*}\otimes \mathbb{R}^{d+1}.\]
Let $\mathsf{Ad}: \mathsf{PGL}(d+1,\mathbb{R})\to \mathsf{SL}(\mathsf{End}(\Rb^{d+1}))$ be the adjoint representation. Note that if $\Gamma \subset \mathsf{PGL}(d+1,\mathbb{R})$ is non-elementary, torsion-free and projective Anosov, then so is $\mathsf{Ad}(\Gamma)$.

Following the notation in \cite[Section~3]{delarue2025locally}, we define
\[\mathbb{L}:= \{[v:\alpha]\in \mathbb{P}(\mathsf{End}(\Rb^{d+1})\times \mathsf{End}(\Rb^{d+1})^*): \alpha(v)>0\}.\]
Then $\mathbb{L}$ is naturally a smooth manifold that is equipped with a $\Ad(\Gamma)$-action and a smooth flow $\phi^t$ given by $\phi^t([v:\alpha]):= [e^{2t} v:e^{-2t}\alpha]$ for all $t\in \mathbb{R}$ and $[v:\alpha]\in \mathbb{L}$. Notice that the $\Ad(\Gamma)$-action commutes with the flow $\phi^t$. Also, both
\[\tilde{\mathcal{M}}_{\mathsf{Ad}(\Gamma)}:= \left\{[v:\alpha]\in \mathbb{L} :\begin{array}{l} \text{for all } ([u],[\psi])\in \Lambda_{\Gamma},\text{ either } \\ u\otimes \psi \text{ is transverse to } \alpha\\ \text{or } \psi\otimes u \text{ is transverse to } v \end{array} \right\}\]
and
\[\tilde{\mathcal{K}}_{\mathsf{Ad}(\Gamma)}:= \left\{[v:\alpha]\in \mathbb{L} :\begin{array}{l} ([v],[\alpha])= ([u\otimes \psi],[u'\otimes \psi']) \text{ for some}\\\text{distinct } ([u],[\psi]),([\psi'],[u'])\in \Lambda_{\Gamma}  \end{array}\right\}\] are $\Ad(\Gamma)$-invariant and $\phi^t$-invariant subsets, and $\tilde{\mathcal{K}}_{\mathsf{Ad}(\Gamma)}\subset\tilde{\mathcal{M}}_{\mathsf{Ad}(\Gamma)}$. As such, $\phi^t$ descends to a flow, also denoted $\phi^t$, on $\mathcal{M}_{\mathsf{Ad}(\Gamma)} = \tilde{\mathcal{M}}_{\mathsf{Ad}(\Gamma)}/\mathsf{Ad}(\Gamma)$, which contains $\mathcal{K}_{\mathsf{Ad}(\Gamma)}: = \tilde{\mathcal{K}}_{\mathsf{Ad}(\Gamma)}/\mathsf{Ad}(\Gamma)\subset \mathcal{M}_{\mathsf{Ad}(\Gamma)}$ as a $\phi^t$-invariant subset. 

The definition of $\mathcal{D}_{\Gamma}$ allows us to define a map
\[\tilde{\pi}_{\mathrm{DMS}}: \mathcal{D}_{\Gamma}\to  \tilde{\mathcal{M}}_{\mathsf{Ad}(\Gamma)}\quad\text{by}\quad z\mapsto  \left[\frac{u_1\otimes\psi_1}{\psi_z(u_1) \psi_1(u_z)}: \frac{u_2\otimes \psi_2}{\psi_z(u_2) \psi_2(u_z)}\right]~,\] where $u_1, u_2,u_z\in \mathbb{R}^{d+1}, \psi_1,\psi_2,\psi_z\in \mathbb{R}^{d+1*}$ are vectors such that $z^+ = ([u_1],[\psi_1])$, $z^- = ([u_2],[\psi_2])$, and $\pi(z)= ([u_z],[\psi_z])$. It can be verified that $\tilde{\pi}_{\mathrm{DMS}}$ is a $\Gamma$-equivariant two-to-one map and $\tilde{\pi}_{\mathrm{DMS}}(\WH(\Lambda_{\Gamma})) =  \tilde{\mathcal{K}}_{\mathsf{Ad}(\Gamma)}$. \Cref{prop: spacelike geodesics} shows that $\tilde{\pi}_{\mathrm{DMS}}\circ \varphi^t = \phi^t \circ \tilde{\pi}_{\mathrm{DMS}}$. Thus, $\tilde{\pi}_{\mathrm{DMS}}$ descends to a two-to-one, flow-equivariant map $\pi_{\mathrm{DMS}}: \mathcal{D}_{\Gamma}/\Gamma\to  \mathcal{M}_{\mathsf{Ad}(\Gamma)}$, whose restriction to $\WH(\Lambda_\Gamma)/\Gamma$ is a two-to-one map onto $\mathcal{K}_{\mathsf{Ad}(\Gamma)}$. It now follows that in this case, \Cref{theoremalpha: nonwandering} and \Cref{theoremalpha: contraction} imply that $(\mathcal{M}_{\mathsf{Ad}(\Gamma)}, \phi^t)$ is an Axiom A flow with limit set $\mathcal{K}_{\mathsf{Ad}(\Gamma)}$, thus recovering \cite[Theorem~A]{delarue2025locally}.

\bibliographystyle{alpha}
\bibliography{reference.bib}

\end{document}